\pdfoutput=1
\documentclass[12pt]{amsart}
\usepackage{lmodern}
\usepackage{ifthen}
\usepackage{amsfonts}
\usepackage{amsxtra}
\usepackage{amssymb}
\usepackage{mathdots}
\usepackage{array}
\usepackage[margin=0.8in]{geometry}
\usepackage{xcolor}
\definecolor{cite}{rgb}{0.30,0.60,1.00}
\definecolor{url}{rgb}{0.00,0.00,0.80}
\definecolor{link}{rgb}{0.40,0.10,0.20}
\usepackage[colorlinks,linkcolor=link,urlcolor=url,citecolor=cite,pagebackref,breaklinks]{hyperref}
\usepackage{bbm}
\usepackage{mathtools}
\usepackage{mathrsfs}
\usepackage{appendix}
\usepackage[all]{xy}
\usepackage[lite,abbrev,msc-links,alphabetic]{amsrefs}

\usepackage{graphicx}
\usepackage{multirow}
\usepackage{pstricks}
\usepackage{pst-pdf}
\usepackage{enumitem}
\usepackage{pifont}
\usepackage{marvosym}

\usepackage[OT2,T1]{fontenc}
\DeclareSymbolFont{cyrletters}{OT2}{wncyr}{m}{n}
\DeclareMathSymbol{\Sha}{\mathalpha}{cyrletters}{"58}

\usepackage{graphicx}

\makeatletter
\providecommand*{\Dashv}{%
  \mathrel{%
    \mathpalette\@Dashv\vDash
  }%
}
\newcommand*{\@Dashv}[2]{%
  \reflectbox{$\m@th#1#2$}%
}
\makeatother

\numberwithin{equation}{section}

\theoremstyle{plain}
\newtheorem{proposition}{Proposition}[subsection]
\newtheorem{conjecture}[proposition]{Conjecture}
\newtheorem{corollary}[proposition]{Corollary}
\newtheorem{lem}[proposition]{Lemma}
\newtheorem{theorem}[proposition]{Theorem}

\theoremstyle{definition}
\newtheorem{definition}[proposition]{Definition}
\newtheorem{construction}[proposition]{Construction}
\newtheorem{notation}[proposition]{Notation}
\newtheorem{assumption}[proposition]{Assumption}

\theoremstyle{remark}
\newtheorem{remark}[proposition]{Remark}
\newtheorem{example}[proposition]{Example}

\renewcommand{\b}[1]{\mathbf{#1}}
\renewcommand{\c}[1]{\mathcal{#1}}
\renewcommand{\d}[1]{\mathbb{#1}}
\newcommand{\f}[1]{\mathfrak{#1}}
\renewcommand{\r}[1]{\mathrm{#1}}
\newcommand{\s}[1]{\mathscr{#1}}
\renewcommand{\sf}[1]{\mathsf{#1}}
\renewcommand{\(}{\left(}
\renewcommand{\)}{\right)}
\newcommand{\res}{\mathbin{|}}
\newcommand{\ol}[1]{\overline{#1}{}}
\newcommand{\wt}[1]{\widetilde{#1}{}}
\newcommand{\ul}{\underline}
\renewcommand{\leq}{\leqslant}
\renewcommand{\geq}{\geqslant}

\newcommand{\bG}{\b G}

\newcommand{\cN}{\c N}
\newcommand{\cO}{\c O}

\newcommand{\cU}{\c U}

\newcommand{\dA}{\d A}

\newcommand{\dC}{\d C}

\newcommand{\dF}{\d F}

\newcommand{\dQ}{\d Q}
\newcommand{\dR}{\d R}

\newcommand{\dT}{\d T}

\newcommand{\dZ}{\d Z}

\newcommand{\fP}{\f P}

\newcommand{\fT}{\f T}

\newcommand{\fa}{\f a}

\newcommand{\fc}{\f c}

\newcommand{\fm}{\f m}

\newcommand{\rA}{\r A}

\newcommand{\rC}{\r C}
\newcommand{\rD}{\r D}

\newcommand{\rH}{\r H}
\newcommand{\rI}{\r I}

\newcommand{\rK}{\r K}
\newcommand{\rL}{\r L}
\newcommand{\rM}{\r M}

\newcommand{\rP}{\r P}
\newcommand{\rQ}{\r Q}

\newcommand{\rS}{\r S}
\newcommand{\rT}{\r T}
\newcommand{\rU}{\r U}
\newcommand{\rV}{\r V}
\newcommand{\rW}{\r W}
\newcommand{\rX}{\r X}

\newcommand{\rZ}{\r Z}

\newcommand{\rf}{\r f}

\newcommand{\rh}{\r h}

\newcommand{\rl}{\r l}

\newcommand{\rt}{\r t}
\newcommand{\ru}{\r u}
\newcommand{\rv}{\r v}

\newcommand{\sC}{\s C}
\newcommand{\sD}{\s D}

\newcommand{\sF}{\s F}
\newcommand{\sG}{\s G}

\newcommand{\sI}{\s I}

\newcommand{\sL}{\s L}
\newcommand{\sM}{\s M}

\newcommand{\sO}{\s O}

\newcommand{\sS}{\s S}

\newcommand{\sfR}{\sf R}

\newcommand{\sfT}{\sf T}

\newcommand{\tS}{\mathtt{S}}

\newcommand{\tV}{\mathtt{V}}

\newcommand{\tc}{\mathtt{c}}

\newcommand{\tv}{\mathtt{v}}

\newcommand{\balpha}{\boldsymbol{\alpha}}

\newcommand{\bbA}{\boldsymbol{A}}
\newcommand{\bbB}{\boldsymbol{B}}

\newcommand{\obj}{\text{\Flatsteel}}

\newcommand{\tp}[1]{\prescript{\rt\!}{}{#1}}

\newcommand{\floor}[1]{\lfloor{#1}\rfloor}

\newcommand{\ab}{\r{ab}}

\newcommand{\an}{\r{an}}

\newcommand{\bad}{\r{bad}}

\newcommand{\CF}{\mathbbm{1}}

\newcommand{\et}{{\acute{\r{e}}\r{t}}}

\newcommand{\FL}{\r{FL}}
\newcommand{\flg}{{\rf.\rl.}}

\newcommand{\gr}{\r{gr}}

\newcommand{\HT}{\r{HT}}

\newcommand{\loc}{\r{loc}}
\newcommand{\lr}{\r{lr}}
\newcommand{\MF}{\sM\!\!\sF\!}
\newcommand{\mix}{\r{mix}}
\newcommand{\mnm}{\r{min}}

\newcommand{\Nilp}{\r{Nil}}

\newcommand{\op}{\r{op}}

\newcommand{\ram}{\r{ram}}

\newcommand{\St}{\r{St}}

\newcommand{\univ}{\r{univ}}
\newcommand{\unr}{\r{unr}}

\newcommand{\pr}{\mathrm{pr}}

\newcommand{\Set}{\sf{Set}}

\DeclareMathOperator{\ad}{ad}

\DeclareMathOperator{\BC}{BC}

\DeclareMathOperator{\Def}{Def}
\DeclareMathOperator{\diag}{diag}

\DeclareMathOperator{\End}{End}
\DeclareMathOperator{\Ext}{Ext}
\DeclareMathOperator{\Fil}{Fil}

\DeclareMathOperator{\Gal}{Gal}

\DeclareMathOperator{\GL}{GL}

\DeclareMathOperator{\Hom}{Hom}

\DeclareMathOperator{\Ind}{Ind}

\DeclareMathOperator{\Ker}{ker}

\DeclareMathOperator{\modulo}{mod}

\DeclareMathOperator{\Res}{Res}

\DeclareMathOperator{\Sh}{Sh}

\DeclareMathOperator{\Spec}{Spec}
\DeclareMathOperator{\Spf}{Spf}

\DeclareMathOperator{\Sym}{Sym}

\DeclareMathOperator{\tr}{tr}

\DeclareMathOperator{\Art}{Art}

\begin{document}

\title{Deformation of rigid conjugate self-dual Galois representations}

\author{Yifeng Liu}
\address{Institute for Advanced Study in Mathematics, Zhejiang University, Hangzhou 310058, China}
\email{liuyf0719@zju.edu.cn}

\author{Yichao Tian}
\address{Morningside Center of Mathematics, AMSS, Chinese Academy of Sciences, Beijing 100190, China}
\email{yichaot@math.ac.cn}

\author{Liang Xiao}
\address{Beijing International Center for Mathematical Research, Peking University, Beijing 100871, China}
\email{lxiao@bicmr.pku.edu.cn}

\author{Wei Zhang}
\address{Department of Mathematics, Massachusetts Institute of Technology, Cambridge MA 02139, United States}
\email{weizhang@mit.edu}

\author{Xinwen Zhu}
\address{Division of Physics, Mathematics and Astronomy, California Institute of Technology, Pasadena CA 91125, United States}
\email{xzhu@caltech.edu}

\date{\today}

\begin{abstract}
  In this article, we study deformations of conjugate self-dual Galois representations. The study has two folds. First, we prove an R=T type theorem for a conjugate self-dual Galois representation with coefficients in a finite field, satisfying a certain property called rigid. Second, we study the rigidity property for the family of residue Galois representations attached to a symmetric power of an elliptic curve, as well as to a regular algebraic conjugate self-dual cuspidal representation.
\end{abstract}

\maketitle

\tableofcontents

\section{Introduction}
\label{ss:1}

Let $F/F^+$ be a CM extension of number fields with $\tc\in\Gal(\ol{F}/F^+)$ the complex conjugation. In this article, we study deformations of conjugate self-dual Galois representations of $F$. The study has two folds. First, we prove an R=T type theorem (Theorem \ref{th:deformation}) for a conjugate self-dual Galois representation $\bar{r}$ of $F$ with coefficients in a finite field, satisfying a certain property called \emph{rigid} (Definition \ref{de:rigid}). It is worth mentioning that unlike many other references in the field, we neither assume that the characteristic of the coefficient field is relatively split in $F/F^+$ nor assume that $\bar{r}$ only ramifies at places that are split in $F$. Second, we study the rigidity property for the family of residue Galois representations attached to a symmetric power of an elliptic curve, as well as to a regular algebraic conjugate self-dual cuspidal representation $\Pi$. In particular, we show in Theorem \ref{th:rigid} that if $\Pi$ has a supercuspidal component, then the residue Galois representation is rigid whenever the characteristic is large enough.

The main purpose of this article is to make preparation for our work \cite{LTXZZ} in which we prove major cases toward the Beilinson--Bloch--Kato conjecture on the relation between Selmer groups and $L$-functions. To see how Galois deformation is used in the study of Selmer groups, we refer to that article.

The two main theorems are both technical to state. In order to give some flavor of what we can prove in this article, we now state a result (Theorem \ref{th:combine} below) that follows from the combination of the two main theorems and is relatively straightforward to formulate. Also, it seems to us this kind of result is new in literature.

More precisely, we consider the following data:
\begin{itemize}
  \item an RACSDC (that is, regular algebraic conjugate self-dual cuspidal) representation $\Pi$ of $\GL_N(\dA_F)$ for some $N\geq 2$, with the archimedean weights $\xi=(\xi_{\tau,1}\leq\cdots\leq\xi_{\tau,N})_\tau\in(\dZ^N)^{\Hom(F,\dC)}$ (Definition \ref{de:relevant}),

  \item a number field $E\subseteq\dC$ such that one has a family of Galois representations
      \[
      \{\rho_{\Pi,\lambda}\colon\Gal(\ol{F}/F)\to\GL_N(E_\lambda)\}_\lambda
      \]
      indexed by primes of $E$ satisfying $\rho_{\Pi,\lambda}^\tc\simeq\rho_{\Pi,\lambda}^\vee(1-N)$ as in Proposition \ref{pr:galois} and Definition \ref{de:weak_field} (see Remark \ref{re:galois}),

  \item a finite set $\Sigma^+$ of nonarchimedean places of $F^+$ such that for every nonarchimedean place $w$ of $F$ not above $\Sigma^+$, $\Pi_w$ is unramified and the underlying rational prime of $w$ is unramified in $F$.
\end{itemize}
For every prime $\lambda$ of $E$, we write $\ell_\lambda$ its underlying rational prime and $O_\lambda$ the ring of integers of $E_\lambda$. Let $\Lambda_{\Pi,\Sigma^+}$ be the set of primes $\lambda$ of $E$ such that either $\rho_{\Pi,\lambda}$ is not residually absolutely irreducible, or $\ell_\lambda$ underlies $\Sigma^+$, or $\ell_\lambda-N-1$ is strictly smaller than the maximal distance of integers in $\xi$. For every prime $\lambda$ of $E$ not in $\Lambda_{\Pi,\Sigma^+}$, we can define
\begin{itemize}
  \item a commutative $O_\lambda$-algebra $\sfR^\univ_{\sS_\lambda}$ that classifies conjugate self-dual deformations of the residue representation of $\rho_{\Pi,\lambda}$ that are crystalline with regular Fontaine--Laffaille weights at places above $\ell_\lambda$ and unramified outside $\Sigma^+$ and places above $\ell_\lambda$, and

  \item a maximal ideal $\fm_\lambda$ of $\dT_N^{\ell_\lambda}$ with the residue field $O_\lambda/\lambda$ determined by the Satake parameters of $\Pi$, where $\dT_N^{\ell_\lambda}$ is the abstract spherical unitary Hecke algebra of rank $N$ away from $\Sigma^+$ and places above $\ell_\lambda$.
\end{itemize}
Consider pairs $(\rV,\rK)$ in which
\begin{itemize}
  \item $\rV$ is a (nondegenerate) hermitian space over $F$ (with respect to $\tc$) of rank $N$ that is split at every nonarchimedean place of $F^+$ not in $\Sigma^+$, and

  \item $\rK=\prod_v\rK_v$ is a neat open compact subgroup of $\rU(\rV)(\dA_{F^+}^\infty)$ such that $\rK_v$ is the stabilizer of a self-dual lattice for $v\not\in\Sigma^+$,
\end{itemize}
satisfying that there exists a cuspidal automorphic representation $\pi$ of $\rU(\rV)(\dA_{F^+})$ with nonzero $\rK$-invariants whose automorphic base change is $\Pi$. Every pair $(\rV,\rK)$ as above gives a Shimura variety $\Sh(\rV,\rK)$, which is a quasi-projective smooth scheme over $\dC$ of dimension $d(\rV)$. For $\lambda\not\in\Lambda_{\Pi,\Sigma^+}$, the weights $\xi$ give rise to an $O_\lambda$-linear local system $\sL_{\xi,\lambda}$ on $\Sh(\rV,\rK)$; and we let $\sfT_\lambda$ be the image of $\dT_N^{\ell_\lambda}$ in $\End_\sO\(\rH^{d(\rV)}_\et(\Sh(\rV,\rK),\sL_{\xi,\lambda})\)$.

We refer to \S\ref{ss:global_deformation} for more details of the above constructions.

\begin{theorem}\label{th:combine}
Let the setup be as above. Suppose that there exists a nonarchimedean place of $F$ at which $\Pi$ is supercuspidal. Then $\Lambda_{\Pi,\Sigma^+}$ is a finite set. Moreover, there exists a finite set $\Lambda'_{\Pi,\Sigma^+}$ of primes of $E$ containing $\Lambda_{\Pi,\Sigma^+}$, such that for every $\lambda\not\in\Lambda'_{\Pi,\Sigma^+}$ and every pair $(\rV,\rK)$ as above with $d(\rV)\leq 1$, we have
\begin{enumerate}
  \item There is a canonical isomorphism $\sfR^\univ_{\sS_\lambda}\xrightarrow{\sim}\sfT_{\lambda,\fm_\lambda}$ of local complete intersection commutative $O_\lambda$-algebras.

  \item The $\sfT_{\lambda,\fm_\lambda}$-module $\rH^{d(\rV)}_\et(\Sh(\rV,\rK),\sL_{\xi,\lambda})_{\fm_\lambda}$ is finite and free.
\end{enumerate}
\end{theorem}

The above theorem is a consequence of the two main results of this article, namely, Theorem \ref{th:deformation} and Theorem \ref{th:rigid}. We have a similar consequence when $d(\rV)$ is general, but under an extra assumption on certain vanishing of localized cohomology off middle degree.

The article is organized as follows. In Section \ref{ss:2}, we make some preparations for conjugate self-dual representations, unitary Hecke algebras, and automorphic representations. In Section \ref{ss:3}, we study both local and global deformations of conjugate self-dual Galois representations. In Section \ref{ss:4}, we study the rigidity property for symmetric powers of elliptic curves and automorphic Galois representations.

\subsubsection*{Notations and conventions}

\begin{itemize}
  \item All rings are commutative and unital; and ring homomorphisms preserve units.

  \item Throughout the article, we fix an integer $N\geq 1$. Denote by $\rM_N$ (resp.\ $\GL_N$) the scheme over $\dZ$ of square matrices (resp.\ invertible square matrices) of rank $N$.

  \item We fix a CM extension $F/F^+$ of number fields and an algebraic closure $\ol{F}$ of $F$, with $\tc\in\Gal(\ol{F}/F^+)$ the complex conjugation. Put $\Gamma_F\coloneqq\Gal(\ol{F}/F)$ and $\Gamma_{F^+}\coloneqq\Gal(\ol{F}/F^+)$. In this article, all hermitian spaces over $F$ are with respect to the convolution $\tc$ and are nondegenerate.

  \item Denote by $\Sigma_\infty$ (resp.\ $\Sigma^+_\infty$) the set of complex embeddings of $F$ (resp.\ $F^+$). For $\tau\in\Sigma_\infty$, we denote by $\tau^\tc\coloneqq\tau\circ\tc\in\Sigma_\infty$ its conjugation and $\ul\tau\coloneqq\tau\res_{F^+}\in\Sigma^+_\infty$ the underlying element.

  \item For every rational prime $p$, denote by $\Sigma^+_p$ the set of all $p$-adic places of $F^+$.

  \item Denote by $\Sigma^+_\bad$ the union of $\Sigma^+_p$ for all $p$ that ramifies in $F$.

  \item Denote by $\eta\coloneqq\eta_{F/F^+}\colon\Gamma_{F^+}\to\{\pm1\}$ the character associated to the extension $F/F^+$.

  \item For every prime $\ell$, denote by $\epsilon_\ell\colon\Gamma_{F^+}\to\dZ_\ell^\times$ the $\ell$-adic cyclotomic character.

  \item For every place $v$ of $F^+$, we
     \begin{itemize}
       \item put $F_v\coloneqq F\otimes_{F^+}F^+_v$; and define $\delta(v)$ to be $1$ (resp.\ $2$) if $v$ splits (resp.\ does not split) in $F$;

       \item fix an algebraic closure $\ol{F}^+_v$ of $F^+_v$ containing $\ol{F}$; and put $\Gamma_{F^+_v}\coloneqq\Gal(\ol{F}^+_v/F^+_v)$ as a subgroup of $\Gamma_{F^+}$;

       \item for a homomorphism $r$ from $\Gamma_{F^+}$ to another group, denote by $r_v$ the restriction of $r$ to the subgroup $\Gamma_{F^+_v}$.
     \end{itemize}

  \item For every nonarchimedean place $w$ of $F$, we
     \begin{itemize}
       \item identify the Galois group $\Gamma_{F_w}$ with $\Gamma_{F^+_v}\cap\Gamma_F$ (resp.\ $\tc(\Gamma_{F^+_v}\cap\Gamma_F)\tc$), where $v$ is the underlying place of $F^+$, if the embedding $F\hookrightarrow\ol{F}^+_v$ induces (resp.\ does not induce) the place $w$;

       \item let $\rI_{F_w}\subseteq\Gamma_{F_w}$ be the inertia subgroup;

       \item denote by $\phi_w\in\Gamma_{F_w}$ an \emph{arithmetic} Frobenius element.
     \end{itemize}
\end{itemize}

\subsubsection*{Acknowledgements}

This article is a side product of the AIM SQuaREs project \emph{Geometry of Shimura varieties and arithmetic application to L-functions} conducted by the five authors from 2017 to 2019. We would like to express our sincere gratitude and appreciation to the American Institute of Mathematics for their constant and generous support of the project, and to the staff members at the AIM facility in San Jose, California for their excellent coordination and hospitality.

We thank Toby Gee for providing us with an improvement on Proposition \ref{pr:rigid}(2), hence Theorem \ref{th:rigid}. The research of Y.~L. is partially supported by the NSF grant DMS--1702019 and a Sloan Research Fellowship. The research of L.~X. is partially supported by the NSF grant DMS--1502147 and DMS--1752703, the Chinese NSF grant under agreement No. NSFC--12071004, and a grant from the Chinese Ministry of Education. The research of W.~Z. is partially supported by the NSF grant DMS--1838118 and DMS--1901642. The research of X.~Z. is partially supported by the NSF grant DMS--1902239 and a Simons Fellowship.

\section{Preparation}
\label{ss:2}

\subsection{Extension of essentially conjugate self-dual representations}
\label{ss:conjugate_selfdual}

In this subsection, we collect some notion and facts on the extension of essentially conjugate self-dual representations.

\begin{notation}\label{no:sg}
We recall the group scheme $\sG_N$ from \cite{CHT08}*{\S1}.\footnote{In fact, a better notion seems to be the $C$-group introduced in \cite{BG14}.} Put
\[
\sG_N\coloneqq(\GL_N\times\GL_1)\rtimes\{1,\fc\}
\]
with $\fc^2=1$ and
\[
\fc(g,\mu)\fc=(\mu\tp{g}^{-1},\mu)
\]
for $(g,\mu)\in\GL_N\times\GL_1$. In what follows, we will often regard $\GL_N$ as a subgroup of $\sG_N$ via the embedding $g\mapsto(g,1)1$. Denote by $\nu\colon\sG_N\to\GL_1$ the homomorphism such that $\nu\res_{\GL_N\times\GL_1}$ is the projection to the factor $\GL_1$ and that $\nu(\fc)=-1$. We have the adjoint action $\ad$ of $\sG_N$ on $\rM_N$, given by
\[
\ad(g,\mu)(x)=gxg^{-1},\quad\ad(\fc)(x)=-\tp{x}
\]
for $x\in\rM_N$ and $(g,\mu)\in\GL_N\times\GL_1$.
\end{notation}

Let $\tilde\Gamma$ be a topological group, and $\Gamma\subseteq\tilde\Gamma$ an open subgroup of index at most two.

\begin{notation}\label{no:sg_extension}
Let $R$ be a (topological) ring.
\begin{itemize}
  \item For a (continuous) homomorphism $\rho\colon\Gamma\to\GL_N(R)$, we denote by $\rho^\vee\colon\Gamma\to\GL_N(R)$ the contragredient homomorphism, which is defined by the formula $\rho^\vee(x)=\tp\rho(x)^{-1}$ for every $x\in\Gamma$.

  \item For a (continuous) homomorphism $\rho\colon\Gamma\to\GL_N(R)$ and an element $\gamma\in\tilde\Gamma$ that normalizes $\Gamma$, we let $\rho^\gamma\colon\Gamma\to\GL_N(R)$ be the (continuous) homomorphism defined by $\rho^\gamma(x)=\rho(\gamma x\gamma^{-1})$ for every $x\in\Gamma$.

  \item For a (continuous) homomorphism
      \[
      r\colon\tilde\Gamma\to\sG_N(R)
      \]
      such that the image of $r\res_\Gamma$ lies in $\GL_N(R)\times R^\times$, we denote
      \[
      r^\natural\colon\Gamma\to\GL_N(R)\times R^\times\to\GL_N(R)
      \]
      the composition of $r\res_\Gamma$ with the projection to $\GL_N(R)$.
\end{itemize}
\end{notation}

\begin{lem}\label{le:representation_selfdual}
Suppose that $[\tilde\Gamma:\Gamma]=2$. Let $R$ be a (topological) ring and $\chi\colon\tilde\Gamma\to R^\times$ a (continuous) character. We have
\begin{enumerate}
  \item If $r\colon\tilde\Gamma\to\sG_N(R)$ is a (continuous) homomorphism satisfying $r^{-1}(\GL_N(R)\times R^\times)=\Gamma$ and $\nu\circ r=\chi$, then for every $\gamma\in\tilde\Gamma\setminus\Gamma$, we have
      \[
      r^{\natural,\gamma}=B\circ \chi r^{\natural,\vee}\circ B^{-1},
      \]
      where $B$ is obtained from $r(\gamma)=(B,-\chi(\gamma))\fc$.

  \item Let $\rho\colon\Gamma\to\GL_N(R)$ be a (continuous) homomorphism, $\gamma$ an element in $\tilde\Gamma\setminus\Gamma$, and $B\in\GL_N(R)$ such that $\rho^\gamma=B\circ\chi\rho^\vee\circ B^{-1}$ and $B\tp{B}^{-1}=\mu_B\chi(\gamma)^{-1}\rho(\gamma^2)$ for some $\mu_B\in\{\pm1\}$. Then there exists a unique (continuous) homomorphism
      \[
      r\colon\tilde\Gamma\to\sG_N(R)
      \]
      satisfying $r\res_\Gamma=(\rho,\chi\res_\Gamma)1$ and $r(\gamma)=(B,\mu_B\chi(\gamma))\fc$.

  \item Suppose in (2) that $R$ is a field and $\rho$ is absolutely irreducible. If $\rho^\gamma$ and $\chi\rho^\vee$ are conjugate, then $\rho$ induces a (continuous) homomorphism $r\colon\tilde\Gamma\to\sG_N(R)$ satisfying $r\res_\Gamma=(\rho,\chi)$, unique up to changing the $\GL_N(R)$-component of $r(\gamma)$ by a scalar in $R^\times$.
\end{enumerate}
\end{lem}

\begin{proof}
Part (1) is a special case of \cite{CHT08}*{Lemma~2.1.1}.

For (2), we check that
\[
r(\gamma^2)=(B,\mu_B\chi(\gamma))\fc\cdot(B,\mu_B\chi(\gamma))\fc
=(\mu_B\chi(\gamma)B\tp{B}^{-1},\chi(\gamma^2))1=(\rho(\gamma^2),\chi(\gamma^2))1.
\]
Since $\tilde\Gamma$ is generated by $\Gamma$ and $\gamma$, we obtain a unique (continuous) homomorphism $r\colon\tilde\Gamma\to\sG_N(R)$ as in (2).

For (3), by Schur's lemma, the element $B$ is unique up to scalar in $R^\times$, which implies the existence and also the uniqueness of $\mu_B$. Thus, (3) follows immediately.
\end{proof}

\subsection{Unitary Satake parameters and unitary Hecke algebras}
\label{ss:hermitian_space}

In this subsection, we introduce the notion of unitary Satake parameters and unitary Hecke algebras.

\begin{definition}[Abstract Satake parameter]\label{de:satake_parameter}
Let $L$ be a ring. For a multi-subset $\balpha\coloneqq\{\alpha_1,\dots,\alpha_N\}\subseteq L$, we put
\[
P_{\balpha}(T)\coloneqq\prod_{i=1}^N(T-\alpha_i)\in L[T].
\]
Consider a nonarchimedean place $v$ of $F^+$ not in $\Sigma^+_\bad$.
\begin{enumerate}
  \item Suppose that $v$ is inert in $F$. We define an \emph{(abstract) Satake parameter} in $L$ at $v$ of rank $N$ to be a multi-subset $\balpha\subseteq L$ of cardinality $N$. We say that $\balpha$ is \emph{unitary} if $P_{\balpha}(T)=(-T)^N\cdot P_{\balpha}(T^{-1})$.

  \item Suppose that $v$ splits in $F$. We define an \emph{(abstract) Satake parameter} in $L$ at $v$ of rank $N$ to be a pair $\balpha\coloneqq(\balpha_1;\balpha_2)$ of multi-subsets $\balpha_1,\balpha_2\subseteq L$ of cardinality $N$, indexed by the two places $w_1,w_2$ of $F$ above $v$. We say that $\balpha$ is \emph{unitary} if $P_{\balpha_1}(T)=c\cdot T^N\cdot P_{\balpha_2}(T^{-1})$ for some constant $c\in L^\times$.
\end{enumerate}
\end{definition}

Let $v$ be a nonarchimedean place of $F^+$ not in $\Sigma^+_\bad$. Let $\Lambda_{N,v}$ be the unique up to isomorphism hermitian space over $O_{F_v}=O_F\otimes_{O_{F^+}}O_{F^+_v}$ of rank $N$, and $\rU_{N,v}$ its unitary group over $O_{F^+_v}$. Under a suitable basis, the associated  hermitian form of $\Lambda_{N,v}$ is given by the matrix
\[
\begin{pmatrix}
0 &\cdots & 0  & 1\\
0 & \cdots & 1 & 0\\
\vdots & \iddots & \vdots &\vdots\\
1 & \cdots & 0 &0
\end{pmatrix}.
\]
Consider the \emph{local spherical Hecke algebra}
\[
\dT_{N,v}\coloneqq\dZ[\rU_{N,v}(O_{F^+_v})\backslash\rU_{N,v}(F^+_v)/\rU_{N,v}(O_{F^+_v})].
\]
According to our convention, the unit element of $\dT_{N,v}$ is $\CF_{\rU_{N,v}(O_{F^+_v})}$. Let $\rA_{N,v}$ be the maximal split diagonal subtorus of $\rU_{N,v}$, and $\rX_*(\rA_{N,v})$ be its cocharacter group. Then there is a well-known Satake transform
\begin{align}\label{eq:satake_transform}
\dT_{N,v}\to\dZ[\|v\|^{\pm\delta(v)/2}][\rA_{N,v}(F^+_v)/\rA_{N,v}(O_{F^+_v})]\simeq \dZ[\|v\|^{\pm\delta(v)/2}][\rX_*(\rA_{N,v})]
\end{align}
as a homomorphism of algebras. Choose a uniformizer $\varpi_v$ of $F^+_v$.

\begin{construction}\label{cs:satake_hecke_pre}
Let $L$ be a ring over $\dZ[\|v\|^{\pm\delta(v)/2}]$. Let $\balpha$ be a unitary Satake parameter in $L$ at $v$ of rank $N$. There are two cases.
\begin{enumerate}
  \item Suppose that $v$ is inert in $F$. Then a set of representatives of $\rA_{N,v}(F^+_v)/\rA_{N,v}(O_{F^+_v})$ can be taken as
      \[
      \{(\varpi_v^{t_1},\dots,\varpi_v^{t_N})\res t_1,\dots,t_N\in\dZ\text{ satisfying }t_i+t_{N+1-i}=0\text{ for all }1\leq i\leq N\}.
      \]
      Choose an ordering of $\balpha$ as $(\alpha_1,\dots,\alpha_N)$ satisfying $\alpha_i\alpha_{N+1-i}=1$; we have a unique homomorphism
      \[
      \dZ[\|v\|^{\pm\delta(v)/2}][\rA_{N,v}(F^+_v)/\rA_{N,v}(O_{F^+_v})]\to L
      \]
      of rings over $\dZ[\|v\|^{\pm\delta(v)/2}]$ sending the class of $(\varpi_v^{t_1},\dots,\varpi_v^{t_N})$ to
      $\prod_{i=1}^{\floor{\tfrac{N}{2}}}\alpha_i^{t_i}$. Composing with the Satake transform \eqref{eq:satake_transform}, we obtain a ring homomorphism
      \[
      \phi_{\balpha}\colon\dT_{N,v}\to L.
      \]
      It is independent of the choices of the uniformizer $\varpi_v$ and the ordering of $\balpha$.

  \item Suppose that $v$ splits in $F$ into two places $w_1$ and $w_2$. Then a set of representatives of $\rA_{N,v}(F^+_v)/\rA_{N,v}(O_{F^+_v})$ can be taken as
      \[
      \left\{
      \left.
      \(
      \diag(\varpi_v^{t_1},\dots,\varpi_v^{t_N}),
      \diag(\varpi_v^{-t_N},\dots,\varpi_v^{-t_1})
      \)
      \right| t_1,\dots,t_N\in\dZ
      \right\},
      \]
      where the first diagonal matrix (resp. the second diagonal matrix) is regarded as an element in $\rA_{N,v}(F_{w_1})$ (resp.\ $\rA_{N,v}(F_{w_2})$). Choose orders in $\balpha_1$ and $\balpha_2$ as $(\alpha_{1,1},\dots,\alpha_{1,N})$ and $(\alpha_{2,1},\dots,\alpha_{2,N})$ satisfying $\alpha_{1,i}\alpha_{2,N+1-i}=1$; we have a unique homomorphism
      \[
      \dZ[\|v\|^{\pm\delta(v)/2}][\rA_{N,v}(F^+_v)/\rA_{N,v}(O_{F^+_v})]\to L
      \]
      of $\dZ[\|v\|^{\pm\delta(v)/2}]$-rings sending the class of $\(\diag(\varpi_v^{t_1},\dots,\varpi_v^{t_N}),\diag(\varpi_v^{-t_N},\dots,\varpi_v^{-t_1})\)$ to $\prod_{i=1}^{N}\alpha_{1,i}^{t_i}$. Composing with the Satake transform \eqref{eq:satake_transform}, we obtain a ring homomorphism
      \[
      \phi_{\balpha}\colon\dT_{N,v}\to L.
      \]
      It is independent of the choices of the uniformizer $\varpi_v$, the order of the two places of $F$ above $v$, and the orders in $\balpha_1$ and $\balpha_2$.
\end{enumerate}
\end{construction}

\begin{definition}[Abstract unitary Hecke algebra]\label{de:abstract_hecke}
For a finite set $\Sigma^+$ of nonarchimedean places of $F^+$ containing $\Sigma^+_\bad$, we define the \emph{abstract unitary Hecke algebra away from $\Sigma^+$} to be the restricted tensor product
\[
\dT_N^{\Sigma^+}\coloneqq{\bigotimes_v}'\dT_{N,v}
\]
over all $v\not\in\Sigma^+_\infty\cup\Sigma^+$ with respect to unit elements. It is a ring.
\end{definition}

\subsection{Automorphic representations}

In this subsection, we collect some facts concerning automorphic representations.

\begin{notation}\label{no:weight}
We denote by $\dZ^N_\leq$ and the subset of $\dZ^N$ consisting of nondecreasing sequences. For a finite set $T$ and an element $\xi=(\xi_\tau)_{\tau\in T}\in(\dZ^N_\leq)^T$, put
\begin{align*}
a_\xi\coloneqq \min_{\tau\in T}\{\xi_{\tau,1}\},\qquad
b_\xi\coloneqq \max_{\tau\in T}\{\xi_{\tau,N}\}+N-1.
\end{align*}
\end{notation}

Let $w$ be a nonarchimedean place of $F$. For every irreducible admissible (complex) representation $\Pi$ of $\GL_N(F_w)$, every rational prime $\ell$, and every isomorphism $\iota_\ell\colon\dC\xrightarrow{\sim}\ol\dQ_\ell$, we denote by $\r{WD}(\iota_\ell\Pi)$ the (Frobenius semisimple) Weil--Deligne representation associated to $\iota_\ell\Pi$ via the local Langlands correspondence \cite{HT01}.

\if false

\begin{lem}\label{le:inertia}
Let $w$ be a nonarchimedean place of $F$.
\begin{enumerate}
  \item For every irreducible admissible representation $\Pi$ of $\GL_N(F_w)$, there exists an open subgroup $\rI_\Pi$ of $\rI_{F_w}$ such that for every rational prime $\ell$ and every isomorphism $\iota_\ell\colon\dC\xrightarrow{\sim}\ol\dQ_\ell$, the Weil--Deligne representation $\r{WD}(\iota_\ell\Pi)$ is trivial on $\rI_\Pi$.

  \item For every open subgroup $\rI$ of $\rI_{F_w}$, there exists an open compact subgroup $U_\rI$ of $\GL_N(F_w)$ such that for every irreducible admissible representation $\Pi$ of $\GL_N(F_w)$ satisfying that $\r{WD}(\iota_\ell\Pi)$ is trivial on $\rI$ for some rational prime $\ell$ and some isomorphism $\iota_\ell\colon\dC\xrightarrow{\sim}\ol\dQ_\ell$, we have $\Pi^{U_\rI}\neq\{0\}$.
\end{enumerate}
\end{lem}

\begin{proof}
We will use the following well-known fact (see, for example, \cite{Yao}*{Lemma~3.2} for a proof): for two irreducible admissible representations $\Pi_1$ and $\Pi_2$ of $\GL_N(F_w)$, they are in the same Bernstein component if and only if $\r{WD}(\iota_\ell\Pi_1)\res_{\rI_{F_w}}\simeq\r{WD}(\iota_\ell\Pi_2)\res_{\rI_{F_w}}$.

For (1), let $(M,\sigma)$ be a cuspidal support of $\Pi$. Then the supercuspidal representation $\sigma$ determines a representation $\rho\colon\rI_{F_w}\to\GL_N(\dC)$. By the above fact, for rational prime $\ell$ and every isomorphism $\iota_\ell\colon\dC\xrightarrow{\sim}\ol\dQ_\ell$, $\r{WD}(\iota_\ell\Pi)\res_{\rI_{F_w}}$ is isomorphic to $\iota_\ell\rho$. Thus, (1) follows by taking $\rI_\Pi$ to be the kernel of $\rho$.

For (2), note that there are only finitely many homomorphisms $\rho\colon\rI_{F_w}\to\GL_N(\dC)$ that are trivial on $\rI$, up to conjugation. Thus, there is a finite set $C_\rI$ of Bernstein components, independent of $\ell$ and $\iota_\ell$, such that if an irreducible admissible representation $\Pi$ of $\GL_N(F_w)$ satisfies that the Weil--Deligne representation associated to $\iota_\ell\Pi$ via the local Langlands correspondence is trivial on $\rI$ from some $\ell$ and some $\iota_\ell$, then $\Pi$ must lie on a Bernstein component in $C_\rI$. Since irreducible admissible representations lying on a given Bernstein component have a common level, (2) follows.
\end{proof}

\fi

\begin{definition}\label{de:relevant}
We say that a (complex) representation $\Pi$ of $\GL_N(\dA_F)$ is \emph{RACSDC} (that is, regular algebraic conjugate self-dual cuspidal) if
\begin{enumerate}
  \item $\Pi$ is an irreducible cuspidal automorphic representation;

  \item $\Pi\circ\tc\simeq\Pi^\vee$;

  \item for every archimedean place $w$ of $F$, $\Pi_w$ is regular algebraic (in the sense of \cite{Clo90}*{Definition~3.12}).
\end{enumerate}
If $\Pi$ is RACSDC, then there exists a unique element $\xi_\Pi=(\xi_{\tau,1},\dots,\xi_{\tau,N})_\tau\in(\dZ^N_\leq)^{\Sigma_\infty}$, which we call the \emph{archimedean weights of $\Pi$}, satisfying
$\xi_{\tau,i}=-\xi_{\tau^\tc,N+1-i}$ for every $\tau$ and $i$, such that $\Pi_\tau$ (as a representation of $\GL_N(\dC)$) is isomorphic to the (irreducible) principal series representation induced by the characters
\[
(\arg^{1-N+2\xi_{\tau,1}},\arg^{3-N+2\xi_{\tau,2}},\dots,\arg^{N-3+2\xi_{\tau,N-1}},\arg^{N-1+2\xi_{\tau,N}}),
\]
where $\arg\colon\dC^\times\to\dC^\times$ is the \emph{argument character} defined by the formula $\arg(z)\coloneqq z/\sqrt{z\ol{z}}$.
\end{definition}

\begin{proposition}\label{pr:galois}
Let $\Pi$ be an RACSDC representation of $\GL_N(\dA_F)$ with the archimedean weights $\xi=\xi_\Pi$.
\begin{enumerate}
  \item For every place $w$ of $F$, $\Pi_w$ is tempered.

  \item For every rational prime $\ell$ and every isomorphism $\iota_\ell\colon\dC\xrightarrow{\sim}\ol\dQ_\ell$, there is a semisimple continuous homomorphism
      \[
      \rho_{\Pi,\iota_\ell}\colon\Gamma_F\to\GL_N(\ol\dQ_\ell),
      \]
      unique up to conjugation, satisfying that
      \begin{enumerate}
        \item for every nonarchimedean place $w$ of $F$, the Frobenius semisimplification of the associated Weil--Deligne representation of $\rho_{\Pi,\iota_\ell}\res_{\Gamma_{F_w}}$ is isomorphic to $\r{WD}(\iota_\ell\Pi_w|\det|_w^{\frac{1-N}{2}})$;

        \item for every place $w$ of $F$ above $\ell$, the representation $\rho_{\Pi,\iota_\ell}\res_{\Gamma_{F_w}}$ is de Rham (crystalline if $\Pi_w$ is unramified) with regular Hodge--Tate weights contained in the range $[a_\xi,b_\xi]$;

        \item $\rho_{\Pi,\iota_\ell}^\tc$ and $\rho_{\Pi,\iota_\ell}^\vee(1-N)$ are conjugate.
      \end{enumerate}
\end{enumerate}
\end{proposition}

\begin{proof}
Part (1) is \cite{Car12}*{Theorem~1.2}. For (2), the Galois representation $\rho_{\Pi,\iota_\ell}$ is constructed in \cite{CH13}*{Theorem~3.2.3}; the local-global compatibility (2a) is obtained in \cite{Car12}*{Theorem~1.1} and \cite{Car14}*{Theorem~1.1}; (2b) is obtained in \cite{CH13}*{Theorem~3.2.3}; and (2c) follows from (2a) and the Chebotarev density theorem.
\end{proof}

\begin{definition}\label{de:weak_field}
Let $\Pi$ be an RACSDC representation of $\GL_N(\dA_F)$. We say that a subfield $E\subseteq\dC$ is a \emph{strong coefficient field} of $\Pi$ if $E$ is a number field; and for every prime $\lambda$ of $E$, there exists a continuous homomorphism
\[
\rho_{\Pi,\lambda}\colon\Gamma_F\to\GL_N(E_\lambda),
\]
necessarily unique up to conjugation, such that for every isomorphism $\iota_\ell\colon\dC\xrightarrow{\sim}\ol\dQ_\ell$ inducing the prime $\lambda$, $\rho_{\Pi,\lambda}\otimes_{E_\lambda}\ol\dQ_\ell$ and $\rho_{\Pi,\iota_\ell}$ are conjugate, where $\rho_{\Pi,\iota_\ell}$ is the homomorphism from Proposition \ref{pr:galois}(2).
\end{definition}

\begin{remark}\label{re:galois}
By \cite{CH13}*{Proposition~3.2.5}, a strong coefficient field of $\Pi$ exists when $\Pi$ is RACSDC.
\end{remark}

Let $\rV$ be a hermitian space over $F$ of rank $N$, and $\pi$ an irreducible admissible representation of $\rU(\rV)(\dA_{F^+})$. An \emph{automorphic base change} of $\pi$ is defined to be an automorphic representation $\BC(\pi)$ of $\GL_N(\dA_F)$ that is a finite isobaric sum of discrete automorphic representations such that $\BC(\pi)_v\simeq\BC(\pi_v)$ holds for all but finitely many nonarchimedean places $v$ of $F^+$ such that $\pi_v$ is unramified. By the strong multiplicity one property for $\GL_N$ \cite{PS79}, if $\BC(\pi)$ exists, then it is unique up to isomorphism. Moreover, for every nonarchimedean place $v$ of $F^+$ that is nonsplit in $F$, we have a notion of local base change, which is defined by \cite{Rog90} when $N\leq 3$ and by \cites{Mok15,KMSW} for general $N$.

\begin{proposition}\label{pr:arthur}
Take an RACSDC representation $\Pi$ of $\GL_N(\dA_F)$ with $\xi_\Pi=(\xi_\tau)_\tau$ the archimedean weights. Let $\rV$ be a hermitian space over $F$ of rank $N$ and $\pi=\otimes_v\pi_v$ a cuspidal automorphic representation of $\rU(\rV)(\dA_{F^+})$ such that $\Pi\simeq\BC(\pi)$. Then
\begin{enumerate}
  \item For every nonarchimedean place $v$ of $F^+$, $\BC(\pi_v)\simeq\Pi_v$.

  \item For every $\tau\in\Sigma_\infty$, $\pi_{\ul\tau}$ is a discrete series representation of Harish-Chandra parameter
      \[
      \{\tfrac{1-N}{2}+\xi_{\tau,1},\tfrac{3-N}{2}+\xi_{\tau,2},\dots,\tfrac{N-3}{2}+\xi_{\tau,N-1},\tfrac{N-1}{2}+\xi_{\tau,N}\}
      \]
      after we identity $\rU(\rV)(F_{\ul\tau})$ as a subgroup of $\GL_N(\dC)$ via $\tau\colon F\otimes_{F^+,\ul\tau}\dR\xrightarrow\sim\dC$.
\end{enumerate}
\end{proposition}

\begin{proof}
This follows from \cite{KMSW}*{Theorem~1.7.1} for generic packets.
\end{proof}

\begin{corollary}\label{co:ram}
Take an RACSDC representation $\Pi$ of $\GL_N(\dA_F)$. Let $\rV$ be a hermitian space over $F$ of rank $N$ that is even, and $\pi=\otimes_v\pi_v$ a cuspidal automorphic representation of $\rU(\rV)(\dA_{F^+})$ such that $\Pi\simeq\BC(\pi)$. If $v$ is a nonarchimedean place of $F^+$ that is inert in $F$ (with $w$ the unique place of $F$ above it) such that $\rV_v$ is not split and that $\pi_v$ has nonzero invariants under a special maximal open compact subgroup of $\rU(\rV)(F^+_v)$, then the monodromy operator of $\r{WD}(\iota_\ell\Pi_w)$ is conjugate to $(\begin{smallmatrix}1&1\\ 0 &1\end{smallmatrix})\oplus 1_{N-2}$ for every rational prime $\ell$ and every isomorphism $\iota_\ell\colon\dC\xrightarrow{\sim}\ol\dQ_\ell$.
\end{corollary}

\begin{proof}
Write $N=2r$ for a positive integer $r$. By Proposition \ref{pr:arthur}, we know that $\Pi_w$ is isomorphic to $\BC(\pi_v)$. Since $\Pi_w$ is tempered by Proposition \ref{pr:galois}(1), $\pi_v$ is also tempered. Since $\pi_v$ has nonzero invariants under a special maximal open compact subgroup of $\rU(\rV)(F^+_v)$, the cuspidal support of $\pi_v$ is of the form $((F^\times)^{r-1}\times\rU(\rV_2)(F^+_v),\chi\boxtimes\b{1})$, where $\rV_2$ is the anisotropic hermitian space of rank $2$ over $F_w$ and $\chi$ is a unitary unramified character of $(F^\times)^{r-1}$. In particular, the cuspidal support of $\Pi_w$ is of the form $((F^\times)^{r-1}\times\GL_2(F_w)\times(F^\times)^{r-1},\chi\boxtimes\St_2\boxtimes\chi^{-1})$, where $\St_2$ denotes the Steinberg representation of $\GL_2(F_w)$. The corollary follows immediately.
\end{proof}

\section{Deformation}
\label{ss:3}

In this section, we fix an \emph{odd} rational prime $\ell$ and a subfield $E\subseteq\ol\dQ_\ell$ finite over $\dQ_\ell$. We denote by $\sO$ the ring of integers of $E$, by $\lambda$ its maximal ideal, and by $\kappa\coloneqq\sO/\lambda$ the residue field. Following \cite{CHT08}, we denote by $\sC_\sO^f$ the category of Artinian local rings over $\sO$ with residue field $k$, and by $\sC_\sO$ the category of Noetherian complete local rings over $\sO$ that are inverse limits of objects of $\sC_\sO^f$. For an object $R$ of $\sC_\sO$, we shall denote by $\fm_R$ its maximal ideal. For an $\sO$-valued character, we will use the same notation for its induced $R$-valued character for every object $R$ of $\sC_\sO$. Recall the character $\eta\colon\Gamma_{F^+}\to\{\pm1\}$ associated to the extension $F/F^+$ and the $\ell$-adic cyclotomic character $\epsilon_\ell\colon\Gamma_{F^+}\to\dZ_\ell^\times(\subseteq\sO^\times)$.

\subsection{Deformation problems}
\label{ss:deformation_problems}

In this subsection, we introduce the notion of deformation problems. Let $\tilde\Gamma$ be a topological group, and $\Gamma\subseteq\tilde\Gamma$ an open subgroup of index at most two.

\begin{notation}\label{no:deformation_pair}
We consider a pair $(\bar{r},\chi)$, where
\begin{itemize}
  \item $\bar{r}\colon\tilde\Gamma\to\sG_N(k)$ is a homomorphism,

  \item $\chi\colon\tilde\Gamma\to\sO^\times$ a continuous homomorphism, known as the \emph{similitude character},
\end{itemize}
subject to the relation $\bar{r}^{-1}(\GL_N(k)\times k^\times)=\Gamma$ and $\nu\circ\bar{r}=\chi$.
\end{notation}

The following definition slightly generalizes \cite{CHT08}*{Definition~2.2.1}.

\begin{definition}\label{de:lifting_deformation}
A \emph{lifting} of $\bar{r}$ to an object $R$ of $\sC_\sO$ is a continuous homomorphism $r\colon\tilde\Gamma\to\sG_N(R)$ satisfying $r\modulo\fm_R=\bar{r}$ and $\nu\circ r=\chi$. We say that two liftings are equivalent if they are conjugate by an element in $1+\rM_N(\fm_R)\subset\GL_N(R)\subset\sG_N(R)$. By a \emph{deformation} of $\bar{r}$, we mean an equivalence class of liftings of $\bar{r}$.\footnote{Strictly speaking, a lifting or a deformation of $\bar{r}$ depends on the similitude character $\chi$. But we choose to follow the terminology in \cite{CHT08} by not spelling the characters out, as the relevance on the similitude character is always clear from the context.}
\end{definition}

Now suppose that $\Gamma$ is topologically finitely generated. Then there exists a universal lifting
\[
r^\univ\colon\tilde\Gamma\to\sG_N(\sfR^\loc_{\bar{r}})
\]
of $\bar{r}$ to an object $\sfR^\loc_{\bar{r}}$ of $\sC_\sO$ such that, for every object $R$ of $\sC_\sO$, the set of liftings of $\bar{r}$ to $R$ is in natural bijection with $\Hom_{\sC_\sO}(\sfR^\loc_{\bar{r}},R)$. Since $\Gamma$ is topologically finitely generated, it is well-known that $\sfR^\loc_{\bar{r}}$ is Noetherian; and there exist natural isomorphisms
\[
\Hom_k\(\fm_{\sfR^\loc_{\bar{r}}}/(\fm^2_{\sfR^\loc_{\bar{r}}},\lambda),k\)
\simeq\Hom_{\sC_\sO}\(\sfR^\loc_{\bar{r}}, k[\varepsilon]/(\varepsilon^2)\)
\simeq\rZ^1(\tilde\Gamma,\ad\bar{r}),
\]
where $\rZ^1(\tilde\Gamma,\ad\bar{r})$ denotes the group of $1$-cocycles of $\tilde\Gamma$ with values in the adjoint representation $(\ad\bar{r},\rM_N(k))$. Let $\widehat\GL_{N,\sO}$ be the completion of the group scheme $\GL_{N,\sO}$ along its unit section, which acts naturally on $\Spf\sfR^\loc_{\bar{r}}$ by conjugation.

\if false

Explicitly, for $\phi\in\rZ^1(\tilde\Gamma,\ad\bar{r})$, the corresponding lifting of $\bar{r}$ to $k[\varepsilon]/(\varepsilon^2)$ is given by
\[
r_\phi(g)=(1+\varepsilon\phi(g))\bar{r}(g)
\]
for every $g\in\tilde\Gamma$. For two cocycles $\phi_1,\phi_2\in\rZ^1(\tilde\Gamma,\ad\bar{r})$, the corresponding liftings $r_{\phi_1}$ and $r_{\phi_2}$ are equivalent if and only if there exists an element $x\in\rM_N(k)$ such that
\[
\phi_1(g)-\phi_2(g)=(1-\ad\bar{r}(g))(x)
\]
for every $g\in\tilde\Gamma$. Thus, the equivalence classes of liftings of $\bar{r}$ to $k[\varepsilon]/(\varepsilon^2)$ is in natural bijection with $\rH^1(\tilde\Gamma,\ad\bar{r})$.

\fi

\begin{definition}\label{de:local_deformation_problem}
A \emph{local deformation problem} of $\bar{r}$ is a closed formal subscheme $\sD$ of $\Spf\sfR^\loc_{\bar{r}}$ that is stable under the action of $\widehat\GL_{N,\sO}$.
\end{definition}

%By the moduli interpretation of $\sfR^\loc_{\bar{r}}$, giving a local deformation problem of $\bar{r}$ is equivalent to giving a collection of liftings of $\bar{r}$ to objects in $\sC_\sO$ satisfying certain conditions (see \cite{CHT08}*{Definition~2.2.2 \& Lemma~2.2.3}).

\begin{definition}\label{de:tangent_space_deformation}
For a local deformation problem $\sD$ of $\bar{r}$, we define the \emph{tangent space of $\sD$}, denoted by $\rL(\sD)$, to be the image of the subspace
\[
\rL^1(\sD)\coloneqq
\Hom_k\(\fm_{\sfR^\loc_{\bar{r}}}/(\fm^2_{\sfR^\loc_{\bar{r}}},\sI,\lambda),k\)\subseteq\rZ^1(\tilde\Gamma,\ad\bar{r})
\]
under the natural map $\rZ^1(\tilde\Gamma,\ad\bar{r})\to\rH^1(\tilde\Gamma,\ad\bar{r})$, where $\sI\subseteq\sfR^\loc_{\bar{r}}$ is the closed ideal defining $\sD$.
\end{definition}

Note that we have the identity
\begin{equation}\label{eq:local_deformation_dimension}
\dim_k\rL^1(\sD)=N^2+\dim_k\rL(\sD)-\dim_k\rH^0(\tilde\Gamma,\ad\bar{r}).
\end{equation}

\begin{remark}\label{re:local_deformation}
Later, when we consider a nonarchimedean place $v$ of $F^+$ and take $\tilde\Gamma=\Gamma_{F^+_v}$, the subgroup $\Gamma$ we implicitly take is always $\Gamma_{F^+_v}\cap\Gamma_F$.
\end{remark}

Now we apply Notation \ref{no:deformation_pair} and Definition \ref{de:lifting_deformation} to the case where $\tilde\Gamma=\Gamma_{F^+}$ and $\Gamma=\Gamma_F$.

\begin{definition}\label{de:global_deformation_problem}
A \emph{global deformation problem} is a tuple $(\bar{r},\chi,\rS,\{\sD_v\}_{v\in\rS})$, where
\begin{itemize}
  \item $(\bar{r},\chi)$ is a pair as in Notation \ref{no:deformation_pair};

  \item $\rS$ is a finite set of nonarchimedean places of $F^+$ containing all $\ell$-adic places and those places $v$ such that $\bar{r}_v$ is ramified;

  \item $\sD_v$ is a local deformation problem of $\bar{r}_v$ (Remark \ref{re:local_deformation}) for each $v\in\rS$.
 \end{itemize}
\end{definition}

We take a global deformation problem $\sS\coloneqq(\bar{r},\chi,\rS,\{\sD_v\}_{v\in\rS})$. For $v\in\rS$, we denote by $\sI_v$ the closed ideal of $\sfR^\loc_{\bar{r}_v}$ defining $\sD_v$. For a subset $\rT\subseteq\rS$, put
\begin{align}\label{eq:global_deformation_problem}
\sfR^\loc_{\sS,\rT}\coloneqq\widehat{\bigotimes}_{v\in\rT}\sfR^\loc_{\bar{r}_v}/\sI_v,
\end{align}
where the completed tensor product is taken over $\sO$. Recall from \cite{CHT08}*{Definition~2.2.1} that a \emph{$\rT$-framed lifting} of $\bar{r}$ to an object $R$ of $\sC_\sO$ is a tuple $(r;\{\beta_v\}_{v\in\rT})$, where $r$ is a lifting of $\bar{r}$ to $R$ (Definition \ref{de:lifting_deformation}), and $\beta_v\in1+\rM_N(\fm_{R})$ for $v\in\rT$. Two $\rT$-framed liftings $(r;\{\beta_v\}_{v\in\rT})$ and $(r';\{\beta'_v\}_{v\in\rT})$ of $\bar{r}$ to $R$ are said to be equivalent, if there exists $x\in1+\rM_N(\fm_R)$ such that $r'=x^{-1}\circ r\circ x$ and $\beta_v'=x^{-1}\beta_v$ for every $v\in\rT$. A \emph{$\rT$-framed deformation} of $\bar{r}$ is an equivalence class of $\rT$-framed liftings of $\bar{r}$. We say that a $\rT$-framed lifting $(r;\{\beta_v\}_{v\in\rT})$ is \emph{of type $\sS$} if $r_v$ belongs to $\sD_v$ for every $v\in\rS$, and is unramified for every $v\notin\rS$. Note that being of type $\sS$ is a property invariant under the conjugate action by $1+\rM_N(\fm_R)$. Thus it makes sense to speak of $\rT$-framed deformation of type $\sS$. Let $\Def_\sS^{\Box_\rT}\colon\sC_\sO\to\Set$ be the functor that sends an object $R$ to the set of $\rT$-framed deformations of $\bar{r}$ to $R$ of type $\sS$.

Let $\Gamma_{F^+,\rS}$ be the Galois group of the maximal subextension of $\ol{F}/F^+$ that is unramified outside $\rS$. Recall the cohomology group $\rH^i_{\sS,\rT}(\Gamma_{F^+,\rS},\ad\bar{r})$ for $i\geq 0$ introduced after \cite{CHT08}*{Definition~2.2.7}. By \cite{CHT08}*{Lemma~2.3.4}, these are finite dimensional $k$-vector spaces, and satisfy $\rH^i_{\sS,\rT}(\Gamma_{F^+,\rS},\ad\bar{r})=0$ for $i>3$.

\begin{proposition}\label{pr:deformation_global}
Assume that $\bar{r}\res_{\Gamma_F}$ is absolutely irreducible. Then for every subset $\rT\subseteq\rS$, the functor $\Def_\sS^{\Box_\rT}$ is represented by a Noetherian ring $\sfR^{\Box_\rT}_\sS$ in $\sC_\sO$. Put $\sfR^\univ_\sS\coloneqq\sfR^{\Box_\emptyset}_\sS$.
\begin{enumerate}
  \item There is a canonical isomorphism
      \[
      \Hom_k\(\fm_{\sfR^{\Box_\rT}_\sS}/(\fm^2_{\sfR^{\Box_\rT}_\sS},\lambda,\fm_{\sfR^\loc_{\sS,\rT}}),k\)
      \simeq\rH^1_{\sS,\rT}(\Gamma_{F^+,\rS},\ad\bar{r}),
      \]
      where we regard $\fm_{\sfR^\loc_{\sS,\rT}}$ as its image under the tautological homomorphism $\sfR^\loc_{\sS,\rT}\to\sfR^{\Box_\rT}_\sS$. Moreover, if $\rH^2_{\sS,\rT}(\Gamma_{F^+,\rS},\ad\bar{r})=0$ and for $v\in\rS\setminus\rT$, $\sD_v$ is formally smooth over $\sO$, then $\sfR^{\Box_\rT}_\sS$ is a power series ring over $\sfR^\loc_{\sS,\rT}$ in $\dim_k\rH^1_{\sS,\rT}(\Gamma_{F^+,\rS},\ad\bar{r})$ variables.

  \item The choice of a lifting $r^\univ_\sS\colon\Gamma_{F^+}\to\sG_N(\sfR^\univ_\sS)$ in the universal deformation determines an extension of the tautological homomorphism $\sfR^\univ_\sS\to\sfR^{\Box_\rT}_\sS$ to an isomorphism
      \[
      \sfR^\univ_\sS[[X_{v;i,j}]]_{v\in\rT;1\leq i,j\leq N}\xrightarrow{\sim}\sfR^{\Box_\rT}_\sS
      \]
      such that, for every $v\in\rT$, the universal frame at $v$ is given by $\beta_v=1+(X_{v;i,j})_{1\leq i,j\leq N}$.
\end{enumerate}
\end{proposition}

\begin{proof}
These are exactly \cite{CHT08}*{Proposition~2.2.9 \& Corollary~2.2.13} except that they consider only local deformation problems at split places (that is, they assume that all places in $\rS$ are split in $F$). However, the same argument can be applied to the general case without change.
\end{proof}

\subsection{Fontaine--Laffaille deformations}
\label{ss:fontaine_laffaille}

In this subsection, we study Fontaine--Laffaille deformations of at $\ell$-adic places. We take a place $v$ of $F^+$ above $\ell$; and let $w$ be the place of $F$ above $v$ induced by the inclusion $F\subseteq\ol{F}^+_v$. We assume that $\ell$ is unramified in $F$, and denote by $\sigma\in\Gal(F_w/\dQ_\ell)$ the absolute Frobenius element.

\begin{assumption}\label{as:fl_large}
The field $E$ contains the image of every embedding of $F_w$ into $\ol\dQ_\ell$.
\end{assumption}

We first assume that $E$ satisfies Assumption \ref{as:fl_large}. Put $\Sigma_w\coloneqq\Hom_{\dZ_\ell}(O_{F_w},\sO)$. Following \cite{CHT08}, we use a covariant version of the Fontaine--Laffaille theory \cite{FL82}. Let $\MF_{\sO,w}$ be the category of $O_{F_w}\otimes_{\dZ_\ell}\sO$-modules $M$ of finite length equipped with
\begin{itemize}
  \item a decreasing filtration $\{\Fil^iM\}_{i\in\dZ}$ by $O_{F_w}\otimes_{\dZ_\ell}\sO$-submodules that are $O_{F_w}$-direct summands, satisfying $\Fil^0M=M$ and $\Fil^{\ell-1}M=0$, and

  \item a Frobenius structure, that is, $\sigma\otimes 1$-linear maps $\Phi^i\colon\Fil^iM\to M$ for $i\in\dZ$, satisfying the relations $\Phi^i\res_{\Fil^{i+1}M}=\ell\Phi^{i+1}$ and $\sum_{i\in\dZ}\Phi^i\Fil^iM=M$.
\end{itemize}
Let $\MF_{k,w}$ be the full subcategory of $\MF_{\sO,w}$ of objects that are annihilated by $\lambda$. For every integer $b$ satisfying $0\leq b\leq \ell-2$, let $\MF_{\sO,w}^{[0,b]}$ be the full subcategory of $\MF_{\sO,w}$ consisting of objects $M$ satisfying $\Fil^{b+1}M=0$. In particular, we have $\MF_{\sO,w}^{[0,\ell-2]}=\MF_{\sO,w}$ by definition.

For an object $M$ of $\MF_{\sO,w}$, there is canonical decomposition
\[
M=\bigoplus_{\tau\in\Sigma_w}M_\tau,
\]
where $M_\tau\coloneqq M\otimes_{O_{F_w}\otimes_{\dZ_\ell}\sO,\tau\otimes 1}\sO$. Then we have $\Fil^iM=\bigoplus_{\tau\in\Sigma_w}\Fil^iM_\tau$ with $\Fil^iM_\tau=M_\tau\cap\Fil^iM$, and that $\Phi^i$ induces $\sO$-linear maps
\[
\Phi^i_\tau\colon\Fil^iM_\tau\to M_{\tau\circ\sigma^{-1}}.
\]
We put
\[
\gr^iM_\tau\coloneqq\Fil^iM_\tau/\Fil^{i+1}M_{\tau},\quad
\gr^\bullet M_\tau\coloneqq\bigoplus_i\gr^iM_\tau,\quad
\gr^\bullet M\coloneqq\bigoplus_{\tau\in\Sigma_w} \gr^\bullet M_\tau.
\]
We define the set of \emph{$\tau$-Fontaine--Laffaille weights} of $M$ to be
\[
\HT_\tau(M)\coloneqq\{i\in \dZ \res \gr^iM_\tau\neq 0\}.
\]
We say that $M$ has \emph{regular Fontaine--Laffaille weights} if $\gr^iM_\tau$ is generated over $\sO$ by at most one element for every $\tau\in\Sigma_w$ and every $i\in\dZ$.

Let $\sO[\Gamma_{F_w}]^\flg$ be the category of $\sO$-modules of finite length equipped with a continuous action of $\Gamma_{F_w}$. In \cite{CHT08}*{2.4.1}, the authors defined an exact fully faithful, covariant $\sO$-linear functor
\[
\bG_w\colon\MF_{\sO,w}\to\sO[\Gamma_{F_w}]^\flg
\]
whose essential image is closed under taking sub-objects and quotient objects. The length of an object $M$ in $\MF_{\sO,w}$ as an $\sO$-module equals $[F_w:\dQ_\ell]$ times the length of $\bG_w(M)$ as an $\sO$-module. For two objects $M_1,M_2$ of $\MF_{\sO,w}$, we have a canonical isomorphism
\[
\Hom_{\MF_{\sO,w}}(M_1,M_2)\xrightarrow{\sim}\rH^0(F_w,\Hom_\sO(\bG_w(M_1),\bG_w(M_2)))
\]
and a canonical injective map
\[
\Ext^1_{\MF_{\sO,w}}(M_1,M_2)\hookrightarrow\Ext^1_{\sO[\Gamma_{F_w}]^\flg}(\bG_w(M_1),\bG_w(M_2)),
\]
where the target is canonically isomorphic to $\rH^1(F_w,\Hom_k(\bG_w(M_1),\bG_w(M_2)))$ if $M_1$ and $M_2$ are both objects of $\MF_{k,w}$.

\begin{example}\label{ex:fl_modules}
For an integer $b$ satisfying $0\leq b\leq \ell-2$ and an object $R$ of $\sC_\sO^f$, we have an object $R\{b\}$ of $\MF_{\sO,w}$ defined as follows: the underlying $O_{F_w}\otimes_{\dZ_\ell}\sO$-module is simply $(O_{F_w}\otimes_{\dZ_\ell}R)e_b$, with the filtration given by
\[
\Fil^iR\{b\}=
\begin{dcases}
(O_{F_w}\otimes_{\dZ_\ell}R)e_b & \text{if $i\leq b$;}\\
0 & \text{if $i>b$.}
\end{dcases}
\]
Finally, the Frobenius structure is determined by $\Phi^b(e_b)=e_b$. Then we have
\[
\bG_w(R\{b\})\simeq R(-b)\res_{\Gamma_{F_w}}
\]
as $\sO[\Gamma_{F_w}]$-modules.
\end{example}

\begin{construction}\label{cs:fl_modules}
We construct a functor $\obj^\sigma\colon\MF_{\sO,w}\to\MF_{\sO,w}$ as follows: for an object $M$ of $\MF_{\sO,w}$, the underlying $O_{F_w}\otimes_{\dZ_\ell}\sO$-module of $M^\sigma$ is $O_{F_w}\otimes_{O_{F_w},\sigma}M$ with the induced filtration and Frobenius structure. Then we have $M^\sigma_\tau=M_{\tau\circ\sigma^{-1}}$ for every $\tau\in\Sigma_w$, and that $\bG_w(M^\sigma)$ is isomorphic to $\bG_w(M)$ but with the action of $\Gamma_{F_w}$ twisted by the absolute Frobenius of $F_w$: if we denote by $\rho$ and $\rho_\sigma$ the actions of $\Gamma_{F_w}$ on $\bG_w(M)$ and $\bG_w(M^\sigma)$, respectively, then they satisfy
\[
\rho_\sigma(g)=\rho(\tilde\sigma^{-1}g\tilde\sigma),
\]
where $\tilde\sigma\in\Gal(\ol{F}^+_v/\dQ_\ell)$ is a lift of the absolute Frobenius.

We now let $\obj^\tc\colon\MF_{\sO,w}\to\MF_{\sO,w}$ be the $[F^+_v:\dQ_\ell]$-th iteration of the functor $\obj^\sigma$ constructed above.
\end{construction}

For an object $R$ of $\sC_\sO^f$ and an integer $b$ with $0\leq b\leq\ell-2$, let $\MF_{\sO,w}^{[0,b]}(R)$ be the full subcategory of  $\MF_{\sO,w}^{[0,b]}$ consisting of objects $M$ such that $M$ is finite free over $O_{F_w}\otimes_{\dZ_{\ell}}R$ and that the $O_{F_w}\otimes_{\dZ_{\ell}}R$-submodule $\Fil^iM$ is a direct summand for every $i$. Let $M$ be an object of $\MF_{\sO,w}^{[0,b]}(R)$. Then $\bG_w(M)$ is finite free over $R$. Thus $\bG_w$ induces a functor $\bG_w\colon\MF_{\sO,w}^{[0,b]}(R)\to R[\Gamma_{F_w}]^{\r{f.r.}}$, where $R[\Gamma_{F_w}]^{\r{f.r.}}$ denotes the category of finite free $R$-modules equipped with a continuous action by $\Gamma_{F_w}$. We have a functor
\[
\obj^\vee\{b\}\colon\MF_{\sO,w}^{[0,b]}(R)^\op\to\MF_{\sO,w}^{[0,b]}(R)
\]
defined as follows: for an object $M$ of $\MF_{\sO, w}^{[0,b]}(R)$, $M^\vee\{b\}$ is the object of $\MF_{\sO,w}^{[0,b]}(R)$ such that:
\begin{itemize}
  \item its underlying $O_{F_w}\otimes_{\dZ_{\ell}}R$-module is $\Hom_{O_{F_w}\otimes_{\dZ_{\ell}}R}(M,O_{F_w}\otimes_{\dZ_{\ell}}R)$,

  \item $\Fil^iM^\vee\{b\}=\Hom_{O_{F_w}\otimes_{\dZ_{\ell}}R}(M/\Fil^{b+1-i}M,O_{F_w}\otimes_{\dZ_{\ell}}R)$,

  \item for every $f\in \Fil^iM^{\vee}\{b\}$ and every $m\in\Fil^jM$, we have
	  \[
      \Phi^i(f)(\Phi^j(m))=
      \begin{dcases}
	  \ell^{b-i-j} f(m)^\sigma &\text{if $i+j\leq b$};\\
	  0 &\text{if $i+j>b$}.
	  \end{dcases}
      \]
\end{itemize}
It is clear that $M^\vee\{b\}$ is a well-defined object of $\MF_{\sO,w}^{[0,b]}(R)$ (see \cite{CHT08}*{Page~34}), and that  $\bG_w(M^\vee\{b\})=\bG_w(M)^\vee(-b)$.
		
Now suppose that we have an isomorphism of $R[\Gamma_{F_w}]$-modules $\bG_w(M^\tc)\simeq\bG_w(M)^\vee(-b)$. Since the functor $\bG_w$ is fully faithful, giving such an isomorphism is equivalent to giving an isomorphism $M^\tc\simeq M^\vee\{b\}$ in $\MF_{\sO,w}^{[0,b]}(R)$, or equivalently an $O_{F_w}\otimes_{\dZ_\ell}R$-bilinear perfect pairing
\begin{align}\label{eq:pairing}
\langle\;, \;\rangle \colon M^\tc\times M  \to  O_{F_w}\otimes_{\dZ_\ell}R,
\end{align}
such that the induced $R$-bilinear perfect pairings $\langle\;,\;\rangle_\tau\colon M_{\tau^\tc}\times M_\tau\to R$ for $\tau\in\Sigma_w$ satisfy the following conditions:
\begin{enumerate}
  \item For every $i,j\in\dZ$, every $x\in\Fil^iM_\tau$, and every $y\in\Fil^jM_\tau$, $\langle\Phi^i_{\tau^\tc}x,\Phi^j_\tau y\rangle_{\tau\circ\sigma^{-1}}$ equals $\ell^{b-i-j}\langle x, y\rangle_\tau$ (resp.\ $0$) if $i+j\leq b$ (resp.\ $i+j> b$).
			
  \item For every $i\in\dZ$, the annihilator of $\Fil^iM_\tau$ under $\langle\;,\;\rangle_\tau$ is $\Fil^{b+1-i}M_{\tau^\tc}$; in particular, $\langle\;,\;\rangle_\tau$ induces an $R$-linear isomorphism $\gr^iM_\tau\simeq\Hom_R(\gr^{b-i}M_{\tau^\tc},R)$.
\end{enumerate}

\if false

Take an object $M$ of $\MF_{\sO,w}$. Suppose that $M$ is finite free over $O_{F_w}\otimes R$ for some object $R$ of $\sC_\sO^f$. Then giving an isomorphism $M\simeq\rD^{[0,b]}(M^\tc)$ is equivalent to giving a perfect pairing
\[
\langle\;, \;\rangle \colon M^\tc\times M  \to  R\{b\}
\]
in the category $\MF_{\sO,w}$, where $R\{b\}$ is the object in Example \ref{ex:fl_modules}. The latter is equivalent to giving, for each $\tau\in\Sigma_w$, an $R$-bilinear perfect pairing $\langle\;,\;\rangle_\tau\colon M_{\tau^\tc}\times M_\tau\to(\dQ_\ell/\dZ_\ell)\otimes R$ satisfying that
\begin{enumerate}
  \item for every $i,j\in\dZ$ and every $x\in\Fil^iM_\tau$ and $y\in\Fil^jM_\tau$, $\langle\Phi^i_{\tau^\tc}x,\Phi^j_\tau y\rangle_\tau$ equals $\ell^{b-i-j}\langle x, y\rangle_\tau$ (resp.\ $0$) if $i+j\leq b$ (resp.\ $i+j> b$); and

  \item for every $i\in\dZ$, the annihilator of $\Fil^iM_\tau$ under $\langle\;,\;\rangle_\tau$ is $\Fil^{b+1-i}M_{\tau^\tc}$; in particular, $\langle\;,\;\rangle_\tau$ induces an $R$-linear isomorphism $\gr^iM_\tau\simeq\Hom_R(\gr^{b-i}M_{\tau^\tc},(\dQ_\ell/\dZ_\ell)\otimes R)$.
\end{enumerate}

\fi

From now on, $E$ will not necessarily be subject to Assumption \ref{as:fl_large}.

\begin{definition}\label{de:crystalline_fl}
Let $R$ be an object of $\sC_\sO$, and $\rho\colon\Gamma_{F_w}\to\GL_N(R)$ a continuous representation.
\begin{enumerate}
  \item Let $a,b$ be integers satisfying $0\leq b-a\leq \ell-2$. For $E$ satisfying Assumption \ref{as:fl_large}, we say that $\rho$ is \emph{crystalline with (regular) Fontaine--Laffaille weights in $[a,b]$} if, for every quotient $R'$ of $R$ in $\sC_\sO^f$, $\rho(a)\otimes_RR'$ lies in the essential image of the functor $\bG_w\colon\MF_{\sO,w}^{[0,b-a]}(R')\to\cO[\Gamma_{F_w}]^\flg$ (and that $\bG_w^{-1}(\rho(a)\otimes_RR')$ has regular Fontaine--Laffaille weights).

  \item For $E$ in general, we say that $\rho$ is \emph{regular Fontaine--Laffaille crystalline} if there exists a finite unramified extension $E_\dag$ of $E$ contained in $\ol\dQ_\ell$ that satisfies Assumption \ref{as:fl_large} with the ring of integer $\sO_\dag$, such that $\rho\otimes_\sO\sO_\dag$ is crystalline with regular Fontaine--Laffaille weights in $[a,b]$ in the sense of (1) for some integers $a,b$ satisfying $0\leq b-a\leq \ell-2$.
\end{enumerate}
\end{definition}

Now we consider a pair $(\bar{r},\chi)$ from Notation \ref{no:deformation_pair} with $\tilde\Gamma=\Gamma_{F^+_v}$ and $\Gamma=\Gamma_{F^+_v}\cap\Gamma_F=\Gamma_{F_w}$.

\begin{definition}\label{de:deformation_fl}
Suppose that $\bar{r}^\natural$ is regular Fontaine--Laffaille crystalline. We define $\sD^\FL$ to be the local deformation problem of $\bar{r}$ that classifies the liftings $r\colon\Gamma_{F^+_v}\to\sG_N(R)$ of $\bar{r}$ to objects $R$ of $\sC_\sO$ such that $r^\natural$ is regular Fontaine--Laffaille crystalline.
\end{definition}

\begin{lem}\label{le:tangent_space_fl}
Suppose that $\bar{r}^\natural$ is regular Fontaine--Laffaille crystalline and that $\chi=\eta_v^\mu\epsilon_{\ell,v}^c$ for some $c\in\dZ$ and $\mu\in\dZ/2\dZ$. Then
\[
\dim_k\rL(\sD^\FL)-\dim_k\rH^0(F_v^+,\ad\bar{r})=[F_v^+:\dQ_\ell]\cdot\frac{N(N-1)}{2}.
\]
\end{lem}

\begin{proof}
For this lemma, we may assume that $E$ satisfies Assumption \ref{as:fl_large}. After replacing $\bar{r}$ by $\bar{r}(a)$ for some integer $a$, we may assume that $\bar{r}^\natural$ is crystalline with regular Fontaine--Laffaille weights in $[0,b]$ with $0\leq b\leq\ell-2$. In this case, $\sD^\FL$ simply classifies liftings $r$ such that $r^\natural$ is crystalline with Fontaine--Laffaille weights in $[0,b]$.

Suppose first that $v$ is split in $F$. Then we have $F_w=F^+_v$, and that a lifting $r$ in $\sD^\FL(R)$ of $\bar{r}$ is of the form $r=(\rho,\epsilon_{\ell,v}^c)\colon\Gamma_{F_w}\to\GL_N(R)\times R^\times$ such that for every Artinian quotient $R'$ of $R$,  $\rho\otimes_R R'$ lies in the essential image of the functor $\bG_w$. Then the lemma is exactly \cite{CHT08}*{Corollary~2.4.3}.

Suppose now that $v$ is inert in $F$. Then we must have $c=-b$. Denote by $\Gamma_{w/v}$ the Galois group of the quadratic extension $F_w/F^+_v$. Then the restriction map induces an isomorphism
\[
\rH^1(F_v^+,\ad\bar{r})\xrightarrow{\sim}\rH^1(F_w,\ad\bar{r})^{\Gamma_{w/v}}.
\]
Put $M\coloneqq\bG_w^{-1}(\bar{r}^\natural)$. Then the deformations of $\bar{r}$ to $k[\varepsilon]/(\varepsilon^2)$ that lie in the essential image of $\bG_w$ are classified by $\Ext^1_{\MF_{k,w}}(M,M)$, which is canonically a $\Gamma_{w/v}$-stable subspace of $\rH^1(F_w,\ad\bar{r})$. Therefore, we have
\[
\rL(\sD^\FL)=\Ext^1_{\MF_{k,w}}(M,M)\cap\rH^1(F_w,\ad\bar{r})^{\Gamma_{w/v}}=\Ext^1_{\MF_{k,w}}(M, M)^{\Gamma_{w/v}}.
\]
In fact, the induced action of $\Gamma_{w/v}$ on $\Ext^1_{\MF_{k,w}}(M,M)$ can be described as follows. Recall the functor $\obj^\tc$ in Construction \ref{cs:fl_modules}. Then $\bG_w(M^\tc)$ is isomorphic to $\bar{r}^{\natural,\tc}|_{\Gamma_{F_w}}$. Since $\bar{r}^{\natural,\tc}$ and $\bar{r}^{\natural,\vee}(-b)$ are conjugate, we have $M^\tc\simeq M^\vee\{b\}$. We fix such an isomorphism, hence obtain a pairing $\langle\;,\;\rangle$ \eqref{eq:pairing} with $R=k$. Then for an element $[P]\in\Ext^1_{\MF_{k,w}}(M,M)$ represented by an extension $0\to M\to P\to M\to 0$, the image of $[P]$ under the action of the (unique) non-trivial element in $\Gamma_{w/v}$ is obtained by applying the functor $(\obj^\tc)^\vee\{b\}$ to $0\to M\to P\to M\to 0$.

To compute $\Ext^1_{\MF_{k,w}}(M,M)^{\Gamma_{w/v}}$, we recall first the following long exact sequence in \cite{CHT08}*{Lemma~2.4.2}:
\begin{align}\label{eq:extension_fl}
\resizebox{\hsize}{!}{
\xymatrix{
0 \ar[r] &  \End_{\MF_{k,w}}(M) \ar[r] & \Fil^0\Hom_{O_{F_w}\otimes_{\dZ_\ell}\sO}(M,M) \ar[r]^-{\alpha}
& \Hom_{O_{F_w}\otimes_{\dZ_\ell}\sO,\sigma\otimes 1}(\gr^\bullet M,M) \ar[r]^-{\beta} & \Ext^1_{\MF_{k,w}}(M,M) \ar[r] & 0,
}
}
\end{align}
where
\begin{itemize}
  \item $\Fil^0\Hom_{O_{F_w}\otimes_{\dZ_\ell}\sO}(M,M)$ denotes the $O_{F_w}\otimes_{\dZ_\ell}\sO$-submodule of $\Hom_{O_{F_w}\otimes_{\dZ_\ell}\sO}(M,M)$ of endomorphisms that preserve the filtration;

  \item the map $\alpha$ takes an element $f\in\Fil^0\Hom_{O_{F_w}\otimes_{\dZ_\ell}\sO}(M,M)$ to $(f\Phi^i-\Phi^if)_{i\in\dZ}$; and

  \item the map $\beta$ is defined as follows: if $\varphi=(\varphi^i)_{i\in\dZ}$ is a $\sigma\otimes 1$-linear map from $\gr^\bullet M$ to $M$, then $\beta(\varphi)$ is given by the extension class of $E=M\oplus M$ with the filtration $\Fil^iE=\Fil^iM\oplus\Fil^iM$ and the Frobenius structure
      \[
      \Phi^i_E\coloneqq
      \begin{pmatrix}
      \Phi^i & \varphi^i\\
      0 & \Phi^i
      \end{pmatrix}.
      \]
\end{itemize}

To prove the lemma, we need to derive an analogous long exact sequence similar to \eqref{eq:extension_fl} but with the last term $\Ext^1_{\MF_{k,w}}(M,M)^{\Gamma_{w/v}}$. For the first term, note that we have a canonical isomorphism $\End_{\MF_{k,w}}(M)\simeq\rH^0(F_w,\ad\bar{r})$, which contains $\rH^0(F^+_v,\ad\bar{r})$ as a submodule. For the second term, let $\Fil^0\Hom_{O_{F_w}\otimes_{\dZ_\ell}\sO}(M,M)^+$ be the submodule of $\Fil^0\Hom_{O_{F_w}\otimes_{\dZ_\ell}\sO}(M,M)$ consisting of elements $f=(f_\tau)_{\tau\in\Sigma_w}$ such that $-f_{\tau^\tc}$ is the adjoint of $f_\tau$ under the pairing $\langle\;,\;\rangle_\tau$ for every $\tau\in\Sigma_w$. For the third term, let $\Hom_{O_{F_w}\otimes_{\dZ_\ell}\sO,\sigma\otimes 1}(\gr^\bullet M, M)^+$ denote by the submodule of $\Hom_{O_{F_w}\otimes_{\dZ_\ell}\sO,\sigma\otimes 1}(\gr^\bullet M,M)$ consisting of $\varphi=(\varphi^i)_{i\in \dZ}$ such that
\begin{equation}\label{eq:relation_fl}
\langle\Phi^i_{\tau^\tc}(x),\varphi^{b-i}_\tau(y)\rangle_\tau + \langle\varphi^i_{\tau^\tc}(x),\Phi^{b-i}_\tau(y)\rangle_\tau =0
\end{equation}
is satisfies for every $x\in\gr^iM_{\tau^\tc}$ and $y\in\gr^{b-i}M_{\tau}$.

Then \eqref{eq:extension_fl} induces an exact sequence
\[
\resizebox{\hsize}{!}{
\xymatrix{
0 \ar[r] & \rH^0(F^+_v,\ad\bar{r}) \ar[r] & \Fil^0\Hom_{O_{F_w}\otimes_{\dZ_\ell}\sO}(M,M)^+ \ar[r]^-{\alpha}
& \Hom_{O_{F_w}\otimes_{\dZ_\ell}\sO,\sigma\otimes 1}(\gr^\bullet M,M)^+ \ar[r]^-{\beta} & \Ext^1_{\MF_{k,w}}(M,M)^{\Gamma_{w/v}} \ar[r] & 0
}
}
\]
of $k$-vector spaces. We now compute the dimension of the middle two terms. From the description of $\Fil^0\Hom_{O_{F_w}\otimes_{\dZ_\ell}\sO}(M,M)^+$, it is clear that $f_{\tau^\tc}$ is determined by $f_\tau$ for every $\tau\in\Sigma_w$.
On the other hand, for each fixed $\tau$, all the possible choices of $f_\tau$ form a $k$-vector space of dimension $\tfrac{N(N+1)}{2}$. Thus, we have
\[
\dim_k\Fil^0\Hom_{O_{F_w}\otimes_{\dZ_\ell}\sO}(M,M)^+=[F_v^+:\dQ_\ell]\cdot\frac{N(N+1)}{2}.
\]
For $\Hom_{O_{F_w}\otimes_{\dZ_\ell}\sO,\sigma\otimes 1}(\gr^\bullet M,M)^+$, we first note that the map
\[
\bigoplus_i \Phi^i_\tau\colon\gr^\bullet M_\tau\to M_{\tau\circ \sigma^{-1}}
\]
is an isomorphism for every $\tau\in\Sigma_w$. It follows from \eqref{eq:relation_fl} that $\varphi_{\tau^\tc}\coloneqq\bigoplus_i\varphi^i_{\tau^\tc}$ is determined by $\varphi_\tau\coloneqq\bigoplus_i\varphi^i_\tau$. On the other hand, for each fixed $\tau$, all the possible choices of $\varphi_\tau\colon\gr^\bullet M_\tau\to M_{\tau\circ\sigma^{-1}}$ form a $k$-vector space of dimension $N^2$. Thus, we have
\[
\dim_k\Hom_{O_{F_w}\otimes_{\dZ_\ell}\sO,\sigma\otimes 1}(\gr^\bullet M,M)^+=[F_v^+:\dQ_\ell]\cdot N^2.
\]
The Lemma follows immediately.
\end{proof}

\begin{proposition}\label{pr:deformation_fl}
Suppose that $\bar{r}^\natural$ is regular Fontaine--Laffaille crystalline and that $\chi=\eta_v^\mu\epsilon_{\ell,v}^c$ for some $c\in\dZ$ and $\mu\in\dZ/2\dZ$. Then the local deformation problem $\sD^\FL$ is formally smooth over $\Spf\sO$ of pure relative dimension $N^2+[F_v^+:\dQ_\ell]\cdot\frac{N(N-1)}{2}$.
\end{proposition}

\begin{proof}
By Lemma \ref{le:tangent_space_fl}, it suffices to show that $\sD^\FL$ is formally smooth over $\sO$. For this, we may again assume that $E$ satisfies Assumption \ref{as:fl_large}. Moreover, we may assume that $\bar{r}^\natural$ is crystalline with regular Fontaine--Laffaille weights in $[0,b]$ with $0\leq b\leq\ell-2$ as in the proof of Lemma \ref{le:tangent_space_fl}. In this case, $\sD^\FL$ simply classifies liftings $r$ such that $r^\natural$ is crystalline with Fontaine--Laffaille weights in $[0,b]$.

When $v$ is split in $F/F^+$, the proposition has been proved in \cite{CHT08}*{Lemma~2.4.1}.

Now we suppose that $v$ is inert in $F$. Then we must have $c=-b$. Fix a subset $\Sigma_w^+\subset\Sigma_w$ such that $\Sigma_w=\Sigma_w^+\coprod\Sigma_w^{+,\tc}$. Let $R$ be an object of $\sC_\sO^f$ and $I\subset R$ an ideal satisfying $\fm_RI=(0)$. Let $r$ be a lifting of $\bar{r}$ to $R/I$, and put $M\coloneqq\bG^{-1}_w(r^\natural)$, which is an object of $\MF_{\sO,w}^{[0,b]}(R/I)$.

Recall the functor $\obj^\tc$ in Construction \ref{cs:fl_modules}. Then $\bG_w(M^\tc)$ is isomorphic to $r^{\natural,\tc}|_{\Gamma_{F_w}}$. Since $r^{\natural,\tc}$ and $r^{\natural,\vee}(-b)$ are conjugate, we have $M^\tc\simeq M^\vee\{b\}$. We fix such an isomorphism, hence obtain a pairing $\langle\;,\;\rangle$ \eqref{eq:pairing} with $R=R/I$. Let $m_{\tau,1}<\cdots<m_{\tau,N}$ be the (regular) $\tau$-Hodge--Tate weights of $M$ for every $\tau\in\Sigma_w$. Then there exists a basis $e_{\tau,1},\dots,e_{\tau,N}$ of $M_\tau$ over $R/I$ satisfying $\Fil^{m_{\tau,N+1-i}}M_\tau=\bigoplus_{j=1}^i(R/I)e_{\tau,j}$ for every $1\leq i\leq N$. By duality, we have $m_{\tau^\tc,i}+m_{\tau,N+1-i}=b$. Then we may choose the basis $(e_{\tau,i})$ such that $\langle e_{\tau^\tc,i},e_{\tau,j}\rangle_\tau=\delta_{i,N+1-j}$ for every $\tau\in\Sigma_w^+$ and every $1\leq i,j\leq N$.

We now define an object $\wt{M}=\bigoplus_{\tau\in\Sigma_w}\wt{M}_\tau$ of $\MF_{\sO,w}^{[0,b]}(R)$ that reduces to $M$, together with a perfect pairing $\wt{M}^\tc\times\wt{M}\to O_{F_w}\otimes_{\dZ_\ell}R$ as in \eqref{eq:pairing} that reduces to the pairing $\langle\;,\;\rangle$, as follows. As an $R$-module, we take $\wt{M}_\tau=R^{\oplus N}$ with the basis $(\wt{e}_{\tau,i})$ that lifts the basis $(\wt{e}_{\tau,i})$ of $M_\tau$. We lift $\langle\;,\;\rangle_\tau$ to an $R$-bilinear perfect pairing   $\wt{M}_{\tau^\tc}\times\wt{M}_\tau\to R$ such that $\langle\wt{e}_{\tau^\tc,i},\wt{e}_{\tau,j}\rangle_\tau=\delta_{i,N+1-j}$ still holds for every $\tau\in\Sigma_w^+$ and every $1\leq i,j\leq N$. For the filtration, we put $\Fil^m\wt{M}_\tau\coloneqq\bigoplus_{j=1}^i R\wt{e}_{\tau,j}$ for $m$ satisfying $m_{\tau,N-i}<m\leq m_{\tau,N+1-i}$. Then $\wt{M}\otimes_RR/I$ is isomorphic to $M$ as filtered $O_{F_w}\otimes R/I$-modules; and the condition (2) in Construction \ref{cs:fl_modules} holds for $\wt{M}$ as well. For the Frobenius structure on $\wt{M}$, we first define maps $\wt\Phi^{m_{\tau,i}}_\tau\colon\Fil^{m_{\tau,i}}\wt{M}_\tau\to\wt{M}_{\tau\circ\sigma^{-1}}$ for $\tau\in\Sigma_w^+$ by the recursive induction on $i$. For $i=N$, we take $\wt\Phi^{m_{\tau,N}}_\tau$ to be an arbitrary lift of $\Phi^{m_{\tau,N}}_\tau\colon\Fil^{m_{\tau,N}}M_\tau\to M_{\tau\circ\sigma^{-1}}$ for $\tau\in \Sigma_w^+$. For $i\leq N-1$, we take $\wt\Phi^{m_{\tau,i}}_\tau$ to be a lift of $\Phi^{m_{\tau,i}}_\tau\colon\Fil^{m_{\tau,i}}M_\tau\to M_{\tau\circ\sigma^{-1}}$ that restricts to $\ell^{m_{\tau,i}-m_{\tau,i+1}}\wt\Phi^{m_{\tau,i+1}}_\tau$ on $\Fil^{m_{\tau,i+1}}\wt{M}_\tau$. By Nakayama's lemma, we have
\[
\wt{M}_{\tau\circ\sigma^{-1}}=\sum_i\wt\Phi^{m_{\tau,i}}_\tau(\Fil^{m_{\tau,i}}\wt{M}_\tau)
\]
for every $\tau\in\Sigma_w^+$. Finally, we define $\wt\Phi^i_{\tau^\tc}\colon\wt{M}_{\tau^\tc}\to\wt{M}_{\tau^\tc\circ\sigma^{-1}}$ for $\tau\in\Sigma_w^+$ to be the unique $R$-linear map satisfying the condition (1) in Construction \ref{cs:fl_modules} for $\wt{M}$. This finishes the construction of $\wt{M}$ together with an isomorphism $\wt{M}^\tc\simeq\wt{M}^\vee\{b\}$, which give rise to a lifting $\wt{r}$ of $\bar{r}$ to $R$ that reduces to $r$ by Lemma \ref{le:representation_selfdual}. Thus, $\sD^\FL$ is formally smooth over $\sO$.

The proposition is proved.
\end{proof}

At the end, we remark that in the self-dual (not conjugate self-dual) case, the Fontaine--Laffaille deformations have been studied in \cite{Boo18}.

\subsection{Representations of the tame group}
\label{ss:tame_group}

In this subsection, we will study conjugate self-dual representations of the tame group, and define the notion of minimally ramified deformations of such representations.

\begin{definition}\label{de:tame_group}
Let $q\geq 1$ be a positive integer coprime to $\ell$. We define the \emph{$q$-tame group}, denoted by $\rT_q$, to be the semidirect product topological group $t^{\dZ_\ell}\rtimes\phi_q^{\widehat\dZ}$ where $\phi_q$ maps $t$ to $t^q$, that is, $\phi_qt\phi_q^{-1}=t^q$. For every integer $b\geq 1$, We identify $\rT_{q^b}$ as a subgroup of $\rT_q$ topologically generated by $t$ and $\phi_{q^b}=\phi_q^b$.
\end{definition}

We consider a reductive group $G$ over $\sO$, together with a surjective homomorphism $\nu\colon G\to H$ over $\sO$, where $H$ is an algebraic group over $\sO$ of multiplicative type. Consider a pair $(\bar\varrho,\mu)$ in which $\bar\varrho\colon\rT_q\to G(k)$ and $\mu\colon\rT_q\to H(\sO)$ are continuous homomorphisms satisfying $\nu\circ\bar\varrho=\bar\mu$ and $\mu(t)=1$. Similar to the case in \S\ref{ss:deformation_problems}, let $\sfR^\loc_{\bar\varrho}$ be the object in $\sC_\sO$ that parameterizes liftings $\varrho$ of $\bar\varrho$ satisfying $\nu\circ\varrho=\mu$.\footnote{Here, once again we omit the similitude character $\mu$ in the ring $\sfR^\loc_{\varrho}$, in order to be consistent with the previous convention.} The following proposition generalizes the tame case of \cite{Sho18}*{Theorem~2.5}.

\begin{proposition}\label{pr:tame_general}
The ring $\sfR^\loc_{\bar\varrho}$ is a local complete intersection, flat and of pure relative dimension $d$ over $\sO$, where $d$ is the relative dimension of the kernel of $\nu$ over $\sO$.
\end{proposition}

\begin{proof}
We follow the same line as in the proof of \cite{Sho18}*{Theorem~2.5}. Let $G_0$ and $G_1$ be the fibers at $1$ and $\mu(\phi_q)$ of the homomorphism $\nu$, respectively. Define the subscheme $\sM(G,q)$ of $G_0\times_{\Spec\sO}G_1$ such that for every object $R$ of $\sC_\sO$, $\sM(G,q)(R)$ consists of pairs $(\bbA,\bbB)\in G_0(R)\times G_1(R)$ satisfying
\begin{align}\label{eq:tame_general}
\bbB\bbA\bbB^{-1}=\bbA^q.
\end{align}
It suffices to show that $\sM(G,q)$ is a local complete intersection, flat and of pure relative dimension $d$ over $\sO$, since $\sfR^\loc_{\bar\varrho}$ is the completion of $\sM(G,q)$ at the $k$-point $(\bar\varrho(t),\bar\varrho(\phi_q))$.

First, we show that every geometric fiber of $\sM(G,q)\to\Spec\sO$ is of pure dimension $d$. Consider the natural projection
\[
p\colon\sM(G,q)\to G_0
\]
to the first factor. Take a geometric point $\Spec K\to\Spec\sO$. For a point $\bbA_0\in G_0(K)$ in the image of $p(K)$, let $Z(\bbA_0)$ be the centralizer of $\bbA_0$ in $G_{0,K}$ as a closed subscheme of $G_{0,K}$, and $C(\bbA_0)$ the conjugacy class of $\bbA_0$, which is a locally closed subscheme of $G_{0,K}$ isomorphic to $G_{0,K}/Z(\bbA_0)$. Then $C(\bbA_0)$ lies in the image of $p_K$. For every point $\bbA\in C(\bbA_0)(K)$, the fiber $p_K^{-1}(\bbA)$ is a torsor under the group $Z(\bbA)$, which is conjugate to $Z(\bbA_0)$. Thus, $p_K^{-1}(C(\bbA_0))$ is irreducible of dimension
\[
\dim p_K^{-1}(C(\bbA_0))=\dim C(\bbA_0)+\dim Z(\bbA_0)=\dim G_{0,K}=d.
\]
To continue, we choose an embedding $e\colon G_K\to\GL_{m,K}$ of algebraic groups over $K$ for some integer $m\geq 1$. By \eqref{eq:tame_general}, the image of $e(K)\circ p(K)$ consists only of matrices whose generalized eigenvalues are $(q^{m!}-1)$-th roots of unity, hence finitely many conjugacy classes in $\GL_m(K)$. We claim that the image of $p(K)$ consists of finitely many conjugacy classes in $G_0(K)$ as well, which implies that $\sM(G,q)_K$ is of pure dimension $d$. In fact, we have the following commutative diagram
\[
\xymatrix{
G_0(K)/\!\!/G_0(K)  \ar[r]\ar[d] & \GL_m(K)/\!\!/\GL_m(K) \ar[d] \\
(G_{0,K}/\!\!/G_{0,K})(K) \ar[r] & (\GL_{m,K}/\!\!/\GL_{m,K})(K)
}
\]
of sets, in which the bottom map is finite since the morphism $G_{0,K}/\!\!/G_{0,K}\to\GL_{m,K}/\!\!/\GL_{m,K}$ is; and the left vertical map is also finite due to the finiteness of unipotent conjugacy classes of a reductive group \cite{Spa82}*{Th\'{e}or\`{e}me~4.1}; it follows that the upper horizontal map is finite as well.

The above discussion shows that the morphism $\sM(G,q)\to\Spec\sO$ is of pure relative dimension $d$. Now we take a closed point $(\bar\bbA,\bar\bbB)$ of $\sM(G,q)$, which induces a homomorphism
\[
\sO_{\sM(G,q),(\bar\bbA,\bar\bbB)}\to \sO_{G_0,\bar\bbA}\widehat\otimes_\sO\sO_{G_1,\bar\bbB}
\]
of corresponding complete local rings. As both $G_0$ and $G_1$ are smooth over $\sO$ of pure relative dimension $d$, both $\sO_{G_0,\bar\bbA}$ and $\sO_{G_1,\bar\bbB}$ are power series rings over $\sO$ in $d$ variables. The relation \eqref{eq:tame_general}, or equivalently, the relation $\bbA=\bbB^{-1}\bbA^q\bbB$, is defined by $d$ equations in $\sO_{G_0,\bar\bbA}\widehat\otimes_\sO\sO_{G_1,\bar\bbB}$. In other words, $\sM(G,q)$ is a local complete intersection, hence Cohen-Macaulay. Therefore, $\sM(G,q)$ is flat over $\sO$. The proposition is proved.
\end{proof}

Take an integer $n\geq 1$. Now we apply the above discussion to the homomorphism $\nu\colon\sG_n\to\GL_1$ in Notation \ref{no:sg}. Consider a pair $(\bar\varrho,\mu)$ from Notation \ref{no:deformation_pair} with $\tilde\Gamma=\rT_q$ and $\Gamma=\rT_{q^2}$, such that $\mu(t)=1$. In particular, $\bar\varrho\colon\rT_q\to\sG_n(k)$ is a homomorphism and $\mu\colon\rT_q\to\sO^\times$ is a (continuous) similitude character. Write
\begin{align}\label{eq:conjugate_tame_0}
\bar\varrho(t)=\bar\bbA=(\bar{A},1)1,\quad
\bar\varrho(\phi_q)=\bar\bbB=(\bar{B},-\mu(\phi_q))\fc
\end{align}
for $\bar{A},\bar{B}\in\GL_n(k)$. For a lifting $\varrho$ of $\bar\varrho$ to an object $R$ of $\sC_\sO$, we write $\varrho(t)=\bbA=(A,1)1$ and $\varrho(\phi_q)=\bbB=(B,-\mu(\phi_q))\fc$. Then the pair $(A,B)$ reduce to $(\bar{A},\bar{B})$, and satisfy the relation
\begin{equation}\label{eq:conjugate_tame}
B \tp{A}^{-1} B^{-1}=A^q.
\end{equation}

\begin{corollary}\label{pr:deformation_tame}
The ring $\sfR^\loc_{\bar\varrho}$ is a local complete intersection, flat and of pure relative dimension $n^2$ over $\sO$.
\end{corollary}

\begin{proof}
This follows immediately from Proposition \ref{pr:tame_general} since the kernel of $\nu\colon\sG_n\to\GL_1$ is of dimension $n^2$.
\end{proof}

From now till the end of this subsection, we assume $\ell\geq n$. For every integer $m\geq 1$, we denote by $J_m$ the standard upper triangular nilpotent Jordan block
\[
\begin{pmatrix}
0 & 1 & 0 & \cdots & 0 \\
& 0 & 1 & \cdots & 0 \\
&  & \ddots & \ddots & \vdots \\
&&& 0 & 1\\
&&&& 0
\end{pmatrix}
\]
or size $m$.

Denote by $\cN_n$ (resp.\ $\cU_n$) the closed subscheme of $\rM_n$ (resp.\ $\GL_n$) defined by the equation $X^n=0$ (resp.\ $(A-1)^n=0$). For every object $R$ of $\sC_\sO$, we have the \emph{truncated exponential map} $\exp\colon\cN_n(R)\to\cU_n(R)$ defined by the formula
\[
\exp X=1+X+\cdots+\frac{X^{n-1}}{(n-1)!},
\]
which is an bijection. Its inverse is given by the \emph{truncated logarithm map} $\log\colon\cU_n(R)\to\cN_n(R)$ defined by the formula
\[
\log A=\sum_{i=1}^{n-1}(-1)^{i-1}\frac{(A-1)^i}{i}.
\]
Let $\fP_n$ be the set of partitions of $n$. By the classification of nilpotent orbits in $\GL_n$, for $K=k,E$, we have canonical surjective maps $\pi\colon\cN_n(K)\to\fP_n$ such that the fibers of $\pi$ are exactly the orbits in $\cN_n(K)$ under the conjugate action of $\GL_n(K)$.

By the continuity of $\bar\varrho$, we know that $\bar{A}$ in \eqref{eq:conjugate_tame_0} is unipotent, which implies $\bar{A}\in\cU_n(k)$. Put $\bar{X}\coloneqq\log\bar{A}\in\cN_n(k)$. Following \cite{Boo19}*{Definition~3.9}, we define the functor $\Nilp_{\bar{X}}\colon\sC_\sO\to\Set$ that sends an object $R$ of $\sC_\sO$ to the set of elements $X\in\cN_n(R)$ that reduce to $\bar{X}$ and are of the form $CX_0C^{-1}$, where $X_0$ is an element in $\cN_n(\sO)$ satisfying $\pi(X_0)=\pi(\bar{X})$ and $C\in\GL_n(R)$, where we regard $X_0$ as an element in $\cN_n(E)$ in the notation $\pi(X_0)$.

\begin{definition}\label{de:minimal_tame}
We say that a lifting $\varrho$ of $\bar\varrho$ to an object $R$ of $\sC_\sO$ is \emph{minimally ramified} if there exists an element $X\in\Nilp_{\bar{X}}(R)$ such that $\varrho^\natural(t)=\exp X$.
\end{definition}

Let $\sD^\mnm_{\bar\varrho}$ be the local deformation problem of $\bar\varrho$ (Definition \ref{de:local_deformation_problem}) that classifies minimally ramified liftings of $\bar\varrho$.

\begin{proposition}\label{pr:minimal_tame}
The local deformation problem $\sD^\mnm_{\bar\varrho}$ is formally smooth over $\Spf\sO$ of pure relative dimension $n^2$.
\end{proposition}

\begin{proof}
We follow the approach of \cite{Boo19}*{Proposition~5.6}, where a similar result for symplectic or orthogonal representations was proved.

Consider the morphism $\alpha\colon\sD^\mnm_{\bar\varrho}\to\Nilp_{\bar{X}}$ that sends a lifting $\varrho$ to $\log A$ if $\varrho^\natural(t)=A$. In the definition of $\Nilp_{\bar{X}}$, we may fix the nilpotent element $X_0\in\cN_n(\sO)$. Moreover, up to conjugation in $\GL_n(\sO)$, we may assume
\[
X_0=
\begin{pmatrix}
J_{n_1}&& \\
&\ddots&\\
&& J_{n_r}
\end{pmatrix},
\]
where $n=n_1+\cdots+n_r$. Let $Z_n(X_0)$ be the centralizer of $X_0$ in $\GL_{n,\sO}$, which is a closed subscheme of $\GL_{n,\sO}$. By \cite{Boo19}*{Remark~4.18}, $Z_n(X_0)$ is smooth over $\sO$. By \cite{Boo19}*{Lemma~3.11}, $\Nilp_{\bar{X}}$ is represented by a formal power series ring over $\sO$ in $n^2-\dim_\sO Z_n(X_0)$ variables, where $\dim_\sO Z_n(X_0)$ denotes the relative dimension of $Z_n(X_0)$ over $\sO$. Thus, it suffices to show that $\alpha$ is represented by a formal scheme formally smooth of pure relative dimension $\dim_\sO Z_n(X_0)$ over $\Nilp_{\bar{X}}$.

Take a lifting $\varrho$ of $\bar\varrho$ to an object $R$ of $\sC_\sO$. Then $\varrho(\phi_q)$ has the form $(B,-\mu(\phi_q))\fc$ with $B\in\GL_n(R)$ that reduces to $\bar{B}$ and satisfies
\begin{equation}\label{eq:conjugate_tame_1}
B \tp{X} B^{-1}=-qX
\end{equation}
by \eqref{eq:conjugate_tame}. For each given $X\in\Nilp_{\bar{X}}$, if there exists $B\in\GL_n(R)$ that reduces to $\bar{B}$ and satisfies \eqref{eq:conjugate_tame_1}, then the set of all elements $B$ form a torsor under the group
\[
\widehat{Z}_n(X)(R)\coloneqq\{g\in 1+\rM_n(\fm_R)\res g X g^{-1}=X\},
\]
which is isomorphic to the group of $R$-valued points of the formal completion of the group scheme $Z_n(X_0)$ along the unit section. Thus, to finish the proof, it suffices to show that the equation \eqref{eq:conjugate_tame_1} admits at least one solution for $B$ that reduces to $\bar{B}$.

Assume first $X=X_0$ in $\cN_n(R)$. Then
\[
B_0\coloneqq
\begin{pmatrix}
A_{n_1}&&
\\&\ddots&
\\&& A_{n_r}
\end{pmatrix}
\text{, where }
A_{n_i}\coloneqq
\begin{pmatrix}
&&& (-q)^{n_i-1}\\
&&\iddots &\\
& -q &&\\
1&&&
\end{pmatrix},
\]
is a solution to \eqref{eq:conjugate_tame_1}. In the general case, we write $X=C X_0C^{-1}$ for some $C\in\GL_n(R)$. Then $B\coloneqq CB_0\tp{C}$ satisfies the equation \eqref{eq:conjugate_tame_1}. Up to multiplying $C$ by an element in $Z_n(X_0)(R)$ from the right, we can make $B\in\GL_n(R)$ to reduce to $\bar{B}$. This finishes the proof of the proposition.
\end{proof}

Recall from Definition \ref{de:tangent_space_deformation} that $\rL(\sD^\mnm_{\bar\varrho})\subseteq\rH^1(\rT_q,\ad\bar\varrho)$ is tangent space of the local deformation problem $\sD^\mnm_{\bar\varrho}$.

\begin{corollary}\label{co:minimal_tame}
We have $\dim_k\rL(\sD^\mnm_{\bar\varrho})=\dim_k\rH^0(\rT_q,\ad\bar\varrho)$. 	
\end{corollary}

\begin{proof}
Suppose that $\sD^\mnm_{\bar\varrho}=\Spf\sfR^\mnm_{\bar\varrho}$. By \eqref{eq:local_deformation_dimension}, we have
\[
\dim_k\fm_{\sfR^\mnm_{\bar\varrho}}/(\lambda,\fm^2_{\sfR^\mnm_{\bar\varrho}})
=\dim_k\rL(\sD^\mnm_{\bar\varrho})+n^2-\dim_k\rH^0(\rT_q,\ad\bar\varrho).
\]
From this, the corollary follows immediately from Proposition \ref{pr:minimal_tame}.
\end{proof}

To end this subsection, we record the following lemma concerning decomposition of representations of the $q$-tame group, in which part (1) will be used later and part (2) is only for complement.

\begin{lem}\label{le:tame_decomposition}
Let $(\bar\rho,\bar{M})$ be an unramified representation of $\rT_q=t^{\dZ_\ell}\rtimes\phi_q^{\widehat\dZ}$ over $k$ of dimension $N$. Suppose that $\bar{M}$ admits a decomposition
\[
\bar{M}=\bar{M}_1\oplus\cdots\oplus\bar{M}_s
\]
stable under the action of $\bar\rho(\phi_q)$ such that the characteristic polynomials of $\bar\rho(\phi_q)$ on $M_i$ are mutually coprime for $1\leq i\leq s$. Let $(\rho,M)$ be a lifting of $(\bar\rho,\bar{M})$ to an object $R$ of $\sC_\sO$. Then
\begin{enumerate}
  \item There is a unique decomposition
 	  \[
      M=M_1\oplus\cdots\oplus M_s
      \]
      of free $R$-modules, such that $M_i$ is stable under the action of $\rho(\phi_q)$ and it is a lifting of $\bar{M}_i$ as a $\phi_q$-module.

  \item Write $\rho(t)=(\rho(t)_{i,j})$ with $\rho(t)_{i,j}\in\Hom_R(M_j,M_i)$. Suppose that $q$ is not an eigenvalue for the canonical action of $\phi_q$ on $\Hom_k(\bar{M}_j,\bar{M}_i)$ for all $i\neq j$. Then we have $\rho(t)_{i,j}=0$ for all $i\neq j$; in other words, the decomposition in (1) is stable under the whole group $\rT_q$.
\end{enumerate}
\end{lem}

\begin{proof}
Part (1) is elementary, which we leave to the readers as an exercise.

\if false

For (1), let $P(T)\in R[T]$ be the characteristic polynomial of $\rho(\phi_q)$ on $M$. By Hensel's lemma, $P(T)$ admits a unique decomposition
\[
P(T)=\prod_{i=1}^sP_i(T)
\]
such that $P_i(T)\modulo\fm_R$ is the characteristic polynomial of $\bar\rho(\phi_q)$ on $\bar{M}_i$ for $1\leq i\leq s$. We put $Q_i(T)=\prod_{j\neq i}P_j(T)$ for $1\leq i\leq s$. We put $M_i\coloneqq Q_i(\phi_q)M$ and $N_i\coloneqq P_i(\phi_q)M$. Then both $M_i$ and $N_i$ are both stable under $\phi_q$; $M_i$ is annihilated by $P_i(\phi_q)$; $N_i$ is annihilated by $Q_i(\phi_q)$; and $M_i\cap N_i=\{0\}$. Using Nakayama's lemma, it is easy to see that $(P_i(T),Q_i(T))=R[T]$. Thus, there exist polynomials $F_i,G_i\in R[T]$ such that $F_i P_i+G_iQ_i=1$ in $R[T]$. We then obtain
\[
F_i(\phi_q)N_i+G_i(\phi_q)M_i=M,
\]
hence $M=M_i\oplus N_i$. To complete the proof of (1), it suffices to show that $N_i=\bigoplus_{j\neq i}M_j$. By definition, it is clear that $M_j\subseteq N_i$ for every $j\neq i$, hence $\bigoplus_{j\neq i}M_j\subseteq N_i$. The inverse inclusion follows from the fact that the ideal of $R[T]$ generated by $Q_j$ for $j\neq i$ is same as the ideal generated by $P_i$.

\fi

For (2), we choose a basis of $M$ over $R$ adapted to the decomposition of $M$ in (1). We identify $\rho(t)$ and $\rho(\phi_q)$ with their matrices under this basis. We have $\rho(\phi_q)_{i,j}=0$ for $i\neq j$ since each $M_i$ is stable under $\rho(\phi_q)$. Let $J\subset R$ be the ideal generated by the coefficients of $\rho(t)_{i,j}$ for $i\neq j$. We have to show that $J=0$. By Nakayama's lemma, it suffices to show that $J=\fm_RJ$. As
\begin{align*}
\rho(t)^q=(1+(\rho(t)-1))^q=1+q(\rho(t)-1)+\sum_{a\geq 2}\binom{q}{a}(\rho(t)-1)^a,
\end{align*}
and $\rho(t)\equiv 1\modulo\fm_R$, we have
\[
(\rho(t)^q)_{i,j}\equiv q\rho(t)_{i,j} \mod\fm_RJ
\]
for $i\neq j$. The relation $\phi_qt=t^q\phi_q$ implies that
\[
\rho(\phi_q)_{i,i}\rho(t)_{i,j}=(\rho(t)^q)_{i,j}\rho(\phi_q)_{j,j}\equiv q\rho(t)_{i,j}\rho(\phi_q)_{j,j} \mod\fm_RJ.
\]
It follows that
\[
\rho(t)_{i,j}P_i(q\rho(\phi_q)_{j,j})\equiv P_i(\rho(\phi_q)_{i,i})\rho(t)_{i,j} \equiv 0 \mod\fm_RJ
\]
for $i\neq j$. By assumption, if $\bar\alpha$ is an eigenvalue of $\bar\rho(\phi_q)_{i,i}$, then $q^{-1}\bar\alpha$ is not an eigenvalue of $\bar\rho(\phi_q)_{j,j}$. It follows that $P_i(q\rho(\phi_q)_{j,j})$ is invertible, hence $\rho(t)_{i,j}\equiv 0\modulo\fm_RJ$.

The lemma is proved.
\end{proof}

\subsection{Minimally ramified deformations}
\label{ss:minimally_ramified}

In this subsection, we define and study the minimally ramified deformations at places coprime to the odd prime $\ell$. Thus, we take a nonarchimedean place $v$ of $F^+$ that is not above $\ell$.

Let $\rI_{F^+_v}\subseteq\Gamma_{F^+_v}$ be the inertia subgroup, and $\rP_{F^+_v}$ the maximal closed subgroup of $\rI_{F^+_v}$ of pro-order coprime to $\ell$. Put $\rT_v\coloneqq\Gamma_{F^+_v}/\rP_{F^+_v}$. Similarly, for every place $w$ of $F$ above $v$, we have $\Gamma_{F_w}$, $\rI_{F_w}$, $\rP_{F_w}$, and $\rT_w$. Finally, put $\rT_w^+\coloneqq\Gamma_{F^+_v}/\rP_{F_w}$.

\begin{remark}
The group $\rT_v$ is a $\|v\|$-tame group (Definition \ref{de:tame_group}).
\begin{itemize}
  \item When $w$ is split over $v$, we have $\rP_{F_w}=\rP_{F^+_v}$, $\rT_w^+=\rT_v$, and $\rT_w=\rT_w^+=\rT_v$.

  \item When $w$ is unramified over $v$, we have $\rP_{F_w}=\rP_{F^+_v}$, $\rT_w^+=\rT_v$, and that the subgroup $\rT_w$ of $\rT_v$ is a $\|v\|^2$-tame group.

  \item When $w$ is ramified over $v$, we have $\rP_{F^+_v}/\rP_{F_w}\simeq\dZ/2\dZ$, that the natural map $\rT_w\to\rT_v$ is an isomorphism, and a canonically split short exact sequence
      \[
      1\to \rP_{F^+_v}/\rP_{F_w} \to \rT_w^+ \to \rT_v \to 1.
      \]
\end{itemize}
\end{remark}

We first recall some facts about extensions of representations of $\rP_{F_w}$ from \cite{CHT08}. For an irreducible representation $\tau$ of $\rP_{F_w}$ with coefficients in $k$, we put
\[
\Gamma_\tau\coloneqq\{\sigma\in\Gamma_{F_w}\res\tau^\sigma\simeq\tau\},
\]
where we recall that $\tau^\sigma$ denotes the representation given by $\tau^\sigma(g)=\tau(\sigma g\sigma^{-1})$ for $g\in\rP_{F_w}$. Let $\rT_\tau$ be the image of $\Gamma_\tau$ in $\rT_w=\Gamma_{F_w}/\rP_{F_w}$. As $\rP_{F_w}$ is normal in $\Gamma_{F^+_v}$, we may similar define
\[
\Gamma_\tau^+\coloneqq\{\sigma\in\Gamma_{F^+_v}\res\tau^\sigma\simeq\tau\},
\]
and denote by $\rT_\tau^+$ its image in $\rT_w^+$.

\begin{lem}\label{le:wild_type}
We have the following properties for $\tau$:
\begin{enumerate}
  \item the dimension of $\tau$ is coprime to $\ell$; and $\tau$ has a unique deformation to a representation $\tilde\tau$ of $\rP_{F_w}$ over $\sO$;

  \item $\tilde\tau$ in (1) admits a unique extension to a representation of $\Gamma_\tau\cap\rI_{F_w}$ over $\sO$ whose determinant has order coprime to $\ell$;

  \item there exists an extension of $\tilde\tau$ in (2) to a representation of $\Gamma_\tau$ over $\sO$.
\end{enumerate}
\end{lem}

\begin{proof}
This is \cite{CHT08}*{Lemma~2.4.11}.
\end{proof}

Now we consider a pair $(\bar{r},\chi)$ from Notation \ref{no:deformation_pair} with $\tilde\Gamma=\Gamma_{F^+_v}$ and $\Gamma=\Gamma_{F^+_v}\cap\Gamma_F=\Gamma_{F_w}$. Our first goal is to define the notion of \emph{minimally ramified liftings} of $\bar{r}$ (Definition \ref{de:minimal_deformation}). Recall from Notation \ref{no:sg_extension} that we have the induced homomorphism $\bar{r}^\natural\colon\Gamma_{F_w}\to\GL_N(k)$, which satisfies that $\bar{r}^\natural\res_{\rP_{F_w}}$ is semisimple.

When $v$ is split in $F$, minimally ramified deformations have already been defined and studied in \cite{CHT08}*{2.4.4}. Thus, we now assume that $v$ is nonsplit in $F$, hence $w$ is the unique prime of $F$ above $v$.

\begin{assumption}\label{as:min_large}
The residue field $k$ of $E$ contains a subfield $k^\flat$ of degree two such that every irreducible summand of $\bar{r}^\natural\res_{\rP_{F_w}}\otimes_k\ol{k}$ is defined over $k^\flat$.
\end{assumption}

We first assume that $E$ satisfies Assumption \ref{as:min_large}. For an irreducible representation $\tau$ of $\rP_{F_w}$ with coefficients in $k$, we put
\[
M_\tau(\bar{r})\coloneqq\Hom_{k[\rP_{F_w}]}(\tau,\bar{r}^\natural).
\]
Then $\tau\otimes_k M_\tau(\bar{r})$ is canonically the $\tau$-isotypic component of $\bar{r}^\natural$. As $\tau$ extends to a representation of $\Gamma_\tau$, the $k$-vector space $M_\tau(\bar{r})$ is equipped with a natural action by $\rT_\tau$; and $\tau\otimes_k M_\tau(\bar{r})$ is equipped with a natural action by $\Gamma_\tau$.

We denote by $\fT=\fT(\bar{r})$ the set of isomorphism classes of irreducible representations $\tau$ of $\rP_{F_w}$ such that $M_\tau(\bar{r})\neq0$. Then $\Gamma_{F_w}$ acts on $\fT$ by conjugation, whose orbits we denote by $\fT/\Gamma_{F_w}$. For $\tau\in\fT$, we write $[\tau]$ for its orbit in $\fT/\Gamma_{F_w}$.

Choose an element $\gamma\in\Gamma_{F^+_v}\setminus\Gamma_{F_w}$. By Lemma \ref{le:representation_selfdual}, the homomorphism $\bar{r}$ is determined by an element $\bar\Psi\in\GL_N(k)$ satisfying
\[
\bar{r}^{\natural,\gamma}=\bar\Psi\circ\chi\bar{r}^{\natural,\vee}\circ\bar\Psi^{-1},\quad
\bar\Psi\tp{\bar\Psi}^{-1}=-\chi(\gamma)^{-1}\bar{r}^\natural(\gamma^2).
\]

In what follows, we will adopt the following simplified notation: for a representation $\tau$ of a subgroup of $\Gamma_{F^+_v}$, we write $\tau^*$ for $\chi\tau^\vee$. Due to the existence of $\bar\Psi$, we know that if $\tau\in\fT$, then $\tau^{\gamma,*}\in\fT$ as well. As $\gamma^2\in\Gamma_{F_w}$, the assignment $\tau\mapsto\tau^{\gamma,*}$ induces an involution on the set $\fT/\Gamma_{F_w}$, which does not depend on the choice of $\gamma$.

\begin{construction}\label{cs:minimal_extension}
We now would like to construct a $\Gamma_{F_w}$-stable partition $\fT=\fT_1\sqcup\fT_2\sqcup\fT_3$. For each subset $\fT_i$, we will specify, for every $\tau\in\fT_i$, an extension of $\tilde\tau$ in Lemma \ref{le:wild_type}(2) to a representation of $\Gamma_\tau$ with coefficients in $\sO$ in a compatible way, specified below.

We start from the following observation. Suppose that $[\tau]=[\tau^{\gamma,*}]$ in $\fT/\Gamma_{F_w}$. Then there exists an element $h\in\Gamma_{F_w}$, unique up to left multiplication by an element in $\Gamma_\tau$, such that $\tau^{\gamma,*}\simeq\tau^{h^{-1}}$, or equivalently, $\tau^{h\gamma}\simeq\tau^*$. Then we have $(h\gamma)^2\in\Gamma_\tau$ but $h\gamma\not\in\Gamma_\tau$. Denote by $\tilde\Gamma_\tau$ the subgroup of $\Gamma_{F^+_v}$ generated by $\Gamma_\tau$ and $h\gamma$, which contains $\Gamma_\tau$ as a subgroup of index two. Let $\tilde\rT_\tau$ be the image of $\tilde\Gamma_\tau$ in $\rT^+_w$, which contains $\rT_\tau$ as a subgroup of index two.

\begin{enumerate}
  \item We define $\fT_1$ to be the subset of $\fT$ consisting of $\tau$ such that $[\tau]\neq[\tau^{\gamma,*}]$. We choose a subset $\fT_1^\heartsuit\subseteq\fT_1$ such that $\{\tau,\tau^{\gamma,*}\res\tau\in\fT_1^\heartsuit\}$ is a set of representatives for the $\Gamma_\tau$-action on $\fT_1$. For each element $\tau\in\fT_1^\heartsuit$, we choose an extension of $\tilde\tau$ in Lemma \ref{le:wild_type}(2) to a representation of $\Gamma_\tau$ with coefficients in $\sO$, which we still denote by $\tilde\tau$. For a general element $\tau\in\fT_1$, there are two cases. If $\tau\simeq\tau_1^h$ for (unique) $\tau_1\in\fT_1^\heartsuit$ and some $h\in\Gamma_{F_w}$, then we choose $\tilde\tau$ to be $\tilde\tau_1^h$, as the extension to $\Gamma_\tau=h^{-1}\Gamma_{\tau_1}h$. If $\tau\simeq(\tau_1^h)^{\gamma,*}$ for (unique) $\tau_1\in\fT_1^\heartsuit$ and some $h\in\Gamma_{F_w}$, then we choose $\tilde\tau$ to be $(\tilde\tau_1^h)^{\gamma,*}$, as the extension to $\Gamma_\tau=\gamma^{-1}h^{-1}\Gamma_{\tau_1}h\gamma$.

  \item We define $\fT_2$ to be the subset of $\fT$ consisting of $\tau$ such that $[\tau]=[\tau^{\gamma,*}]$, and that the images of $\Gamma_\tau$ and $\tilde\Gamma_\tau$ in $\Gamma_{F^+_v}/\rI_{F^+_v}$ are different. We choose a subset $\fT_2^\heartsuit\subseteq\fT_2$ of representatives for the $\Gamma_\tau$-action on $\fT_2$. For each element $\tau\in\fT_2^\heartsuit$, we choose an extension $\tilde\tau$ from Lemma \ref{le:minimal_extension}(1) below to a representation of $\Gamma_\tau$ with coefficients in $\sO$. For $\tau\in\fT_2$ in general, we have $\tau\simeq\tau_2^h$ for (unique) $\tau_2\in\fT_2^\heartsuit$ and some $h\in\Gamma_{F_w}$; and we choose $\tilde\tau$ to be $\tilde\tau_2^h$, as the extension to $\Gamma_\tau=h^{-1}\Gamma_{\tau_2}h$.

  \item We define $\fT_3$ to be the subset of $\fT$ consisting of $\tau$ such that $[\tau]=[\tau^{\gamma,*}]$, and that the images of $\Gamma_\tau$ and $\tilde\Gamma_\tau$ in $\Gamma_{F^+_v}/\rI_{F^+_v}$ are the same. We choose a subset $\fT_3^\heartsuit\subseteq\fT_3$ of representatives for the $\Gamma_\tau$-action on $\fT_3$. For each element $\tau\in\fT_3^\heartsuit$, we choose an extension $\tilde\tau$ from Lemma \ref{le:minimal_extension}(2) below to a representation of $\Gamma_\tau$ with coefficients in $\sO$. For $\tau\in\fT_3$ in general, we have $\tau\simeq\tau_3^h$ for (unique) $\tau_3\in\fT_3^\heartsuit$ and some $h\in\Gamma_{F_w}$; and we choose $\tilde\tau$ to be $\tilde\tau_3^h$, as the extension to $\Gamma_\tau=h^{-1}\Gamma_{\tau_3}h$.
\end{enumerate}

In addition, we put $\fT^\heartsuit\coloneqq\fT^\heartsuit_1\sqcup\fT^\heartsuit_2\sqcup\fT^\heartsuit_3$.
\end{construction}

\begin{remark}
The partition $\fT=\fT_1\sqcup\fT_2\sqcup\fT_3$ does not depend on the choice of $\gamma$. Moreover, if $\fT_3$ is nonempty, then $w$ is ramified over $v$.
\end{remark}

\begin{lem}\label{le:minimal_extension}
Let $\tau\in\fT$ be an element of dimension $d$.
\begin{enumerate}
  \item If $\tau\in\fT_2$, then the representation $\tilde\tau$ in Lemma \ref{le:wild_type}(2) extends to a representation of $\Gamma_\tau$ with coefficients in $\sO$ such that $\tilde\tau^{\gamma'}\simeq\tilde\tau^*$ still holds for every $\gamma'\in\tilde\Gamma_\tau\setminus\Gamma_\tau$.

  \item If $\tau\in\fT_3$, then the representation $\tilde\tau$ in Lemma \ref{le:wild_type}(2) extends to a representation of $\Gamma_\tau$ with coefficients in $\sO$ such that $\tilde\tau^{\gamma'}\simeq\tilde\tau^*$ still holds for every $\gamma'\in\tilde\Gamma_\tau\setminus\Gamma_\tau$.
\end{enumerate}
\end{lem}

\begin{proof}
We fix a splitting $\Gamma_{F^+_v}\simeq\rP_{F^+_v}\rtimes\rT_v$ and an isomorphism $\rT_v\simeq\rT_q=t^{\dZ_\ell}\rtimes\phi_q^{\widehat\dZ}$ with the $q$-tame group (Definition \ref{de:tame_group}) where $q=\|v\|$. Then we have the induced splitting $\Gamma_\tau\simeq\rP_{F_w}\rtimes\rT_\tau$, where $\rT_\tau=t_\tau^{\dZ_\ell}\rtimes\phi_\tau^{\widehat\dZ}$ is a subgroup of $\rT_q$, with $t_\tau=t^{\ell^a}$ and $\phi_\tau=\phi_q^b$ for unique integers $a\geq 0$ and $b>0$. To extend $\tilde\tau$ in Lemma \ref{le:wild_type}(2) to a representation of $\Gamma_\tau$, it suffices to specify $\tilde\tau(\phi_\tau)$.

For (1), there are two cases.

First, we suppose that $w$ is unramified over $v$. Then $b$ is even; and $\tilde\rT_\tau$ is the image of $\tilde\Gamma_\tau$ in $\rT_v$. Then $\tilde\rT_\tau$ is generated by $\rT_\tau$ and an element $\gamma'\in\rT_v$ of the form $(\tilde{t}_\tau,\phi_q^{b/2})$ such that ${\gamma'}^2=(\tilde{t}_\tau^{q^{b/2}+1},\phi_q^b)$ lies in $\Gamma_\tau$. As $[\tau]=[\tau^{\gamma,*}]$, we have $\tau^{\gamma'}\simeq\tau^*$. We choose a basis of $\tau$, hence regard $\tau$ as a homomorphism $\tau\colon\rP_{F_w}\to\GL_d(k)$. By Lemma \ref{le:wild_type}(1,2), we have a continuous homomorphism $\tilde\tau\colon\Gamma_\tau\cap\rI_{F_w}\to\GL_d(\sO)$ such that $\tilde\tau\res_{\rP_{F_w}}$ is a lifting of $\tau$, unique up to conjugation in $1+\rM_d(\lambda)$. In particular, there is an element $Bp\in\GL_d(\sO)$, unique up to scalar in $\sO^\times$, such that $\tilde\tau^{\gamma'}(g)=B\tilde\tau^*(g)B^{-1}$ for every $g\in\Gamma_\tau\cap\rI_{F_w}$. Since $\phi_\tau=\tilde{t}_\tau^{-q^{b/2}-1}{\gamma'}^2$, we have
\begin{align*}
\tilde\tau(\phi_\tau g\phi_\tau^{-1})
&=\tilde\tau(\tilde{t}_\tau^{-q^{b/2}-1}{\gamma'}^2g{\gamma'}^{-2}\tilde{t}_\tau^{q^{b/2}+1})
=\tilde\tau(\tilde{t}_\tau^{-q^{b/2}-1})\tilde\tau({\gamma'}^2g{\gamma'}^{-2})\tilde\tau(\tilde{t}_\tau^{q^{b/2}+1}) \\
&=\tilde\tau(\tilde{t}_\tau^{-q^{b/2}-1})B\tilde\tau^*({\gamma'}^{-1}g\gamma')B^{-1}\tilde\tau(\tilde{t}_\tau^{q^{b/2}+1}) \\
&=\tilde\tau(\tilde{t}_\tau^{-q^{b/2}-1})(B\tp{B}^{-1})\tilde\tau(g)(\tp{B}B^{-1})\tilde\tau(\tilde{t}_\tau^{q^{b/2}+1})
\end{align*}
for every $g\in\Gamma_\tau\cap\rI_{F_w}$. We put $\tilde\tau(\phi_\tau)\coloneqq-\chi(\phi_q^{b/2})\tilde\tau(\tilde{t}_\tau^{-q^{b/2}-1})(B\tp{B}^{-1})$. Then we obtain the desired extension as in (1).

Second, we suppose that $w$ is ramified over $v$. By the definition of $\fT_2$, the image of $\tilde\Gamma_\tau$ in $\Gamma_{F^+_v}/\rI_{F^+_v}$ contains $\phi_\tau^{\widehat\dZ}$ as a subgroup of index two. Thus, there exists an element $\gamma'\in\tilde\Gamma_\tau\setminus\Gamma_\tau$ such that $\gamma^2=h\phi_\tau$ for some $h\in\Gamma_\tau\cap\rI_{F_w}$. The remaining argument is same to the above case.

For (2), by the definition of $\fT_3$, the image of $\tilde\Gamma_\tau$ in $\Gamma_{F^+_v}/\rI_{F^+_v}$ coincides with $\phi_\tau^{\widehat\dZ}$. In particular, we can find an element $\gamma'\in\tilde\Gamma_\tau\setminus\Gamma_\tau$ contained in $\rI_{F^+_v}\setminus\rI_{F_w}$. By Lemma \ref{le:wild_type}(1,2), we have a continuous homomorphism $\tilde\tau\colon\Gamma_\tau\cap\rI_{F_w}\to\GL_d(\sO)$ such that $\tilde\tau\res_{\rP_{F_w}}$ is a lifting of $\tau$, unique up to conjugation in $1+\rM_d(\lambda)$. As we have $\tau^{\gamma'}\simeq\tau^*$ and $\tau^{\phi_\tau}\simeq\tau$, there are elements $A,B\in\GL_d(\sO)$ such that
\begin{align}\label{eq:minimal_extension1}
\tilde\tau^{\gamma'}(g)=A\tilde\tau^*(g)A^{-1},
\end{align}
\begin{align}\label{eq:minimal_extension2}
\tilde\tau^{\phi_\tau}(g)=B\tilde\tau(g)B^{-1},
\end{align}
for every $g\in\Gamma_\tau\cap\rI_{F_w}$. It follows from \eqref{eq:minimal_extension1} that the desired element $\tilde\tau(\phi_\tau)\in\GL_d(\sO)$ has to satisfy the equation
\begin{align}\label{eq:minimal_extension3}
\chi(\phi_\tau)A\tp{\tilde\tau}(\phi_\tau)^{-1}A^{-1}=\tilde\tau(\gamma'\phi_\tau{\gamma'}^{-1})
=\tilde\tau(\gamma'\phi_\tau{\gamma'}^{-1}\phi_\tau^{-1})\tilde\tau(\phi_\tau),
\end{align}
where we note that $\gamma'\phi_\tau{\gamma'}^{-1}\phi_\tau^{-1}\in\Gamma_\tau\cap\rI_{F_w}$. However, by \eqref{eq:minimal_extension2}, we have
\[
\tilde\tau^{\gamma'\phi_\tau{\gamma'}^{-1}}(g)
=(\tilde\tau(\gamma'\phi_\tau{\gamma'}^{-1}\phi_\tau^{-1})B)\tilde\tau(g)(\tilde\tau(\gamma'\phi_\tau{\gamma'}^{-1}\phi_\tau^{-1})B)^{-1}
\]
for every $g\in\Gamma_\tau\cap\rI_{F_w}$. On the other hand, by \eqref{eq:minimal_extension1} and \eqref{eq:minimal_extension2}, we have
\[
\tilde\tau^{{\gamma'}^{-1}\phi_\tau\gamma'}(g)
=(A\tp{B}^{-1}A^{-1})\tilde\tau(g)(A\tp{B}^{-1}A^{-1})^{-1}
\]
for every $g\in\Gamma_\tau\cap\rI_{F_w}$. Since $\tau$ is absolutely irreducible, it follows that there exists $\beta\in\sO^\times$ such that
\[
A\tp{B}^{-1}A^{-1}=\beta\cdot\tilde\tau(\gamma'\phi_\tau{\gamma'}^{-1}\phi_\tau^{-1})B.
\]
Take an element $\alpha\in\sO^\times$ such that $\alpha^2=\beta\chi(\phi_\tau)$, which is possible by Assumption \ref{as:min_large}. Then it is clear that $\tilde\tau(\phi_\tau)=\alpha B\in\GL_d(\sO)$ is a solution to \eqref{eq:minimal_extension3}.

The lemma is proved.
\end{proof}

Using Construction \ref{cs:minimal_extension}, we now discuss the structure of liftings of $\bar{r}$. Let $r\colon\Gamma_{F^+_v}\to\sG_N(R)$ be a lifting of $\bar{r}$ to an object $R$ of $\sC_{\sO}$. By Lemma \ref{le:representation_selfdual}, to give such a lifting $r$ is equivalent to giving an element $\Psi\in\GL_N(R)$ that reduces to $\bar\Psi$ and satisfies
\[
r^{\natural,\gamma}=\Psi\circ\chi r^{\natural,\vee}\circ\Psi^{-1},\quad
\Psi\tp{\Psi}^{-1}=-\chi(\gamma)^{-1}r^\natural(\gamma^2).
\]
For every $\tau\in\fT$, put
\[
M_\tau(r)\coloneqq\Hom_{R[\rP_{F_w}]}(\tilde\tau\otimes_{\sO}R,r^\natural),
\]
which is a finite free $R$-module equipped with the induced continuous action by $\rT_\tau$. Denote by $m_\tau\geq 1$ the rank of $M_\tau(r)$. Let $\underline\tau\in\fT$ be the unique element such that $\tau^\gamma\simeq\underline\tau^*$. Choose an isomorphism $\iota_\tau\colon\tau^\gamma\xrightarrow{\sim}\underline\tau^*$, which, by construction \ref{cs:minimal_extension}, lifts to an isomorphism $\iota_{\tilde\tau}\colon\tilde\tau^\gamma\xrightarrow{\sim}\tilde{\underline\tau}^*$ of representations of $\Gamma_{\underline\tau}$. Then we have isomorphisms
\[
M_\tau(r)^\gamma\xrightarrow{\sim}\Hom_{R[\rP_{F_w}]}(\tilde\tau^\gamma\otimes_{\sO}R,r^{\natural,\gamma})
\xrightarrow{\sim}\Hom_{R[\rP_{F_w}]}(\tilde{\underline\tau}^*\otimes_{\sO}R,r^{\natural,*})
\xrightarrow{\sim}\Hom_{R[\rP_{F_w}]}(r^\natural,\tilde{\underline\tau}\otimes_{\sO}R),
\]
where the second isomorphism is induced by $\iota_{\tilde\tau}$ and $\Psi$. As $\underline\tau$ is absolutely irreducible, we obtain a perfect $R$-bilinear pairing
\[
M_\tau(r)^\gamma\times M_{\underline\tau}(r)\to\End_{R[\rP_{F_w}]}(\tilde{\underline\tau}\otimes_{\sO}R)=R,
\]
which induces an isomorphism
\[
\theta_{\tilde\tau,r}\colon M_\tau(r)^\gamma\xrightarrow{\sim}M_{\underline\tau}(r)^\vee\coloneqq\Hom_R(M_{\underline\tau}(r),R)
\]
of $R[\rT_{\underline\tau}]$-modules. In particular, we have
\begin{align}\label{eq:minimal_prestructure}
\resizebox{\hsize}{!}{
\xymatrix{
r^\natural\simeq\(\bigoplus_{\tau\in\fT^\heartsuit_1}\(\Ind^{\Gamma_{F_w}}_{\Gamma_\tau}(\tilde\tau\otimes_{\sO}M_\tau(r))\oplus
\Ind^{\Gamma_{F_w}}_{\Gamma_{\tau^\gamma}}(\tilde\tau^{\gamma,*}\otimes_{\sO}M_\tau(r)^{\gamma,\vee})\)\)\bigoplus
\(\bigoplus_{\tau\in\fT^\heartsuit_2\sqcup\fT^\heartsuit_3}\Ind^{\Gamma_{F_w}}_{\Gamma_\tau}(\tilde\tau\otimes_{\sO}M_\tau(r))\)
}
}
\end{align}
as representations of $\Gamma_{F_w}$.

Now for every $\tau$, we fix an isomorphism $b_\tau\colon M_\tau(\bar{r})\xrightarrow{\sim}k^{\oplus m_\tau}$ of $k$-vector spaces, and let $\rT_\tau\to\GL_{m_\tau}(k)$ be the induced homomorphism. There are two cases.
\begin{enumerate}[label=(\alph*)]
  \item Suppose that $\tau\in\fT_1$. Then $M_{\underline\tau}(r)$ is determined by $M_\tau(r)$. If we choose an isomorphism $M_\tau(r)\simeq R^{\oplus m_\tau}$ of $R$-modules that reduces to $b_\tau$, then we obtain a continuous homomorphism
      \[
      \varrho_\tau\colon\rT_\tau\to\GL_{m_\tau}(R)
      \]
      that reduces to $\rT_\tau\to\GL_{m_\tau}(k)$.

  \item Suppose that $\tau\in\fT_2\sqcup\fT_3$. Let $h$ be element from Construction \ref{cs:minimal_extension}. Then $\theta_{\tilde\tau,r}$ induces an isomorphism $M_\tau(r)^{h\gamma}\xrightarrow{\sim}M_\tau(\tau)^\vee$ of $R[\rT_\tau]$-modules. Applying Lemma \ref{le:representation_selfdual}(3) to $\rT_\tau\to\GL_{m_\tau}(k)$, we obtain a homomorphism
      \[
      \bar\varrho_\tau\colon\tilde\rT_\tau\to\sG_{m_\tau}(k)
      \]
      satisfying $\bar\varrho_\tau^{-1}(\GL_{m_\tau}(k)\times(k)^\times)=\rT_\tau$ and $\nu\circ\bar\varrho_\tau=\eta_v^{\mu_\tau}$ for some $\mu_\tau\in\dZ/2\dZ$ determined by $\tilde\tau$.\footnote{In fact, when $\tau\in\fT_2$, one can always modify $\tilde\tau$ to make $\mu_\tau=0$; but when $\tau\in\fT_3$, $\mu_\tau$ is determined by $\tau$.} In general, if we choose an isomorphism $M_\tau(r)\simeq R^{\oplus m_\tau}$ of $R$-modules that reduces $b_\tau$, then we obtain a continuous homomorphism
      \[
      \varrho_\tau\colon\tilde\rT_\tau\to\sG_{m_\tau}(R)
      \]
      that reduces to $\bar\varrho_\tau$ and satisfies $\nu\circ\rho_\tau=\eta_v^{\mu_\tau}$.
\end{enumerate}

The following proposition is the counterpart of \cite{CHT08}*{Corollary~2.4.13} when $v$ is nonsplit in $F$.

\begin{proposition}\label{pr:minimal_structure}
Suppose that $E$ satisfies Assumption \ref{as:min_large}. We keep the choices of $\gamma\in\Gamma_{F^+_v}\setminus\Gamma_{F_w}$, those in Construction \ref{cs:minimal_extension}, $\iota_\tau$, and $b_\tau$. For every object $R$ of $\sC_{\sO}$, the assignment
\[
r\mapsto (\varrho_\tau)_{\tau\in\fT^\heartsuit}
\]
establishes a bijection between deformations of $\bar{r}$ to $R$ and equivalence classes of tuples $(\varrho_\tau)_{\tau\in\fT^\heartsuit}$ where
\begin{enumerate}[label=(\alph*)]
  \item for $\tau\in\fT^\heartsuit_1$, $\varrho_\tau\colon\rT_\tau\to\GL_{m_\tau}(R)$ is a continuous homomorphism that reduces to $\bar\varrho_\tau$;

  \item for $\tau\in\fT^\heartsuit_2\sqcup\fT^\heartsuit_3$, $\varrho_\tau\colon\tilde\rT_\tau\to\sG_{m_\tau}(R)$ is a continuous homomorphism that reduces to $\bar\varrho_\tau$ and satisfies $\nu\circ\rho_\tau=\eta_v^{\mu_\tau}$.
\end{enumerate}
Here, two tuples $(\varrho_\tau)_{\tau\in\fT^\heartsuit}$ and $(\varrho'_\tau)_{\tau\in\fT^\heartsuit}$ are said to be equivalent if $\varrho_\tau$ and $\varrho'_\tau$ are conjugate by elements in $1+\rM_{m_\tau}(\fm_R)$ for every $\tau\in\fT^\heartsuit$.
\end{proposition}

\begin{proof}
We now attach to every tuple $(\varrho_\tau)_{\tau\in\fT^\heartsuit}$ as in the statement a lifting $r$ explicitly. Denote by $M_\tau$ the $R[\rT_\tau]$-module corresponding to $\varrho_\tau$. Consider
\[
M\coloneqq\(\bigoplus_{\tau\in\fT^\heartsuit_1}\(\Ind^{\Gamma_{F_w}}_{\Gamma_\tau}(\tilde\tau\otimes_{\sO}M_\tau)\oplus
\Ind^{\Gamma_{F_w}}_{\Gamma_{\tau^\gamma}}(\tilde\tau^{\gamma,*}\otimes_{\sO}M_\tau^{\gamma,\vee})\)\)\bigoplus
\(\bigoplus_{\tau\in\fT^\heartsuit_2\sqcup\fT^\heartsuit_3}\Ind^{\Gamma_{F_w}}_{\Gamma_\tau}(\tilde\tau\otimes_{\sO}M_\tau)\),
\]
which is a free $R$-module of rank $N$, equipped with a continuous action by $\Gamma_{F_w}$. Moreover, we have $M\otimes_RR/\fm_R\simeq\bar{r}^\natural$ as representations of $\Gamma_{F_w}$ by \eqref{eq:minimal_prestructure}. Thus, we may fix an isomorphism $M\simeq R^{\otimes N}$ such that the induced continuous homomorphism $\rho=\rho_M\colon\Gamma_{F_w}\to\GL_N(R)$ reduces to $\bar{r}^\natural$. Thus, by Lemma \ref{le:representation_selfdual}, to construct the desired lifting $r$ from $\rho$, it amounts to finding an element $\Psi\in\GL_N(R)$ satisfying
\begin{align}\label{eq:minimal_structure}
\rho^\gamma=\Psi\circ\chi\rho^\vee\circ\Psi^{-1},\quad
\Psi\tp\Psi^{-1}=-\chi(\gamma)^{-1}\rho(\gamma^2).
\end{align}
We will construct $\Psi$ as a direct sum of $\Psi_\tau$ for $\tau\in\fT^\heartsuit$.

For $\tau\in\fT^\heartsuit_1$, we note that $\tilde\tau(\gamma^{-2})\otimes\varrho_\tau(\gamma^{-2})$ induces an isomorphism
\[
\Ind^{\gamma^{-2}}\colon\Ind^{\Gamma_{F_w}}_{\Gamma_{\tau^{\gamma^2}}}(\tilde\tau^{\gamma^2}\otimes_{\sO}M_\tau^{\gamma^2})
\simeq\(\Ind^{\Gamma_{F_w}}_{\Gamma_{\tau^\gamma}}(\tilde\tau^{\gamma,*}\otimes_{\sO}M_\tau^{\gamma,\vee})\)^{\gamma,*}
\xrightarrow{\sim}\Ind^{\Gamma_{F_w}}_{\Gamma_\tau}(\tilde\tau\otimes_{\sO}M_\tau).
\]
Thus, we obtain an isomorphism
\begin{align}\label{eq:minimal_structure_1}
\(\Ind^{\Gamma_{F_w}}_{\Gamma_\tau}(\tilde\tau\otimes_{\sO}M_\tau)\oplus
\Ind^{\Gamma_{F_w}}_{\Gamma_{\tau^\gamma}}(\tilde\tau^{\gamma,*}\otimes_{\sO}M_\tau^{\gamma,\vee})\)^{\gamma,*}
\xrightarrow{\sim}\Ind^{\Gamma_{F_w}}_{\Gamma_\tau}(\tilde\tau\otimes_{\sO}M_\tau)\oplus
\Ind^{\Gamma_{F_w}}_{\Gamma_{\tau^\gamma}}(\tilde\tau^{\gamma,*}\otimes_{\sO}M_\tau^{\gamma,\vee})
\end{align}
as the composition of the canonical isomorphism
\[
\resizebox{\hsize}{!}{
\xymatrix{
\(\Ind^{\Gamma_{F_w}}_{\Gamma_\tau}(\tilde\tau\otimes_{\sO}M_\tau)\oplus
\Ind^{\Gamma_{F_w}}_{\Gamma_{\tau^\gamma}}(\tilde\tau^{\gamma,*}\otimes_{\sO}M_\tau^{\gamma,\vee})\)^{\gamma,*}
\xrightarrow{\sim}\Ind^{\Gamma_{F_w}}_{\Gamma_{\tau^\gamma}}(\tilde\tau^{\gamma,*}\otimes_{\sO}M_\tau^{\gamma,\vee})
\oplus\(\Ind^{\Gamma_{F_w}}_{\Gamma_{\tau^\gamma}}(\tilde\tau^{\gamma,*}\otimes_{\sO}M_\tau^{\gamma,\vee})\)^{\gamma,*},
}
}
\]
and the isomorphism
\[
\resizebox{\hsize}{!}{
\xymatrix{
\Ind^{\Gamma_{F_w}}_{\Gamma_{\tau^\gamma}}(\tilde\tau^{\gamma,*}\otimes_{\sO}M_\tau^{\gamma,\vee})
\oplus\(\Ind^{\Gamma_{F_w}}_{\Gamma_{\tau^\gamma}}(\tilde\tau^{\gamma,*}\otimes_{\sO}M_\tau^{\gamma,\vee})\)^{\gamma,*}
\xrightarrow{\sim}
\Ind^{\Gamma_{F_w}}_{\Gamma_\tau}(\tilde\tau\otimes_{\sO}M_\tau)\oplus
\Ind^{\Gamma_{F_w}}_{\Gamma_{\tau^\gamma}}(\tilde\tau^{\gamma,*}\otimes_{\sO}M_\tau^{\gamma,\vee})
}
}
\]
given by the matrix
\[
\begin{pmatrix} 0 & -\chi(\gamma)\Ind^{\gamma^{-2}} \\ 1 & 0 \end{pmatrix}.
\]
We now let $\Psi_\tau$ be the matrix representing the isomorphism
\[
\(\Ind^{\Gamma_{F_w}}_{\Gamma_\tau}(\tilde\tau\otimes_{\sO}M_\tau)\oplus
\Ind^{\Gamma_{F_w}}_{\Gamma_{\tau^\gamma}}(\tilde\tau^{\gamma,*}\otimes_{\sO}M_\tau^{\gamma,\vee})\)^*
\xrightarrow{\sim}\(\Ind^{\Gamma_{F_w}}_{\Gamma_\tau}(\tilde\tau\otimes_{\sO}M_\tau)\oplus
\Ind^{\Gamma_{F_w}}_{\Gamma_{\tau^\gamma}}(\tilde\tau^{\gamma,*}\otimes_{\sO}M_\tau^{\gamma,\vee})\)^\gamma
\]
induced from \eqref{eq:minimal_structure_1} by duality.

For $\tau\in\fT^\heartsuit_2\sqcup\fT^\heartsuit_3$, let $h$ be the element in Construction \ref{cs:minimal_extension}. Put $\gamma'\coloneqq h\gamma$, which is an element in $\tilde\Gamma_\tau\setminus\Gamma_\tau$. The homomorphism $\varrho_\tau\colon\tilde\rT_\tau\to\sG_{m_\tau}(R)$ induces an isomorphism $M_\tau^{\gamma'}\xrightarrow{\sim}M_\tau^\vee$ by Lemma \ref{le:representation_selfdual}(1), which induces an isomorphism $M_\tau^{\gamma,\vee}\xrightarrow{\sim}M_\tau^{h^{-1}}$. On the other hand, by Lemma \ref{le:minimal_extension}, we have an isomorphism $\tilde\tau^{\gamma,*}\simeq\tilde\tau^{h^{-1}}$. Thus, we obtain an isomorphism
\begin{align}\label{eq:minimal_structure_2}
\(\Ind^{\Gamma_{F_w}}_{\Gamma_\tau}(\tilde\tau\otimes_{\sO}M_\tau)\)^{\gamma,*}
\xrightarrow{\sim}\Ind^{\Gamma_{F_w}}_{\Gamma_\tau}(\tilde\tau\otimes_{\sO}M_\tau)
\end{align}
as the composition of the canonical isomorphism
\[
\(\Ind^{\Gamma_{F_w}}_{\Gamma_\tau}(\tilde\tau\otimes_{\sO}M_\tau)\)^{\gamma,*}\xrightarrow{\sim}
\Ind^{\Gamma_{F_w}}_{\Gamma_{\tau^\gamma}}(\tilde\tau^{\gamma,*}\otimes_{\sO}M_\tau^{\gamma,\vee}),
\]
the isomorphism
\[
\Ind^{\Gamma_{F_w}}_{\Gamma_{\tau^\gamma}}(\tilde\tau^{\gamma,*}\otimes_{\sO}M_\tau^{\gamma,\vee})\xrightarrow{\sim}
\Ind^{\Gamma_{F_w}}_{\Gamma_{\tau^{h^{-1}}}}(\tilde\tau^{h^{-1}}\otimes_{\sO}M_\tau^{h^{-1}})
\]
specified above, and the isomorphism
\[
\Ind^{\Gamma_{F_w}}_{\Gamma_{\tau^{h^{-1}}}}(\tilde\tau^{h^{-1}}\otimes_{\sO}M_\tau^{h^{-1}})\xrightarrow{\sim}
\Ind^{\Gamma_{F_w}}_{\Gamma_\tau}(\tilde\tau\otimes_{\sO}M_\tau)
\]
given by the action of $h^{-1}$. We now let $\Psi_\tau$ be the matrix representing the isomorphism
\[
\(\Ind^{\Gamma_{F_w}}_{\Gamma_\tau}(\tilde\tau\otimes_{\sO}M_\tau)\)^*\xrightarrow{\sim}
\(\Ind^{\Gamma_{F_w}}_{\Gamma_\tau}(\tilde\tau\otimes_{\sO}M_\tau)\)^\gamma
\]
induced from \eqref{eq:minimal_structure_2} by duality.

Finally, we put $\Psi\coloneqq\bigoplus_{\tau\in\fT^\heartsuit}\Psi_\tau$. Then \eqref{eq:minimal_structure} follows by construction. In other words, we have assigned a lifting $r$ from the tuple $(\varrho_\tau)_{\tau\in\fT^\heartsuit}$. It is straightforward to check that such assignment is inverse to the assignment in the proposition. The proposition follows.
\end{proof}

From now on, $E$ will not necessarily be subject to Assumption \ref{as:min_large}. Using Proposition \ref{pr:minimal_structure}, we can define minimally ramified liftings of $\bar{r}$ for $E$ in general.

\begin{definition}\label{de:minimal_deformation}
Suppose that $l\geq N$. Choose a finite unramified extension $E_\dag$ of $E$ contained in $\ol\dQ_\ell$ that satisfies Assumption \ref{as:min_large}, with the ring of integers $\sO_\dag$ and the residue field $k_\dag$. We also keep the choices of $\gamma\in\Gamma_{F^+_v}\setminus\Gamma_{F_w}$, those in Construction \ref{cs:minimal_extension}, $\iota_\tau$, and $b_\tau$, as in Proposition \ref{pr:minimal_structure} (with respect to $E$).

We say that a lifting $r$ of $\bar{r}$ to some object $R$ of $\sC_\sO$ is \emph{minimally ramified} if in the tuple $(\varrho_\tau)_{\tau\in\fT^\heartsuit}$ corresponding to the lifting $r\otimes_\sO\sO_\dag$ from Proposition \ref{pr:minimal_structure}, every homomorphism $\varrho_\tau$ is a minimally ramified lifting of $\bar\varrho_\tau$ in the following sense.
\begin{enumerate}
  \item For $\tau\in\fT^\heartsuit_1$, minimally ramified liftings of $\bar\varrho_\tau$ is defined in Definition \ref{de:minimal_tame} (which is equivalent to \cite{CHT08}*{Definition~2.4.14}).

  \item For $\tau\in\fT^\heartsuit_2$, note that $\tilde\rT_\tau$ is isomorphic to the $q_\tau$-tame group for some power $q_\tau$ of $\|v\|$ under which the subgroup $\rT_\tau$ is the $q_\tau^2$-tame group. Thus, we may define minimally ramified liftings of $\bar\varrho_\tau$ using Definition \ref{de:minimal_tame} (with respect to the similitude character $\eta_v^{\mu_\tau}$, which is trivial on $\rT_\tau$);

  \item For $\tau\in\fT^\heartsuit_3$, note that $\tilde\rT_\tau\simeq\rT_\tau\times\dZ/2\dZ$. Then, by Lemma \ref{le:representation_selfdual}, we may regard the homomorphism $\varrho_\tau$ as a continuous homomorphism $\varrho_\tau\colon\rT_\tau\to G(R)$, where $G$ is a symplectic (resp.\ orthogonal) group of rank $m_\tau$ if $\mu_\tau$ is $0$ (resp.\ $1$). Thus, we may define minimally ramified liftings of $\bar\varrho_\tau$ using \cite{Boo19}*{Definition~5.4}.
\end{enumerate}
\end{definition}

\begin{remark}
It is straightforward to check that Definition \ref{de:minimal_deformation} do not depend on the choices of $E_\dag$, $\gamma\in\Gamma_{F^+_v}\setminus\Gamma_{F_w}$, those in Construction \ref{cs:minimal_extension}, $\iota_\tau$, and $b_\tau$.
\end{remark}

Now we allow $v$ to be a nonarchimedean place of $F^+$ that is not above $\ell$, but not necessarily nonsplit in $F$. Again, we consider a pair $(\bar{r},\chi)$ from Notation \ref{no:deformation_pair} with $\tilde\Gamma=\Gamma_{F^+_v}$ and $\Gamma=\Gamma_{F^+_v}\cap\Gamma_F$.

\begin{corollary}\label{co:bernstein}
Let $w$ be a place of $F$ above $v$ and suppose that $\ell\nmid\prod_{i=1}^N(q_w^i-1)$. Let $\iota_\ell\colon\dC\xrightarrow{\sim}\ol\dQ_\ell$ be an isomorphism, and $R$ an object of $\sC_\sO$ contained in $\ol\dQ_\ell$. Let $r_1$ and $r_2$ be two $R$-valued liftings of $\bar{r}$. If $\Pi_1$ and $\Pi_2$ are two irreducible admissible representations of $\GL_N(F_w)$ such that $r_i^\natural\otimes_R\ol\dQ_\ell$ is the induced representation of $\r{WD}(\iota_\ell\Pi_i)$ for $i=1,2$, then $\Pi_1$ and $\Pi_2$ are in the same Bernstein component.
\end{corollary}

\begin{proof}
Choose a finite unramified extension $E_\dag$ of $E$ contained in $\ol\dQ_\ell$ satisfying Assumption \ref{as:min_large}, with the ring of integers $\sO_\dag$ and the residue field $k_\dag$. Let $R_\dag$ be the subring of $\ol\dQ_\ell$ generated by $R$ and $\sO_\dag$, which is an object of $\sC_{\sO_\dag}$. Then both $r_1\otimes_RR_\dag$ and $r_2\otimes_RR_\dag$ are liftings of $\bar{r}\otimes_kk_\dag$. By Proposition \ref{pr:minimal_structure} (resp.\ Lemma \ref{le:wild_type}) when $v$ is nonsplit (resp.\ split) in $F$ and the condition that $\ell\nmid\prod_{i=1}^N(q_w^i-1)$, $\r{WD}(\iota_\ell\Pi_1)\res_{\rI_{F_w}}$ and $\r{WD}(\iota_\ell\Pi_2)\res_{\rI_{F_w}}$ are conjugate. Then it is well-known that $\Pi_1$ and $\Pi_2$ are in the same Bernstein component (see, for example, \cite{Yao}*{Lemma~3.2} for a proof).
\end{proof}

\begin{definition}\label{de:deformation_min}
When $\ell\geq N$, we define $\sD^\mnm$ to be the local deformation problem of $\bar{r}$ that classifies minimally ramified liftings in the sense of Definition \ref{de:minimal_deformation} (resp.\ \cite{CHT08}*{Definition~2.4.14}) when $v$ is nonsplit (resp.\ split) in $F$.
\end{definition}

\begin{proposition}\label{pr:deformation_min}
We have
\begin{enumerate}
  \item The ring $\sfR^\loc_{\bar{r}}$ is a reduced local complete intersection, flat and of pure relative dimension $N^2$ over $\sO$.

  \item Every irreducible component of $\Spf\sfR^\loc_{\bar{r}}$ is a local deformation problem (Definition \ref{de:local_deformation_problem}).

  \item When $\ell\geq N$, $\sD^\mnm$ is an irreducible component of $\Spf\sfR^\loc_{\bar{r}}$ and is formally smooth over $\Spf\sO$ of pure relative dimension $N^2$.
\end{enumerate}
\end{proposition}

\begin{proof}
For this proposition, we may assume that $E$ satisfies Assumption \ref{as:min_large}.

For (1), when $v$ splits in $F$, this is \cite{Sho18}*{Theorem~2.5}. Thus, we may assume that $v$ is nonsplit in $F$. By Proposition \ref{pr:minimal_structure}, $\sfR^\loc_{\bar{r}}$ is a power series ring over
\[
\widehat\bigotimes_{\tau\in\fT^\heartsuit}\sfR^\loc_{\bar\varrho_\tau}.
\]
We now claim that for every $\tau\in\fT^\heartsuit$, $\sfR^\loc_{\bar\varrho_\tau}$ a local complete intersection, flat and equidimensional. Indeed, for $\tau\in\fT_1^\heartsuit$, this is \cite{Sho18}*{Theorem~2.5}; for $\tau\in\fT_2^\heartsuit$, this is Corollary \ref{pr:deformation_tame}; for $\tau\in\fT_3^\heartsuit$, this is Proposition \ref{pr:tame_general} for $G$ a symplectic or orthogonal group with the trivial similitude character. On the other hand, by \cite{BG19}*{Theorem~3.3.2} or \cite{BP19}*{Theorem~1}, we know that $\sfR^\loc_{\bar{r}}[1/\ell]$ is reduced and of pure dimension $\dim\sG_N=N^2$. Thus, $\sfR^\loc_{\bar{r}}$ is a local complete intersection, flat and of pure relative dimension $N^2$ over $\sO$. Since $\sfR^\loc_{\bar{r}}$ is generically reduced and Cohen--Macaulay, it is reduced. (1) is proved.

For (2), take an irreducible component $\sD$ of $\Spf\sfR^\loc_{\bar{r}}$, and let $\sL_N$ be the formal completion of $\GL_{N,\sO}$ along the unit section. Then the conjugation action induces a homomorphism $\sL_N\times_{\Spf\sO}\sD\to\Spf\sfR^\loc_{\bar{r}}$ whose image contains $\sD$. Since $\sL_N$ is irreducible, the image is irreducible, hence has to be $\sD$. In other words, $\sD$ is a local deformation problem.

For (3), since $\sD^\mnm$ is Zariski closed in $\Spf\sfR^\loc_{\bar{r}}$ from its definition, it suffices to show that $\sD^\mnm$ is formally smooth over $\Spf\sO$ of pure relative dimension $N^2$. When $v$ splits in $F$, this is \cite{CHT08}*{Corollary~2.4.21}. Thus, we may assume that $v$ is nonsplit in $F$. For $\tau\in\fT^\heartsuit$, let $\sD_{\bar\varrho_\tau}^\mnm$ be the local deformation problem of $\bar\varrho_\tau$ classifying minimally ramified liftings of $\bar\varrho_\tau$ in various cases in Definition \ref{de:minimal_deformation}. By Proposition \ref{pr:minimal_structure} and Definition \ref{de:minimal_deformation}, $\sD^\mnm$ is formally smooth over
\[
\prod_{\tau\in\fT^\heartsuit}\sD_{\bar\varrho_\tau}^\mnm.
\]
We claim that for every $\tau\in\fT^\heartsuit$, $\sD_{\bar\varrho_\tau}^\mnm$ is formally smooth over $\Spf\sO$. Indeed, for $\tau\in\fT^\heartsuit_1$, this is \cite{CHT08}*{Lemma~2.4.19}; for $\tau\in\fT^\heartsuit_2$, this is Proposition \ref{pr:minimal_tame}; for $\tau\in\fT^\heartsuit_3$, this is a part of \cite{Boo19}*{Theorem~1.1}. Thus, $\sD^\mnm$ is formally smooth over $\Spf\sO$.

It remains to compute the dimension. By \eqref{eq:local_deformation_dimension}, it suffices to show that
\begin{align}\label{eq:deformation_min}
\dim_k\rL(\sD^\mnm)=\dim_k\rH^0(F^+_v,\ad\bar{r}).
\end{align}
For every $\tau\in\fT^\heartsuit$, let $\rL(\sD_{\bar\varrho_\tau}^\mnm)$ be the tangent space of the deformation problem $\sD_{\bar\varrho_\tau}^\mnm$, which is a subspace of $\rH^1(\rT_\tau,\ad\bar\varrho)$ (resp.\ $\rH^1(\tilde\rT_\tau,\ad\bar\varrho)$) if $\tau\in\fT^\heartsuit_1$ (resp.\ $\tau\in\fT^\heartsuit_2\sqcup\fT^\heartsuit_3$). By Proposition \ref{pr:minimal_structure}, we have
\begin{align}\label{eq:deformation_min_1}
\dim_k\rL(\sD^\mnm)=\sum_{\tau\in\fT^\heartsuit}\dim_k\rL(\sD_{\bar\varrho_\tau}^\mnm).
\end{align}
We claim that
\begin{align}\label{eq:deformation_min_2}
\dim_k\rL(\sD_{\bar\varrho_\tau}^\mnm)=
\begin{dcases}
\dim_k\rH^0(\rT_\tau,\ad\bar\varrho_\tau)  & \text{if $\tau\in\fT^\heartsuit_1$;} \\
\dim_k\rH^0(\tilde\rT_\tau,\ad\bar\varrho_\tau)  & \text{if $\tau\in\fT^\heartsuit_2\sqcup\fT^\heartsuit_3$.}
\end{dcases}
\end{align}
Indeed, for $\tau\in\fT^\heartsuit_1$, this is \cite{CHT08}*{Corollary~2.4.20}; for $\tau\in\fT^\heartsuit_2$, this is Corollary \ref{co:minimal_tame}; for $\tau\in\fT^\heartsuit_3$, this is a part of \cite{Boo19}*{Theorem~1.1} as $\dim_k\rH^0(\tilde\rT_\tau,\ad\bar\varrho_\tau)=\dim_k\rH^0(\rT_\tau,\ad^0\bar\varrho_\tau)$. From \eqref{eq:minimal_prestructure} for $\bar{r}$, we have
\[
\rH^0(F^+_v,\ad\bar{r})\simeq\(\bigoplus_{\tau\in\fT^\heartsuit_1}\rH^0(\rT_\tau,\ad\bar\varrho_\tau)\)\bigoplus
\(\bigoplus_{\tau\in\fT^\heartsuit_2\sqcup\fT^\heartsuit_3}\rH^0(\tilde\rT_\tau,\ad\bar\varrho_\tau)\).
\]
Together with \eqref{eq:deformation_min_1} and \eqref{eq:deformation_min_2}, we obtain \eqref{eq:deformation_min}.

The proposition is proved.
\end{proof}

\subsection{Level-raising deformations}
\label{ss:level_raising}

In this subsection, we discuss level-raising deformations. Assume $\ell\geq N\geq 2$. We take a nonarchimedean place $v$ of $F^+$ that is inert in $F$ and not above $\ell$. Let $w$ be the unique place of $F$ above $v$. Recall that we have $\rT_v=\Gamma_{F^+_v}/\rP_{F^+_v}$ and $\rT_w=\Gamma_{F_w}/\rP_{F_w}$. Then $\rT_v$ is isomorphic to the $q$-tame group and the subgroup $\rT_w$ is the $q^2$-tame group (Definition \ref{de:tame_group}), where $q=\|v\|$.

We consider a pair $(\bar{r},\chi)$ from Notation \ref{no:deformation_pair} with $\tilde\Gamma=\Gamma_{F^+_v}$ and $\Gamma=\Gamma_{F^+_v}\cap\Gamma_F=\Gamma_{F_w}$, such that $\bar{r}$ is unramified and $\chi=\eta_v^\mu\epsilon_{\ell,v}^{1-N}$ for some $\mu\in\dZ/2\dZ$. Then by Lemma \ref{le:wild_type}(1), every lifting $r$ of $\bar{r}$ to an object $R$ of $\sC_\sO$ factors through $\rT_v$. In particular, we may apply the discussion in \S\ref{ss:tame_group} to the pair $(\bar{r},\chi)$.

Now assume $\ell\nmid(q^2-1)$ and that the generalized eigenvalues of $\bar{r}^\natural(\phi_w)$ in $\ol\dF_\ell$ contain the pair $\{q^{-N},q^{-N+2}\}$ exactly once. By Lemma \ref{le:tame_decomposition}(1), for every lifting $r$ of $\bar{r}$ to an object $R$ of $\sC_\sO$, we have a canonical decomposition
\begin{align}\label{eq:tame_decomposition}
R^{\oplus N}=M_0\oplus M_1
\end{align}
of free $R$-modules stable under the action of $r^\natural(\phi_w)$, such that if we write $P_0(T)$ for the characteristic polynomial of $r^\natural(\phi_w)$ on $M_0$, then $P_0(T)\equiv(T-q^{-N})(T-q^{-N+2})\modulo\fm_R$.

\begin{definition}\label{de:deformation_level_raising}
Let $(\bar{r},\chi)$ be as above. We define $\sD^\mix$ to be the local deformation problem of $\bar{r}$ (Definition \ref{de:local_deformation_problem}) that classifies liftings $r$ to an object $R$ of $\sC_\sO$ such that in the decomposition \eqref{eq:tame_decomposition}, $r^\natural(\rI_{F_w})$ preserves $M_0$ and acts trivially on $M_1$.\footnote{Note that since $\ell\nmid(q^2-1)$, the characteristic polynomial of $r^\natural(t)$ on $M_0$ is automatically $(T-1)^2$ for every $t\in\rI_{F_w}$.}

We define
\begin{enumerate}
  \item $\sD^\unr$ to be the local deformation problem contained in $\sD^\mix$ such that the action of $r^\natural(\rI_{F_w})$ on $M_0$ is also trivial;

  \item $\sD^\ram$ to be the local deformation problem contained in $\sD^\mix$ such that $P_0(T)=(T-q^{-N})(T-q^{-N+2})$ in $R[T]$.
\end{enumerate}
\end{definition}

It is clear that $\sD^\unr$ coincides with $\sD^\mnm$ from Definition \ref{de:deformation_min}.

\begin{proposition}\label{pr:deformation_level_raising}
Suppose that $\ell\nmid(q^2-1)$ and that the generalized eigenvalues of $\bar{r}^\natural(\phi_w)$ in $\ol\dF_\ell$ contain the pair $\{q^{-N},q^{-N+2}\}$ exactly once. Then the formal scheme $\sD^\mix$ is formally smooth over $\Spf\sO[[x_0,x_1]]/(x_0x_1)$ of pure relative dimension $N^2-1$ such that the irreducible components defined by $x_0=0$ and $x_1=0$ are $\sD^\unr$ and $\sD^\ram$, respectively. In particular, $\sD^\ram$ is formally smooth over $\Spf\sO$ of pure relative dimension $N^2$.
\end{proposition}

\begin{proof}
We fix an isomorphism $\rT_v\simeq\rT_q=t^{\dZ_\ell}\rtimes\phi_q^{\widehat\dZ}$ such that $\phi_w=\phi_q^2$. We write $k^{\oplus N}=\bar{M}_0\oplus\bar{M}_1$ such that $\bar{r}^\natural(\phi_q^2)$ has eigenvalues $q^{-N}$ and $q^{-N+2}$ on $\bar{M}_0$. Without loss of generality, we may assume that $\bar{M}_0$ is spanned by the first two factors and $\bar{M}_0$ is spanned by the last $N-2$ factors. Thus, we obtain two unramified homomorphisms $\bar{r}_0\colon\rT_q\to\sG_2(k)$ and $\bar{r}_1\colon\rT_q\to\sG_{N-2}(k)$. Let $\sD_0$ be the local deformation problem of $\bar{r}_0$ classifying liftings of $\bar{r}_0$. Let $\sD_1$ be the local deformation problem of $\bar{r}_1$ classifying unramified liftings.

Suppose that $N\geq 3$. We say that lifting $r$ of $\bar{r}$ to an object $R$ of $\sC_\sO$ is \emph{standard} if
\[
r^\natural(t)=
\begin{pmatrix}
A_0 & 0 \\
0 & 1_{N-2}
\end{pmatrix},\quad
r(\phi_q)=\(
\begin{pmatrix}
B_0 & 0 \\
0 & B_1
\end{pmatrix},(-1)^{\mu+1}q^{1-N}\)\fc
\]
for some $A_0,B_0\in\GL_2(R)$ and $B_1\in\GL_{N-2}(R)$. Let $\sD^\mix_{0,1}\subseteq\sD^\mix$ be the locus of standard liftings. Then we have a natural isomorphism
\[
\sD^\mix_{0,1}\simeq\sD_0\times_{\Spf\sO}\sD_1
\]
of formal schemes over $\Spf\sO$.

For $n\geq 1$, denote by $\sL_n$ the formal completion of $\GL_{n,\sO}$ along the unit section. Then $\sL_N$ acts on $\sD^\mix$ by conjugation. We claim that $\sD^\mix_{0,1}$ generates $\sD^\mix$ under the action of $\sL_N$. For this, it suffices to show that for every lifting $r$ of $\bar{r}$ to an object $R$ of $\sC_\sO$, the maps
\[
B\colon M_0\to R^{\oplus N}\to M_1,\quad
B\colon M_1\to R^{\oplus N}\to M_0
\]
induced by $B$ from Lemma \ref{le:representation_selfdual}(1) for $\gamma=\phi_q$, are both zero. Since the two maps intertwine the actions $r^\natural$ and $r^{\natural,\vee}\otimes\epsilon_\ell^{1-N}$ of $\rT_{q^2}$, it suffices to show that the generalized eigenvalues of $r^{\natural,\vee}_0\otimes\epsilon_\ell^{1-N}(\phi_q^2)$ and the generalized eigenvalues of $r^\natural_1(\phi_q^2)$ are disjoint. However, this follows from the condition that the generalized eigenvalues of $\bar{r}^\natural(\phi_w)$ in $\ol\dF_\ell$ contain the pair $\{q^{-N},q^{-N+2}\}$ exactly once.

The above claim induces a canonical isomorphism
\[
\sD^\mix_{0,1}\times_{\Spf\sO}(\sL_2\times_{\Spf\sO}\sL_{N-2}\backslash\sL_N)\xrightarrow{\sim}\sD^\mix.
\]
By Proposition \ref{pr:minimal_tame}, $\sD_1$ is formally smooth over $\Spf\sO$ of pure relative dimensions $(N-2)^2$. Since $\sL_2\times_{\Spf\sO}\sL_{N-2}\backslash\sL_N$ is formally smooth over $\Spf\sO$ of pure relative dimension $N^2-(N-2)^2-4$, it suffices to prove the proposition for $N=2$.

Now we assume $N=2$. After changing a basis, we may assume
\[
\bar{r}(\phi_q)=(\bar{B},(-1)^{\mu+1}q^{1-N})\fc,\quad
\bar{B}=
\begin{pmatrix}
0 & (-1)^{\mu+1} \\
q & 0
\end{pmatrix}.
\]
Then we have
\[
\bar{r}^\natural(\phi_q^2)=(-1)^{\mu+1}q^{1-N}\bar{B}\tp{\bar{B}}^{-1}=
\begin{pmatrix}
q^{-N} &  0 \\
0 & q^{-N+2}
\end{pmatrix}.
\]
For every object $R$ of $\sC_\sO$, the set $\sD^\mix(R)$ is bijective to the set of pairs $(B,X)$ where $B\in\GL_2(R)$ and $X\in\rM_2(\fm_R)$ satisfying $B\equiv\bar{B}\modulo\fm_R$, that the characteristic polynomial of $X$ is $T^2$, and the relation
\begin{align}\label{eq:deformation_level_raising}
B\tp{X}B^{-1}=-qX.
\end{align}
Indeed, the bijection is given by $r(\phi_q)=(B,(-1)^{\mu+1}q^{1-N})\fc$ and $r^\natural(t)=1_2+X$. We let $\sD^\mix_0$ be the subscheme of $\sD^\mix$ defined by the condition that $r^\natural(\phi_q^2)=(-1)^{\mu+1}q^{1-N}B\tp{B}^{-1}$ is a diagonal matrix. Take a lifting $r\in\sD^\mix_0(R)$ corresponding to the pair $(B,X)$; we must have
\[
B=
\begin{pmatrix}
0 &  (-1)^{\mu+1}(1+x) \\
q(1+y) & 0
\end{pmatrix},\quad
r^\natural(\phi_q^2)=
\begin{pmatrix}
q^{-N}(1+x)(1+y)^{-1} & 0 \\
0 & q^{-N+2}(1+y)(1+x)^{-1}
\end{pmatrix}
\]
for some $x,y\in\fm_R$. Then by \eqref{eq:deformation_level_raising}, $X=(\begin{smallmatrix}0&0\\x_0&0\end{smallmatrix})$ for some $x_0\in\fm_R$ satisfying $(x-y)x_0=0$. Put $x_1\coloneqq x-y$. Then we obtain an isomorphism
\[
\sD^\mix_0\simeq\Spf\sO[[x_0,x_1,y]]/(x_0x_1)
\]
such that
\begin{itemize}
  \item $x_0=0$ if and only if $r$ is unramified;

  \item $x_1=0$ if and only if $P_0(T)=(T-q^{-N})(T-q^{-N+2})$, where $P_0$ is the characteristic polynomial of $r^\natural(\phi_w)=r^\natural(\phi_q^2)$.
\end{itemize}
Finally, not that $\sL_2$ acts on $\sD^\mix$ by conjugation, which induces a canonical isomorphism
\[
\sD^\mix_0\times_{\Spf\sO}(\sL_1\times_{\Spf\sO}\sL_1\backslash\sL_2)\xrightarrow{\sim}\sD^\mix.
\]
The proposition (for $N=2$) follows as $\sL_1\times_{\Spf\sO}\sL_1\backslash\sL_2$ is formally smooth over $\Spf\sO$ of pure relative dimension $2$. The entire proposition is now proved.
\end{proof}

\subsection{An almost minimal R=T theorem}
\label{ss:global_deformation}

In this subsection, we prove a version of the R=T theorem for a global Galois representation. Assume $N\geq 2$.

We take an element $\xi=(\xi_\tau)_\tau\in(\dZ^N_\leq)^{\Sigma_\infty}$ satisfying
$\xi_{\tau,i}=-\xi_{\tau^\tc,N+1-i}$ for every $\tau$ and $i$. Assume $\ell\geq(b_\xi-a_\xi)+2$ (Notation \ref{no:weight}) and that $\ell$ is unramified in $F$. Fix an isomorphism $\iota_\ell\colon\dC\xrightarrow{\sim}\ol\dQ_\ell$ and assume that the complex algebraic representation of $\Res_{F/\dQ}\GL_N$ determined by $\xi$ can be defined over $\iota_\ell^{-1}E$.

We consider a pair $(\bar{r},\chi)$ from Notation \ref{no:deformation_pair} with $\tilde\Gamma=\Gamma_{F^+}$ and $\Gamma=\Gamma_F$, in which $\chi=\eta^\mu\epsilon_\ell^{1-N}$ for some $\mu\in\dZ/2\dZ$. We take two finite sets $\Sigma^+_\mnm$ and $\Sigma^+_\lr$ of nonarchimedean places of $F^+$ such that
\begin{itemize}
  \item $\Sigma^+_\mnm$, $\Sigma^+_\lr$, and $\Sigma^+_\ell$ are mutually disjoint;

  \item $\Sigma^+_\mnm$ contains $\Sigma^+_\bad$;

  \item every place $v\in\Sigma^+_\lr$ is inert in $F$ and satisfies $\ell\nmid(\|v\|^2-1)$.
\end{itemize}

\begin{definition}\label{de:rigid}
We say that $\bar{r}$ is \emph{rigid for $(\Sigma^+_\mnm,\Sigma^+_\lr)$} if the following are satisfied:
\begin{enumerate}
  \item For $v$ in $\Sigma^+_\mnm$, every lifting of $\bar{r}_v$ is minimally ramified (Definition \ref{de:minimal_deformation}).

  \item For $v$ in $\Sigma^+_\lr$, the generalized eigenvalues of $\bar{r}^\natural_v(\phi_w)$ in $\ol\dF_\ell$ contain the pair $\{\|v\|^{-N},\|v\|^{-N+2}\}$ exactly once, where $w$ is the unique place of $F$ above $v$.

  \item For $v$ in $\Sigma^+_\ell$, $\bar{r}_v^\natural$ is regular Fontaine--Laffaille crystalline (Definition \ref{de:crystalline_fl}).

  \item For a nonarchimedean place $v$ of $F^+$ not in $\Sigma^+_\mnm\cup\Sigma^+_\lr\cup\Sigma^+_\ell$, the homomorphism $\ol{r}_v$ is unramified.
\end{enumerate}
\end{definition}

Suppose now that $\bar{r}$ is rigid for $(\Sigma^+_\mnm,\Sigma^+_\lr)$. Consider a global deformation problem (Definition \ref{de:global_deformation_problem})
\[
\sS\coloneqq(\bar{r},\eta^\mu\epsilon_\ell^{1-N},\Sigma^+_\mnm\cup\Sigma^+_\lr\cup\Sigma^+_\ell,
\{\sD_v\}_{v\in\Sigma^+_\mnm\cup\Sigma^+_\lr\cup\Sigma^+_\ell})
\]
where
\begin{itemize}
  \item for $v\in\Sigma^+_\mnm$, $\sD_v$ is the local deformation problem classifying all liftings of $\bar{r}_v$;

  \item for $v\in\Sigma^+_\lr$, $\sD_v$ is the local deformation problem $\sD^\ram$ of $\bar{r}_v$ from Definition \ref{de:deformation_level_raising};

  \item for $v\in\Sigma^+_\ell$, $\sD_v$ is the local deformation problem $\sD^\FL$ of $\bar{r}_v$ from Definition \ref{de:deformation_fl}.
\end{itemize}
Then we have the \emph{global universal deformation ring} $\sfR^\univ_\sS$ from Proposition \ref{pr:deformation_global}.

\begin{remark}
It is possible that $\bar{r}$ is rigid for two pairs $(\Sigma^+_\mnm,\Sigma^+_\lr)$ and $(\Sigma^{+\prime}_\mnm,\Sigma^{+\prime}_\lr)$. Then $\sfR^\univ_\sS$ and $\sfR^\univ_{\sS'}$ are different in general, where $\sS'$ denotes the corresponding global deformation problem for $(\Sigma^{+\prime}_\mnm,\Sigma^{+\prime}_\lr)$.
\end{remark}

Now we state an R=T theorem. Let $\rV$ be a hermitian space over $F$ of rank $N$ such that $\rV_v$ is not split for $v\in\Sigma^+_\lr$. Let $(p_{\ul\tau},q_{\ul\tau})_{\ul\tau\in\Sigma^+_\infty}$ be the signature of $\rV$, and put $d(\rV)\coloneqq\sum_{\ul\tau\in\Sigma^+_\infty}p_{\ul{\tau}}q_{\ul{\tau}}$.

Take a self-dual $\prod_{v\not\in\Sigma^+_\infty\cup\Sigma^+_\mnm\cup\Sigma^+_\lr}O_{F_v}$-lattice $\Lambda$ in $\rV\otimes_F\dA_F^{\Sigma^+_\infty\cup\Sigma^+_\mnm\cup\Sigma^+_\lr}$ and a neat open compact subgroup $\rK$ of $\rU(\rV)(\dA_{F^+}^\infty)$ of the form
\[
\rK=\prod_{v\in\Sigma^+_\mnm\cup\Sigma^+_\lr}\rK_v
\times\prod_{v\not\in\Sigma^+_\infty\cup\Sigma^+_\mnm\cup\Sigma^+_\lr}\rU(\Lambda)(O_{F^+_v})
\]
in which $\rK_v$ is special maximal for $v\in\Sigma^+_\lr$. We have the system of (complex) Shimura varieties $\{\Sh(\rV,\rK')\res\rK'\subseteq\rK\}$ associated to $\Res_{F^+/\dQ}\rU(\rV)$ indexed by open compact subgroups $\rK'\subseteq\rK$, which are quasi-projective smooth complex schemes of dimension $d(\rV)$.\footnote{Strictly, we need to choose a CM type $\Phi$ of $F$ to define the Deligne homomorphism for the Shimura varieties. More precisely, the Deligne homomorphism $\rh\colon\Res_{\dC/\dR}\bG_m\to(\Res_{F^+/\dQ}\rU(\rV))\otimes_\dQ\dR
=\prod_{\ul\tau\in\Sigma^+_\infty}\rU(\rV_{\ul\tau})$ is the one such that for $z\in(\Res_{\dC/\dR}\bG_m)(\dR)=\dC^\times$, the $\ul\tau$-component of $\rh(z)$ equals $1_{p_{\ul\tau}}\oplus(\ol{z}/z)1_{q_{\ul\tau}}$ after we identify $\rU(\rV_{\ul\tau})(\dR)$ as a subgroup of $\GL_N(\dC)$ via the unique place of $F$ above $\ul\tau$ in $\Phi$.} The element $\xi$ gives rise to a continuous homomorphism
\[
\prod_{v\in\Sigma^+_\ell}\rU(\Lambda)(O_{F^+_v})\to\GL_\sO(L_\xi)
\]
where $L_\xi$ is a finite free $\sO$-module, hence induces a $\sO$-linear (\'{e}tale) local system $\sL_\xi$ on $\Sh(\rV,\rK')$ for every $\rK'\subseteq\rK$, compatible under restriction.

Let $\Sigma^+$ be a finite set of nonarchimedean places of $F^+$ containing $\Sigma^+_\mnm\cup\Sigma^+_\lr\cup\Sigma^+_\ell$. In particular, we have the abstract unitary Hecke algebra $\dT^{\Sigma^+}_N$ (Definition \ref{de:abstract_hecke}). Let $\phi\colon\dT^{\Sigma^+}_N\to k$ be the homomorphism such that
\begin{itemize}
  \item for every nonarchimedean place $v$ of $F^+$ not in $\Sigma^+$ that induces one place $w$ of $F$, we have $\phi\res_{\dT_{N,v}}=\phi_{\balpha}$ (Construction \ref{cs:satake_hecke_pre}) where $\balpha=(\alpha_1,\dots,\alpha_N)$ is the unitary abstract Hecke parameter at $v$ (Definition \ref{de:satake_parameter}) satisfying that $\{\alpha_1\|v\|^{N-1},\dots,\alpha_N\|v\|^{N-1}\}$ are the generalized eigenvalues of $\bar{r}_v^\natural(\phi_w^{-1})$ in $\ol\dF_\ell$;

  \item for every nonarchimedean place $v$ of $F^+$ not in $\Sigma^+$ that splits into two places $w_1$ and $w_2$ of $F$, we have $\phi\res_{\dT_{N,v}}=\phi_{\balpha}$ (Construction \ref{cs:satake_hecke_pre}) where $\balpha=((\alpha_{1,1},\dots,\alpha_{1,N});(\alpha_{2,1},\dots,\alpha_{2,N}))$ is the unitary abstract Hecke parameter at $v$ (Definition \ref{de:satake_parameter}) satisfying that for $i=1,2$, $\{\alpha_{i,1}\|v\|^{\frac{N-1}{2}},\dots,\alpha_{i,N}\|v\|^{\frac{N-1}{2}}\}$ are the generalized eigenvalues of $\bar{r}_v^\natural(\phi_{w_i}^{-1})$ in $\ol\dF_\ell$.
\end{itemize}
We write $\fm$ for the kernel of $\phi$.

\begin{theorem}\label{th:deformation}
Suppose that $\Sigma^+_\lr=\emptyset$ if $N$ is odd. Under the above setup, we assume
\begin{description}
  \item[(D0)] (already assumed) $\ell$ is odd, $\ell\geq(b_\xi-a_\xi)+2$, and $\ell$ is unramified in $F$;

  \item[(D1)] $\ell\geq 2(N+1)$;

  \item[(D2)] $\bar{r}^\natural\res_{\Gal(\ol{F}/F(\zeta_\ell))}$ is absolutely irreducible;

  \item[(D3)] $\bar{r}$ is rigid for $(\Sigma^+_\mnm,\Sigma^+_\lr)$ (Definition \ref{de:rigid});

  \item[(D4)] for every finite set $\Sigma^{+\prime}$ of nonarchimedean places of $F^+$ containing $\Sigma^+$, and every open compact subgroup $\rK'\subseteq\rK$ satisfying $\rK'_v=\rK_v$ for $v\not\in\Sigma^{+\prime}$, we have
      \[
      \rH^d_\et(\Sh(\rV,\rK'),\sL_\xi\otimes_\sO k)_{\dT_N^{\Sigma^{+\prime}}\cap\fm}=0
      \]
      when $d\neq d(\rV)$.\footnote{This is automatic when $d(\rV)=0$, and follows from (D2) when $d(\rV)=1$.}
\end{description}
Let $\sfT$ be the image of $\dT^{\Sigma^+}_N$ in $\End_\sO\(\rH^{d(\rV)}_\et(\Sh(\rV,\rK),\sL_\xi)\)$. If $\sfT_\fm\neq0$, then
\begin{enumerate}
  \item There is a canonical isomorphism $\sfR^\univ_\sS\xrightarrow{\sim}\sfT_\fm$ of local complete intersection rings over $\sO$.

  \item The $\sfT_\fm$-module $\rH^{d(\rV)}_\et(\Sh(\rV,\rK),\sL_\xi)_\fm$ is finite and free.

  \item We have $\mu\equiv N\modulo 2$.
\end{enumerate}
\end{theorem}

The rest of this subsection is devoted to the proof of the theorem. We will use the Taylor--Wiles patching argument following \cite{CHT08} and \cite{Tho12}. Put $\rS\coloneqq\Sigma^+_\mnm\cup\Sigma^+_\lr\cup\Sigma^+_\ell$. To prove the theorem, we may replace $E$ by a finite unramified extension in $\ol\dQ_\ell$. Thus, we may assume that $k$ contains all eigenvalues of matrices in $\bar{r}^\natural(\Gamma_F)$.

\begin{remark}\label{re:adequate}
By (D0), we know that $F$ is not contained in $F^+(\zeta_\ell)$. Thus, by \cite{Tho12}*{Theorem~A.9}, (D1) and (D2) imply that $\bar{r}(\Gal(\ol{F}/F^+(\zeta_\ell)))$ is adequate in the sense of \cite{Tho12}*{Definition~2.3}.
\end{remark}

Recall that a prime $v$ of $F^+$ is called a \emph{Taylor--Wiles prime} for the global deformation problem $\sS$ if
\begin{itemize}
  \item $v\notin \rS$; $v$ splits in $F$; and $\|v\|\equiv 1\modulo\ell$;

  \item $\bar{r}_v$ is unramified;
		
  \item $\bar{r}_v^\natural(\phi_w)$ is not a scalar and admits an eigenvalue $\bar\alpha_v\in k$, called \emph{special eigenvalue}, such that $\bar{r}_v^\natural(\phi_w)$ acts semisimply on the generalized eigenspace for $\bar\alpha_v$, where $w$ is the place of $F$ above $v$ induced by the inclusion $F\subseteq\ol{F}^+_v$.
\end{itemize}

A \emph{Taylor--Wiles system} is a tuple $(\rQ,\{\bar\alpha_v\}_{v\in\rQ})$ where $\rQ$ is a (possibly empty) finite set of Taylor--Wiles primes, and $\bar\alpha_v$ is a special eigenvalue for every $v\in\rQ$. For such a system, we write $r^\natural_v=r^\bullet_v\oplus r^\circ_v$ for every $v\in\rQ$, where $r^\bullet_v$ (resp.\ $r^\circ_v$) is the generalized eigenspace for $\bar\alpha_v$ (resp.\ for generalized eigenvalues other than $\alpha_v$). Then we have another global deformation problem (see \cite{Tho12}*{Definition~4.1})
\[
\sS(\rQ)\coloneqq(\bar{r},\eta^\mu\epsilon_\ell^{1-N},\rS\cup\rQ,\{\sD_v\}_{v\in\rS\cup\rQ})
\]
where $\sD_v$ is the same as in $\sS$ for $v\in\rS$; and for $v\in\rQ$, $\sD_v$ is the local deformation problem of $\bar{r}_v$ that classifies liftings $r_v$ such that $r_v^\natural$ is of the form $r^\bullet_v\oplus r^\circ_v$ in which $r^\bullet_v$ is a lifting of $\bar{r}^\bullet_v$ on which $\rI_{F_w}$ acts by scalars, and $r^\circ_v$ is an unramified lifting of $\bar{r}^\circ_v$.

For each $v\in\rQ$, we
\begin{itemize}
  \item put $d_v\coloneqq\dim_k\bar{r}^\bullet_v$;

  \item let $\rP_{d_v}\subseteq\GL_N$ be the standard upper-triangular parabolic subgroup corresponding to the partition $(N-d_v,d_v)$;

  \item let $\kappa_v$ be the residue field of $F^+_v$, and $\Delta_v$ the maximal quotient of $\kappa_v^\times$ of $\ell$-power order;

  \item fix an isomorphism $\rK_v\simeq\GL_N(O_{F^+_v})$ and denote by $\rK_{v,0}\subseteq\rK_v$ the parahoric subgroup corresponding to $\rP_{d_v}$;

  \item let $\rK_{v,1}$ be the kernel of the canonical map
     \[
     \rK_{v,0}\to\rP_{d_v}(\kappa_v)\to\GL_{d_v}(\kappa_v)\xrightarrow{\det}\kappa_v^\times\to\Delta_v.
     \]
\end{itemize}
We then
\begin{itemize}
  \item put $\Delta_\rQ\coloneqq\prod_{v\in\rQ}\Delta_v$; and let $\fa_\rQ$ be the augmentation ideal of $\sO[\Delta_\rQ]$;

  \item put $\fm_\rQ\coloneqq\fm\cap\dT^{\Sigma^+\cup\rQ}_N$;

  \item put
     \[
     \rK_i(\rQ)=\prod_{v\not\in\rQ}\rK_v\times\prod_{v\in\rQ}\rK_{v,i}
     \]
     for $i=0,1$, which are open compact subgroups of $\rK$.
\end{itemize}
In particular, $\rK_1(\rQ)$ is a normal subgroup of $\rK_0(\rQ)$; and we have a canonical isomorphism
\begin{align}\label{eq:diamond}
\rK_0(\rQ)/\rK_1(\rQ)\xrightarrow{\sim}\Delta_\rQ.
\end{align}
Note that when $\rQ=\emptyset$, $\rK_0(\rQ)=\rK_1(\rQ)=\rK$.

For $i=0,1$, we put
\[
H_{\rK_i(\rQ)}\coloneqq\Hom_\sO\(\rH^{d(\rV)}_{\et}(\Sh(\rV,\rK_i(\rQ)),\sL_\xi),\sO\).
\]
By \eqref{eq:diamond}, $H_{\rK_1(\rQ)}$ is canonically a module over $\sO[\Delta_\rQ]$.

\begin{lem}\label{le:taylor_wiles_delta}
Let the situation be as in Theorem \ref{th:deformation}. The $\sO[\Delta_\rQ]$-module $H_{\rK_1(\rQ),\fm_\rQ}$ is a finite and free. Moreover, the canonical map
\[
H_{\rK_1(\rQ),\fm_\rQ}/\fa_\rQ\to H_{\rK_0(\rQ),\fm_\rQ}
\]
is an isomorphism.
\end{lem}

\begin{proof}
For a smooth complex scheme $\rX$ and an abelian group $A$, we denote by $\rX^\an$ the underlying complex manifold and $\rC_\bullet(\rX^\an,A)$ the complex of singular chains for $\rX^\an$ with coefficients in $A$.

For $i=0,1$ and every positive integer $m$, we denote by $\rK_i^m(\rQ)$ the kernel of the composite map
\[
\rK_i(\rQ)\to\rK\to\ru_m\coloneqq\prod_{v\in\Sigma^+_\ell}\rU(\Lambda)(O_{F^+_v}/\ell^m).
\]
Then $\rK_i^m(\rQ)$ acts trivially on $L_\xi\otimes_\sO\sO/\ell^m$, hence $\sL_\xi\otimes_\sO\sO/\ell^m$ becomes constant on $\Sh(\rV,\rK_i^m(\rQ))$. By (D4), we have
\[
\rH^{d(\rV)}_\et(\Sh(\rV,\rK_i(\rQ)),\sL_\xi)_{\fm_\rQ}=
\varprojlim_m\rH^{d(\rV)}_\et(\Sh(\rV,\rK_i^m(\rQ)),L_\xi\otimes_\sO\sO/\ell^m)_{\fm_\rQ}^{\ru_m}
\]
and is $\sO$-torsion free. By Artin's comparison theorem between the singular cohomology and the \'etale cohomology, the dual complex
\[
\Hom_\sO(\rC_\bullet(\Sh(\rV,\rK_i^m(\rQ))^\an,L_\xi\otimes_\sO\sO/\ell^m),\sO/\ell^m)
\]
calculates $\rH^\bullet_\et(\Sh(\rV,\rK_i^m(\rQ)),L_\xi\otimes_\sO\sO/\ell^m)$. By (D4), we have
\[
\rH_d(\Sh(\rV,\rK_i^m(\rQ))^\an,L_\xi\otimes_\sO\sO/\ell^m)_{\fm_\rQ}=0
\]
for $d\neq d(\rV)$.

On the other hand, $\rK_i^m(\rQ)$ is neat, which implies that $t^{-1}\rU(\rV)(F^+)t\cap\rK_i^m(\rQ)$ has no torsion elements for every $t\in\rU(\rV)(\dA_{F^+}^\infty)$. By (the same proof of) \cite{KT17}*{Lemma~6.9}, for every $m\geq 1$, $\rC_\bullet(\Sh(\rV,\rK_1^m(\rQ))^\an,L_\xi\otimes_\sO\sO/\ell^m)$ is a perfect complex of free $(\sO/\ell^m)[\Delta_\rQ]$-modules; and there is a canonical isomorphism
\[
\rC_\bullet(\Sh(\rV,\rK_1^m(\rQ))^\an,L_\xi\otimes_\sO\sO/\ell^m)\otimes_{\sO[\Delta_\rQ]}\sO[\Delta_\rQ]/\fa_\rQ
\simeq\rC_\bullet(\Sh(\rV,\rK_0^m(\rQ))^\an,L_\xi\otimes_\sO\sO/\ell^m)
\]
of complexes of $(\sO/\ell^m)[\dT^{\Sigma^+\cup\rQ}_N]$-modules. It follows easily that the canonical map
\[
\(H_{\rK_1(\rQ),\fm_\rQ}\otimes_\sO\sO/\ell^m\)/\fa_\rQ\to H_{\rK_0(\rQ),\fm_\rQ}\otimes_\sO\sO/\ell^m
\]
is an isomorphism, after localizing at $\fm_\rQ$ and taking $\ru_m$-invariants.

Then the lemma follows by passing to the limit for $m$.
\end{proof}

We now discuss the existence of Taylor--Wiles systems. For each $v\in\rS$, we have the tangent space $\rL(\sD_v)\subseteq\rH^1(F^+_v,\ad\bar{r})$ from Definition \ref{de:tangent_space_deformation}. Let $\rL(\sD_v)^\perp\subseteq\rH^1(F^+_v,\ad\bar{r}(1))$ be the annihilator of $\rL(\sD_v)$ under the local Tate duality induced by the perfect pairing $\ad\bar{r}\times\ad\bar{r}(1)\to k(1)$ sending $(x,y)$ to $\tr(xy)$. Recall that $\Gamma_{F^+,\rS}$ is the Galois group of the maximal subextension of $\ol{F}/F^+$ that is unramified outside $\rS$. For every subset $\rT\subseteq\rS$, we define $\rH^1_{\sL^\perp,\rT}(\Gamma_{F^+,\rS},\ad\bar{r}(1))$ to be the kernel of the natural map
\[
\rH^1(\Gamma_{F^+,\rS},\ad\bar{r}(1))\to\bigoplus_{v\in\rS\setminus\rT}\rH^1(F^+_v,\ad\bar{r}(1))/\rL(\sD_v)^\perp.
\]
Recall the rings $\sfR^\loc_{\sS,\rT}$ \eqref{eq:global_deformation_problem} and $\sfR^{\Box_\rT}_{\sS(\rQ)}$ from Proposition \ref{pr:deformation_global}. Moreover, $\sfR^{\Box_\rT}_{\sS(\rQ)}$ is naturally an algebra over $\sfR^\loc_{\sS,\rT}$.

\begin{lem}\label{le:taylor_wiles}
Let the situation be as in Theorem \ref{th:deformation}. Let $\rT$ be a subset of $\rS$. For every integer $b\geq\dim_k\rH^1_{\sL^\perp,\rT}(\Gamma_{F^+,\rS},\ad\bar{r}(1))$ and every integer $n\geq 1$, there is a Taylor--Wiles system $(\rQ_n,\{\bar\alpha_v\}_{v\in\rQ_n})$ satisfying
\begin{enumerate}
  \item $|\rQ_n|=b$;

  \item $\|v\|\equiv 1\modulo\ell^n$;

  \item $\sfR^{\Box_\rT}_{\sS(\rQ_n)}$ can be topologically generated over $\sfR^\loc_{\sS,\rT}$ by
     \[
     g_{b,\rT}\coloneqq b-\sum_{v\in\rT\cap\Sigma^+_\ell}[F^+_v:\dQ_\ell]\frac{N(N-1)}{2}-N[F^+:\dQ]\frac{1+(-1)^{\mu+1-N}}{2}
     \]
     elements.
\end{enumerate}
\end{lem}

\begin{proof}
By \eqref{eq:local_deformation_dimension}, Proposition \ref{pr:deformation_fl} for $v\in\Sigma^+_\ell$, Proposition \ref{pr:deformation_min} for $v\in\Sigma^+_\mnm$ (which is applicable by (D3)), and Proposition \ref{pr:deformation_level_raising} for $v\in\Sigma^+_\lr$, we have for every $v\in\tS$ that
\[
\dim_k\rL(\sD_v)-\dim_k\rH^0(F^+_v,\ad\bar{r}(1))=
\begin{dcases}
[F^+_v:\dQ_\ell]\tfrac{N(N-1)}{2} & \text{if $v\mid\ell$;} \\
0 & \text{if $v\nmid\ell$.}
\end{dcases}
\]
Then the lemma follows from \cite{Tho12}*{Proposition~4.4} in view of Remark \ref{re:adequate}.\footnote{Strictly speaking, the set $\rS$ in \cite{Tho12}*{Proposition~4.4} consists of only places split in $F$. But the same argument works in our case as well.}
\end{proof}

\begin{proof}[Proof of Theorem \ref{th:deformation}]
By definition, we have
\[
H_{\rK,\fm}=\Hom_\sO\(\rH^{d(\rV)}_\et(\Sh(\rV,\rK),\sL_\xi)_\fm,\sO\),
\]
under which $\sfT_\fm$ is identified with the image of $\dT^{\Sigma^+}_N$ in $\End_\sO(H_{\rK,\fm})$ since $\rH^{d(\rV)}_\et(\Sh(\rV,\rK),\sL_\xi)_\fm$ is $\sO$-torsion free by (D4).

First, we need to construct a canonical homomorphism $\sfR^\univ_\sS\to\sfT_\fm$. It is well-known that $\sfT_\fm[1/\ell]$ is a reduced ring finite over $E$. As $H_{\rK,\fm}$ is a finite free $\sO$-module, $\sfT_\fm$ is a reduced ring finite flat over $\sO$. Via $\iota_\ell$, every (closed) point $x$ of $\Spec\sfT_\fm[1/\ell]$ gives rise to an RACSDC representation $\Pi_x$ of $\GL_N(\dA_F)$ (Definition \ref{de:relevant}), satisfying that
\begin{enumerate}[label=(\alph*)]
  \item the associated Galois representation $\rho_{\Pi_x,\iota_\ell}$ from Proposition \ref{pr:galois}(2) is residually isomorphic to $\bar{r}^\natural\otimes_k\ol\dF_\ell$ (hence residually absolutely irreducible by (D2));

  \item $\Pi=\BC(\pi)$ for a cuspidal automorphic representation $\pi$ of $\rU(\rV)(\dA_{F^+})$ satisfying that $(\pi^\infty)^\rK$ appears (nontrivially) in $\rH^{d(\rV)}_\et(\Sh(\rV,\rK),\sL_\xi\otimes_{\sO,\iota_\ell^{-1}}\dC)$;

  \item the archimedean weights of $\Pi$ equals $\xi$, which follows from (b) and Proposition \ref{pr:arthur}.
\end{enumerate}
The representation $\rho_{\Pi_x,\iota_\ell}$ induces a continuous homomorphism
\[
\rho_x\colon\Gamma_F\to\GL_N(\sfT_x),
\]
where $\sfT_x$ denotes the localization of $\sfT_\fm[1/\ell]$ at $x$. By a theorem of Carayol \cite{Car94}*{Th\'{e}or\`{e}me~2}, the product homomorphism
\[
\prod_{x\in\Spec\sfT_\fm[1/\ell]}\rho_x\colon\Gamma_F\to\GL_N\(\prod_x\sfT_x\)
\]
is conjugate to some continuous homomorphism $\rho_\fm\colon\Gamma_F\to\GL_N(\sfT_\fm)$ that is a lifting of $\bar{r}^\natural$. Moreover, by Proposition \ref{pr:galois}(2c) and Lemma \ref{le:representation_selfdual}, we obtain a continuous homomorphism
\[
r_\fm\colon\Gamma_{F^+}\to\sG_N(\sfT_\fm)
\]
satisfying $r_\fm^\natural=\rho_\fm$ that is a lifting of $\bar{r}$. We claim that $r_\fm$ satisfies the global deformation problem $\sS$. Indeed, by (b) and Proposition \ref{pr:arthur}, $\Pi_{x,w}$ is unramified for nonarchimedean places $w$ of $F$ not above $\Sigma^+_\mnm\cup\Sigma^+_\lr$. Thus, by (c) and Proposition \ref{pr:galois}(2b), $\bar{r}_{\fm,v}$ belongs to $\sD^\FL_v$ for $v\in\Sigma^+_\ell$; and by Proposition \ref{pr:galois}(2a), $\bar{r}_{\fm,v}$ is unramified for $v\not\in\rS$. By (b), Corollary \ref{co:ram}, and Proposition \ref{pr:galois}(2a), $\bar{r}_{\fm,v}$ belongs to $\sD^\ram_v$ for $v\in\Sigma^+_\lr$.\footnote{This is not correct if $N$ is odd, which is the only reason that we suppose that $\Sigma^+_\lr=\emptyset$ if $N$ is odd in the statement of the theorem.}

Therefore, by the universal property of $\sfR^\univ_\sS$, we obtain a canonical homomorphism
\begin{align}\label{eq:deformation}
\varphi\colon\sfR^\univ_\sS\to\sfT_\fm
\end{align}
of rings over $\sO$. Moreover, it is clear that our homomorphism $r_\fm$ satisfies \cite{CHT08}*{Proposition~3.4.4(2,3)} as well, which implies that $\varphi$ is surjective. Thus, it remains to show that $\varphi$ is injective.

We follow the strategy for \cite{Tho12}*{Theorem~6.8}. Fix a subset $\rT\subseteq\rS$. We take an integer $n\geq 1$, and a Taylor--Wiles system $(\rQ_n,\{\bar\alpha_v\}_{v\in\rQ_n})$ from Lemma \ref{le:taylor_wiles}. For each $v\in\rQ_n$, we
\begin{itemize}
  \item let $\Art_v\colon F_v^\times\to\Gamma_{F^+_v}^\ab$ be the local Artin map;

  \item let $\varpi_v\in F^+_v$ be the uniformizer such that $\Art_v(\varpi_v)$ coincides with the image of $\phi_v^{-1}$ in $\Gamma_{F^+_v}^\ab$;

  \item let $\pr_{\varpi_v}$ be the commuting projection defined in \cite{Tho12}*{Propositions~5.9~\&~5.12};

  \item for every $\alpha\in O_{F^+_v}^\times$, let $\tV_v^\alpha\in\dZ[\rK_{v,1}\backslash\rK_{v,0}/\rK_{v,1}]$ be the characteristic function of the double coset
      \[
      \rK_{v,1}
      \begin{pmatrix}
      1_{N-1} & 0 \\
      0 & \alpha
      \end{pmatrix}
      \rK_{v,1}
      \].
\end{itemize}
For $i=0,1$, we put
\[
M_{i,\rQ_n}\coloneqq\(\prod_{v\in\rQ_n}\pr_{\varpi_v}\)H_{\rK_i(\rQ_n),\fm_{\rQ_n}},
\]
and let $\sfT_{i,\rQ_n}$ be the image of $\dT^{\Sigma^+\cup\rQ_n}_N$ in $\End_\sO(M_{i,\rQ_n})$. We also put
\[
M\coloneqq H_{\rK,\fm}.
\]
Then the canonical map $M\to H_{\rK,\fm_{\rQ_n}}$ is an isomorphism, hence we obtain canonical surjective homomorphisms
\[
\sfT_{1,\rQ_n}\twoheadrightarrow\sfT_{0,\rQ_n}\twoheadrightarrow\sfT_\fm
\]
of rings over $\sO$. Similar to $\sfT_\fm$, we obtain a continuous homomorphism
\[
r_{i,\rQ_n}\colon\Gamma_{F^+}\to\sG_N(\sfT_{i,\rQ_n}),
\]
which is a lifting of $\bar{r}$, for $i=0,1$. We have the following two claims:
\begin{enumerate}
  \item For every $v\in\rQ_n$, there is a continuous character $\tv_v\colon O^\times_{F^+_v}\to\sfT_{1,\rQ_n}^\times$ that factors through $\Delta_v$ such that
    \begin{itemize}
      \item for every $\alpha\in O_{F^+_v}^\times$, the actions of $\tV_v^\alpha$ and $\tv_v(\alpha)$ on $M_{1,\rQ_n}$ coincide;

      \item $r_{1,\rQ_n,v}^\natural$ has a (unique) decomposition $r_{1,\rQ_n,v}^\bullet\oplus r_{1,\rQ_n,v}^\circ$ such that $r_{1,\rQ_n,v}^\bullet$ is a lifting of $\bar{r}_v^\bullet$ on which $\rI_{F^+_v}$ acts via the character $\tv_v\circ\Art_v^{-1}$, and $r_{1,\rQ_n,v}^\circ$ is an unramified lifting of $\bar{r}_v^\circ$.
    \end{itemize}

  \item The composite map
     \[
     M=H_{\rK,\fm_{\rQ_n}}\to H_{\rK_0(\rQ_n),\fm_{\rQ_n}}\xrightarrow{\prod_{v\in\rQ_n}\pr_{\varpi_v}} M_{0,\rQ_n}
     \]
     is an isomorphism. In particular, the canonical homomorphism $\sfT_{0,\rQ_n}\to\sfT_\fm$ is an isomorphism; and $r_{0,\rQ_n}$ and $r_\fm$ are equivalent liftings of $\bar{r}$.
\end{enumerate}
Indeed, these claims follow easily from \cite{Tho12}*{Propositions~5.9~\&~5.12}.

It follows from (1) that $r_{1,\rQ_n}$ satisfies the global deformation problem $\sS(\rQ_n)$, which induces a canonical surjective homomorphism
\[
\varphi_n\colon\sfR^\univ_{\sS(\rQ_n)}\twoheadrightarrow\sfT_{1,\rQ_n}
\]
of rings over $\sO$. We regard $\sfT_{1,\rQ_n}$ as an $\sO[\Delta_{\rQ_n}]$-module by (1). Now we claim that $\sfR^\univ_{\sS(\rQ_n)}$ is naturally an $\sO[\Delta_{\rQ_n}]$-module as well, and that $\varphi_n$ is $\sO[\Delta_{\rQ_n}]$-linear. Indeed, take a universal lifting $r^\univ_{\sS(\rQ_n)}$ for $\bar{r}$ over $\sfR^\univ_{\sS(\rQ_n)}$. Then for each $v\in\rQ_n$, there is a unique character $\rv^\univ_v\colon\Delta_v\to(\sfR^\univ_{\sS(\rQ_n)})^\times$ such that $\rI_{F^+_v}$ acts on $r^{\univ,\bullet}_{\sS(\rQ_n),v}$ via the character
\[
\rI_{F^+_v}\xrightarrow{\Art_v^{-1}}O_{F^+_v}^\times\to\kappa_v^\times\to\Delta_v\xrightarrow{\tv^\univ_v}(\sfR^\univ_{\sS(\rQ_n)})^\times.
\]
Then $\sfR^\univ_{\sS(\rQ_n)}$ becomes a ring over $\sO[\Delta_{\rQ_n}]$ via the character $\prod_{v\in\rQ_n}\rv^\univ_v\colon\Delta_{\rQ_n}\to(\sfR^\univ_{\sS(\rQ_n)})^\times$. Moreover, $\varphi_n$ is a homomorphism of rings over $\sO[\Delta_{\rQ_n}]$ by the local-global compatibility. By (2) and Lemma \ref{le:taylor_wiles_delta}, we obtain a canonical commutative diagram
\[
\xymatrix{
\sfR^\univ_{\sS(\rQ_n)}/\fa_{\rQ_n} \ar[rr]^-{\sim}\ar[d]_-{\varphi_n/\fa_{\rQ_n}} && \sfR^\univ_\sS  \ar[d]^-{\varphi}  \\
\sfT_{1,\rQ_n}/\fa_{\rQ_n} \ar[r]^-{\sim} & \sfT_{0,\rQ_n} \ar[r]^-{\sim} & \sfT_\fm
}
\]
of rings over $\sO$, where all horizontal arrows are isomorphisms.

Choose universal liftings
\[
r^\univ_\sS\colon\Gamma_{F^+}\to\sG_N(\sfR^\univ_\sS),\quad
r^\univ_{\sS(\rQ_n)}\colon\Gamma_{F^+}\to\sG_N(\sfR^\univ_{\sS(\rQ_n)})
\]
for $\bar{r}$ over $\sfR^\univ_\sS$ and $\sfR^\univ_{\sS(\rQ_n)}$, respectively, such that $r^\univ_\sS=r^\univ_{\sS(\rQ_n)}\modulo\fa_{\rQ_n}$. By Proposition \ref{pr:deformation_global}(2), we obtain isomorphisms
\[
\sfR^\univ_\sS[[X_{v;i,j}]]_{v\in\rT;1\leq i,j\leq N}\xrightarrow{\sim}\sfR^{\Box_\rT}_{\sS},\quad
\sfR^\univ_{\sS(\rQ_n)}[[X_{v;i,j}]]_{v\in\rT;1\leq i,j\leq N}\xrightarrow{\sim}\sfR^{\Box_\rT}_{\sS(\rQ_n)}
\]
of rings over $\sO$. In particular, we have a surjective homomorphism $\sfR^{\Box_\rT}_{\sS}\to\sfR^\univ_\sS$, which makes $\sfR^\univ_\sS$ an algebra over $\sfR^\loc_{\sS,\rT}$.

We put
\[
S_\infty\coloneqq\sO[[X_{v;i,j}]]_{v\in\rT;1\leq i,j\leq N}[[Y_1,\dots,Y_b]];
\]
and let $\fa_\infty\subseteq S_\infty$ be the augmentation ideal. Put
\[
R_\infty\coloneqq\sfR^\loc_{\sS,\rT}[[Z_1,\dots,Z_{g_{b,\rT}}]]
\]
where $g_{b,\rT}$ is the number appearing in Lemma \ref{le:taylor_wiles}. Applying the usual patching lemma (see the proof of \cite{BLGG11}*{Theorem~3.6.1}, or \cite{Tho12}*{Lemma~6.10}), we have the following:
\begin{itemize}
  \item There exists a homomorphism $S_\infty\to R_\infty$ of rings over $\sO$ such that we have an isomorphism $R_\infty/\fa_\infty R_\infty\simeq\sfR^\univ_\sS$ of rings over $\sfR^\loc_{\sS,\rT}$.

  \item There exist an $R_\infty$-module $M_\infty$ and an isomorphism $M_\infty/\fa_\infty M_\infty\simeq M$ of $\sfR^\univ_\sS$-modules.

  \item As an $S_\infty$-module, $M_\infty$ is finite and free.
\end{itemize}
In particular, we have
\[
\r{depth}_{R_\infty}(M_\infty)\geq\dim S_\infty=1+|\rT|N^2+b.
\]
On the other hand, by Proposition \ref{pr:deformation_fl} for $v\in\rT\cap\Sigma^+_\ell$, Proposition \ref{pr:deformation_min} for $v\in\rT\cap\Sigma^+_\mnm$ (which is applicable by (D3)), and Proposition \ref{pr:deformation_level_raising} for $v\in\rT\cap\Sigma^+_\lr$, we know that $\sfR^\loc_{\sS,\rT}$ is a formal power series ring over $\sO$ in
\[
|\rT| N^2+\sum_{v\in\rT\cap\Sigma^+_\ell}[F^+_v:\dQ_\ell]\frac{N(N-1)}{2}
\]
variables. It follows that $R_\infty$ is a regular local ring of dimension
\[
1+|\rT|N^2+\sum_{v\in\rT\cap\Sigma^+_\ell}[F^+_v:\dQ_\ell]\frac{N(N-1)}{2}+g_{b,\rT}
=1+|\rT|N^2+b-N[F^+:\dQ]\frac{1+(-1)^{\mu+1-N}}{2}.
\]
As $\dim R_\infty\geq\r{depth}_{R_\infty}(M_\infty)$, we obtain Theorem \ref{th:deformation}(3). By the Auslander--Buchsbaum theorem, $M_\infty$ is a finite free $R_\infty$-module. Thus, $M$ is a finite free $\sfR^\univ_\sS$-module. In particular, the surjective homomorphism $\varphi$ \eqref{eq:deformation} is injective, hence an isomorphism. Theorem \ref{th:deformation}(1,2) are proved.
\end{proof}

\section{Rigidity}
\label{ss:4}

\subsection{Rigidity of symmetric powers of elliptic curves}

In this subsection, we study rigidity of symmetric powers of elliptic curves.

Let $A$ be an elliptic curve over $F^+$. For every rational prime $\ell$, we fix an isomorphism $\rH^1_\et({A_\alpha}_{\ol{F}},\dZ_\ell)\simeq\dZ_\ell^{\oplus 2}$, hence obtain a continuous homomorphism $\rho_{A,\ell}\colon\Gamma_{F^+}\to\GL_2(\dZ_\ell)$. Suppose that $N\geq 2$. We obtain a continuous homomorphism
\[
r_{A,\ell}\colon\Gamma_{F^+}\to\sG_N(\dZ_\ell)=(\GL_N(\dZ_\ell)\times\dZ_\ell^\times)\rtimes\{1,\fc\}
\]
by the formula $r_{A,\ell}(\gamma)=(\Sym^{N-1}\rho_{A,\ell}(\gamma),\eta_v^{N-1}\epsilon_{\ell,v}^{1-N}(\gamma))\fc(\gamma)$, where $\fc(\gamma)=\fc$ if and only if $\gamma\in\Gamma_{F^+}\setminus\Gamma_F$. Denote by $\bar{r}_{A,\ell}$ the composition of $r_{A,\ell}$ and the projection $\sG_N(\dZ_\ell)\to\sG_N(\dF_\ell)$.

\begin{proposition}\label{pr:minimal_elliptic}
Let $v$ be a nonarchimedean place of $F$. For all but finitely many rational primes $\ell\geq N$, every lifting of $\bar{r}_{A,\ell,v}$ to an object $R$ of $\sC_{\dZ_\ell}$ (with respect to the similitude character $\eta_v^N\epsilon_{\ell_v}^{1-N}$) is minimally ramified in the sense of Definition \ref{de:minimal_deformation}.
\end{proposition}

\begin{proof}
For simplicity, we only prove the proposition for $v$ nonsplit in $F$. The split case is similar and easier, which we leave to the readers. Thus, let $w$ be the unique place of $F$ above $v$. Fix a finite totally ramified extension $F'_w$ of $F_w$ in $\ol{F}^+_v$ such that $A'\coloneqq A\otimes_{F^+_v}F'_w$ has either good or multiplicative reduction. Let $\rT'_w$ be the image of the subgroup $\Gal(\ol{F}_v^+/F'_w)$ of $\Gamma_{F_w}$ in $\rT_w=\Gamma_{F_w}/\rP_{F_w}$. We fix an isomorphism $\rT_w\simeq\rT_q=t^{\dZ_\ell}\rtimes\phi_q^{\widehat\dZ}$ with the $q$-tame group, where $q=\|w\|$. We now assume $\ell>[F'_w:F_w]$ and $\ell\nmid\prod_{i=1}^N(q^i-1)$. Then $\rT'_w=\rT_w$. Let $\fT=\fT(\bar{r}_{A,\ell,v})$ be the set of isomorphism classes of absolutely irreducible representations of $\rP_{F_w}$ appearing in $\bar{r}_{A,\ell,v}^\natural$ as before.

We first consider the case where $A'$ has split multiplicative reduction. Let $u$ be the valuation of the $j$-invariant $j(A)$ in $F'_w$, which is a negative integer. Assume further that $\ell$ is coprime to $u$. Then $\rho_{A',\ell}(t)$ is conjugate to $1+J_2=(\begin{smallmatrix}1&1\\0&1\end{smallmatrix})$ in $\GL_2(\dZ_\ell)$, which implies that $\Sym^{N-1}\rho_{A',\ell}(t)$ is conjugate to $1+J_N$ in $\GL_N(\dZ_\ell)$. It follows that $\fT$ is a singleton, say $\{\tau\}$; and every lifting $\varrho_\tau$ of $\bar\varrho_\tau$ is minimally ramified since $\ell\nmid\prod_{i=1}^N(q^i-1)$. Thus, every lifting $r$ of $\bar{r}_{A,\ell,v}$ is minimally ramified.

We then consider the case where $A'$ has good reduction. Let $\alpha,\beta\in\ol\dQ_\ell$ be the two eigenvalues of $\rho_{A',\ell}(\phi_q)$. Then $\alpha,\beta$ are Weil $q^{-1/2}$-numbers in $\ol\dQ$, which depend only on $A'$, not on $\ell$. We further assume that $\ell$ satisfies that $\alpha,\beta\in\ol\dZ_\ell^\times$, and that the image of the set
\[
\{(\alpha/\beta)^{N-1},(\alpha/\beta)^{N-3},\dots,(\alpha/\beta)^{3-N},(\alpha/\beta)^{1-N}\}
\]
in $\ol\dF_\ell^\times$ does not contain $q$. It follows that for every $\tau\in\fT$, every lifting $\varrho_\tau$ of $\bar\varrho_\tau$ is actually unramified by Lemma \ref{le:tame_decomposition}(2) as $\ell\nmid\prod_{i=1}^N(q^i-1)$, hence minimally ramified. Thus, every lifting $r$ of $\bar{r}_{A,\ell,v}$ is minimally ramified.

Since in both cases, we only exclude finitely many rational primes $\ell$, the proposition follows.
\end{proof}

The proposition has the following immediate corollary.

\begin{corollary}\label{co:minimal_elliptic}
Let $\Sigma^+$ be a finite set of nonarchimedean places of $F^+$ containing $\Sigma^+_\bad$ such that $A$ has good reduction outside $\Sigma^+$. Then for all but finitely many rational primes $\ell$, $\bar{r}_{A,\ell}$ is rigid for $(\Sigma^+,\emptyset)$ (Definition \ref{de:rigid} with $\sO=\dZ_\ell$).
\end{corollary}

\begin{proof}
We need show that each of the four conditions in Definition \ref{de:rigid} excludes only finitely many rational primes $\ell$. By Proposition \ref{pr:minimal_elliptic}, condition (1) excludes only finitely many $\ell$. Condition (2) is empty. Condition (3) holds if $\ell$ satisfies $\ell\geq N+1$ and $\Sigma^+_\ell\cap\Sigma^+=\emptyset$. Condition (4) is automatic.
\end{proof}

\subsection{Rigidity of automorphic Galois representations}
\label{ss:rigidity}

In this subsection, we study rigidity of reduction of automorphic Galois representations.

Let $\Pi$ be an RACSDC representation of $\GL_N(\dA_F)$ (Definition \ref{de:relevant}) for $N\geq 2$, and denote by $\Sigma^+_\Pi$ the smallest (finite) set of nonarchimedean places of $F^+$ containing $\Sigma^+_\bad$ such that $\Pi_w$ is unramified for every nonarchimedean place $w$ of $F$ not above $\Sigma^+_\Pi$. Let $E\subseteq\dC$ a strong coefficient field of $\Pi$ (Definition \ref{de:weak_field}). Then for every prime $\lambda$ of $E$, we have a continuous homomorphism $\rho_{\Pi,\lambda}\colon\Gamma_F\to\GL_N(E_\lambda)$.

When $\rho_{\Pi,\lambda}$ is residually absolutely irreducible, by Proposition \ref{pr:galois}(2c) and Lemma \ref{le:representation_selfdual}(3), we have an extension
\begin{align*}
\bar\rho_{\Pi,\lambda,+}\colon\Gamma_{F^+}\to\sG_N(O_E/\lambda)
\end{align*}
of $\bar\rho_{\Pi,\lambda}$, with the similitude character $\chi_\lambda\coloneqq\eta^N\epsilon_\ell^{1-N}$.

\begin{conjecture}\label{co:rigid}
Let $\Pi$ and $E$ be as above. Fix a finite set $\Sigma^+$ of nonarchimedean place of $F^+$ containing $\Sigma^+_\Pi$. Then for all but finitely many primes $\lambda$ of $E$, we have
\begin{enumerate}
  \item $\rho_{\Pi,\lambda}$ is residually absolutely irreducible;

  \item $\bar\rho_{\Pi,\lambda}\res_{\Gal(\ol{F}/F(\zeta_\ell))}$ is absolutely irreducible, where $\ell$ is the underlying rational prime of $\lambda$;

  \item $\bar{r}_{\Pi,\lambda}\coloneqq\bar\rho_{\Pi,\lambda,+}$ is rigid for $(\Sigma^+,\emptyset)$ (Definition \ref{de:rigid} with $\sO$ the ring of integers of $E_\lambda$).
\end{enumerate}
\end{conjecture}

\begin{remark}
When $N=2$, Conjecture \ref{co:rigid} is not hard to verify. Part (3) of Conjecture \ref{co:rigid} was also studied in \cite{Gui} under several simplifying restrictions.
\end{remark}

Concerning Conjecture \ref{co:rigid}(1,2), we have the following proposition.

\begin{proposition}\label{pr:rigid}
Let $\Pi$ and $E$ be as above. Suppose that there exists a nonarchimedean place $w$ of $F$ such that $\Pi_w$ is supercuspidal. Then
\begin{enumerate}
  \item There exists a finite set $\Lambda_1$ of primes of $E$ depending on $\Pi_w$ only, such that for every $\lambda\not\in\Lambda_1$, $\rho_{\Pi,\lambda}$ is residually absolutely irreducible.

  \item There exists a finite set $\Lambda_2$ containing $\Lambda_1$ from (1) such that for every $\lambda\not\in\Lambda_2$, the restriction $\bar\rho_{\Pi,\lambda}\res_{\Gal(\ol{F}/F(\zeta_\ell))}$ remains absolutely irreducible, where $\bar\rho_{\Pi,\lambda}$ denotes the residual representation of $\rho_{\Pi,\lambda}$ and $\ell$ is the underlying rational prime of $\lambda$.
\end{enumerate}
\end{proposition}

The proof of part (2) was suggested to us by Toby Gee. Originally, our alternative argument needs to further assume for (2) that $\Pi$ is a twist of the Steinberg representation at some nonarchimedean place of $F$ not above $\Sigma^+_\bad$.

\begin{proof}
Let $\rW_{F_w}$ be the Weil group of $F_w$. Since $\Pi_w$ is supercuspidal, we have the induced continuous representation $\rho_{\Pi_w}\colon\rW_{F_w}\to\GL_N(\dC)$ via the local Langlands correspondence, which is irreducible. Fix an arithmetic Frobenius element $\phi_w$ in $\rW_{F_w}$, which determines a natural quotient map $\rW_{F_w}\to\dZ$ sending $\phi_w$ to $1$. For every integer $b\geq 1$, let $\rW^b_{F_w}$ be the inverse image of $b\dZ$. Then there exist an absolutely irreducible representation $\tau$ of $\rI_{F_w}$ and a character $\chi$ of $\rW^b_{F_w}$, such that $\rho_{\Pi_w}$ is isomorphic to $\Ind_{\rW^b_{F_w}}^{\rW_{F_w}}\tau\otimes\chi$, where $b$ is the smallest positive integer satisfying $\tau^{\phi_w^b}\simeq\tau$. We may choose a finite extension $E'$ of $E$ inside $\dC$, and a finite set $\Lambda'$ of primes of $E'$, such that both $\tau$ and $\chi$ are defined over $O_{E',(\Lambda')}$. In particular, the image of $\rho_{\Pi_w}$ is contained in $\GL_N(O_{E',(\Lambda')})$, up to conjugation.

Now let $\Lambda'_1$ be the smallest set of primes of $E'$ containing $\Lambda'$ such that every $\lambda'\not\in\Lambda'_1$ satisfies
\begin{itemize}
  \item $w$ does not divide $\ell$;

  \item the underlying rational prime does not divide $b|\rI_{F_w}/\Ker\rho_{\Pi_w}|$;

  \item $\bar\tau_{\lambda'}\coloneqq\tau\otimes_{O_{E',(\Lambda')}}\ol{O_{E'}/\lambda'}$ remains irreducible;

  \item $b$ remains the smallest positive integer that satisfies $\bar\tau_{\lambda'}^{\phi_w^b}\simeq\bar\tau_{\lambda'}$.
\end{itemize}
Then $\Lambda'_1$ is a finite set, satisfying that the composite map
\[
\bar\rho_{\Pi_w,\lambda'}\colon\rW_{F_w}\to\GL_N(O_{E',(\Lambda')})\to\GL_N(\ol{O_{E'}/\lambda'})
\]
is irreducible for $\lambda'\not\in\Lambda'_1$. Let $\xi=\xi_\Pi$ be the archimedean weights of $\Pi$ (Definition \ref{de:relevant}).

For (1), we let $\Lambda_1$ be the set of primes of $E$ underlying $\Lambda'_1$; and then (1) follows by Proposition \ref{pr:galois}(2a).

For (2), let $\Lambda_2$ be the union of $\Lambda_1$ constructed in (1) above and all primes $\lambda$ of $F$ whose underlying rational prime $\ell$ satisfies either $\ell\leq N(b_\xi-a_\xi)+1$ or $\Sigma^+_\ell\cap\Sigma^+_\Pi\neq\emptyset$. Take a prime $\lambda\not\in\Lambda_2$. By (1), $\bar\rho_{\Pi,\lambda}$ is absolutely irreducible, whose coefficients we may assume to be just $\ol{O_E/\lambda}$. Since the degree of the extension $F(\zeta_\ell)/F$ is coprime to $\ell$, the representation $\bar\rho_{\Pi,\lambda}\res_{\Gal(\ol{F}/F(\zeta_\ell))}$ is semisimple. We claim that $\bar\rho_{\Pi,\lambda}$ is an induction of an irreducible representation $\rho'$ of $\Gal(\ol{F}/F')$ for some field extensions $F\subseteq F'\subseteq F(\zeta_\ell)$ such that $[F':F]$ equals the number of irreducible summands of $\bar\rho_{\Pi,\lambda}\res_{\Gal(\ol{F}/F(\zeta_\ell))}$. By \cite{CG13}*{Lemma~4.3}, it suffices to show that the irreducible summands of $\bar\rho_{\Pi,\lambda}\res_{\Gal(\ol{F}/F(\zeta_\ell))}$ are pairwise non-isomorphic. Since $w$ is unramified in $F(\zeta_\ell)$, it suffices to check that the irreducible summands of $\bar\rho_{\Pi,\lambda}\res_{\rI_{F_w}}$ are pairwise non-isomorphic, which is already known by our choice of $\Lambda'_1$ above.

By our definition of $\Lambda_2$ and Proposition \ref{pr:galois}(2b), $\bar\rho_{\Pi,\lambda}$ is crystalline with regular Fontaine--Laffaille weights in $[a_\xi,b_\xi]$ and $\ell>N(b_\xi-a_\xi)+1\geq (b_\xi-a_\xi)+2$. Thus, we must have $F'=F$ by Lemma \ref{le:rigid_1} below. Therefore, $\bar\rho_{\Pi,\lambda}\res_{\Gal(\ol{F}/F(\zeta_\ell))}$ remains absolutely irreducible, hence (2) follows.

The proposition is proved.
\end{proof}

To finish the proof of Proposition \ref{pr:rigid}(2), we need two lemmas, both of which are suggested to us by Toby Gee. We start with some notation. For every finite extension $L$ of $\ol\dQ_\ell$ contained in $\ol\dQ_\ell$, we put $\Gamma_L\coloneqq\Gal(\ol\dQ_\ell/L)$, denote by $\rI_L\subseteq\Gamma_L$ the inertia subgroup and by $\rP_L\subseteq\rI_L$ the wild inertia subgroup, and put $\rT_L\coloneqq\Gamma_L/\rP_L$.

\begin{lem}\label{le:rigid_2}
Let $\bar\rho\colon\Gamma_L\to\GL_N(\ol\dF_\ell)$ be a continuous homomorphism such that $\bar\rho(\rP_L)=\{1\}$. Then there exists a finite unramified extension $L'$ of $L$ inside $\ol\dQ_\ell$ such that $\bar\rho\res_{\Gamma_{L'}}$ is a direct sum of characters.
\end{lem}

\begin{proof}
Let $t\in\rI_L/\rP_L$ be a topological generator and $\phi\in\rT_L$ a lift of the arithmetic Frobenius. Then we have $\rT_L=t^{\prod_{p\neq\ell}\dZ_p}\rtimes\phi^{\widehat\dZ}$ subject to the relation $\phi t\phi^{-1}=t^{\ell^a}$, where $\ell^a$ is the residue cardinality of $L$. We regard $\bar\rho$ as a representation of $\rT_L$. As $\rI_L/\rP_L$ has pro-order prime to $\ell$, the element $\bar\rho(t)$ is semisimple. Let $b\geq 1$ be an integer such that the eigenvalues of $\bar\rho(t)$ are contained in $\dF_{\ell^{ab}}^\times\subseteq\ol\dF_\ell^\times$. Then $\bar\rho(\phi^b)$ commutes with $\bar\rho(t)$. Let $c\geq 1$ be an integer such that $\bar\rho(\phi^{bc})$ is semisimple. Then the unique unramified extension of $L$ inside $\ol\dQ_\ell$ of degree $bc$ satisfies the requirement of the lemma.
\end{proof}

\begin{lem}\label{le:rigid_1}
Consider a field tower $\dQ_\ell\subseteq L\subseteq L'\subseteq\ol\dQ_\ell$ in which $L/\dQ_\ell$ is finite unramified and $L'/L$ is finite Galois. Let $\bar\psi\colon\Gamma_{L'}\to\GL_m(\ol\dF_\ell)$ be a continuous homomorphism for some integer $m\geq 1$, and put $\bar\rho\coloneqq\Ind_{\Gamma_{L'}}^{\Gamma_L}\bar\psi$. Suppose that
\begin{enumerate}[label=(\alph*)]
  \item $\bar\rho$ is crystalline with regular Fontaine--Laffaille weights in $[a,b]$ (Definition \ref{de:crystalline_fl}) for some integers $a$ and $b$ satisfying $0\leq b-a\leq \ell-2$; and

  \item $(b-a)[L':L]< \ell-1$.
\end{enumerate}
Then $L'$ is unramified over $L$.
\end{lem}

\begin{proof}
We may assume $a=0$. We assume on the contrary that $L'/L$ is ramified. Up to replacing $L$ by the maximal subfield of $L'$ that is unramified over $L$, we may assume that $L'/L$ is totally ramified. Up to replacing $\bar\psi$ by its semisimplification, we may assume that $\bar\psi$ is semisimple. Since $[L':L]$ is coprime to $\ell$, $\bar\rho=\Ind_{\Gamma_{L'}}^{\Gamma_L}(\bar\psi)$ is also semisimple. As $\rP_L$ is a pro-$\ell$ normal subgroup of $\Gamma_L$, we have $\bar\rho(\rP_L)=\{1\}$ by \cite{Ser72}*{Proposition~4}. By Lemma \ref{le:rigid_2}, there exists a finite unramified extension $L''/L$ inside $\ol\dQ_\ell$ such that $\bar\rho\res_{\Gamma_{L''}}$ is a direct sum of characters. Up to replacing $L$ by $L''$ and $L'$ by $L'L''$, we may assume that $\bar\rho$ itself is a direct sum of characters of $\Gamma_L$, say $\chi_1,\dots,\chi_N\colon\Gamma_L\to\ol\dF_\ell^\times$. On the other hand, since $\bar\rho=\Ind_{\Gamma_{L'}}^{\Gamma_L}(\bar\psi)$, there exist two distinct characters $\chi_i$ and $\chi_j$ such that $\chi_i\chi_j^{-1}$ is trivial on $\rI_{L'}$, which is the unique subgroup of  $\rI_L$ of index $[L':L]$. However, condition (a) implies that $\chi_i\chi_j^{-1}$ is a crystalline character of Fontaine--Laffaille weights contained in $[-b,b]$. By the Fontaine--Laffaille theory, we have
\[
\chi_i\chi_j^{-1}\res_{\rI_L}=\bigotimes_{\tau\colon L\hookrightarrow\ol\dQ_\ell}\omega_\tau^{b_\tau},
\]
where $\omega_\tau\colon\rI_L\to\ol\dF_\ell^\times$ is the fundamental character of level $1$ corresponding to $\tau$, and $b_\tau$ is an integer in $[-b,b]$. Since the Fontaine--Laffaille weights of $\bar\rho$ are regular, we have $b_\tau\neq0$ for all $\tau$. Now condition (b) implies that $\chi_i\chi_j^{-1}$ can not be trivial on $\rI_{L'}$, which is a contradiction. The lemma is proved.
\end{proof}

Concerning the entire Conjecture \ref{co:rigid}, we have the following theorem.

\begin{theorem}\label{th:rigid}
Let $\Pi$ and $E$ be as above. Suppose that there exists a nonarchimedean place of $F$ at which $\Pi$ is supercuspidal. Then Conjecture \ref{co:rigid} holds for $\Pi$ and $E$.
\end{theorem}

\begin{proof}
Let $\xi=\xi_\Pi$ be the archimedean weights of $\Pi$ (Definition \ref{de:relevant}). Let $\Lambda_2$ be the set in Proposition \ref{pr:rigid}(2). It suffices to prove (3) of Conjecture \ref{co:rigid}. We need show that each of the four conditions in Definition \ref{de:rigid} excludes only finitely many primes $\lambda$. Condition (2) is empty. Condition (3) holds if the underlying rational prime $\ell$ of $\lambda$ satisfies $\ell\geq(b_\xi-a_\xi)+2$ and $\Sigma^+_\ell\cap\Sigma^+=\emptyset$ by Proposition \ref{pr:galois}(2b). Condition (4) is automatic.

It remains to consider condition (1). Let $\lambda$ be a prime of $E$ not in $\Lambda_2$ whose underlying rational prime $\ell$ satisfies
\[
\Sigma^+_\ell\cap\Sigma^+=\emptyset,\qquad\ell\geq 2(N+1), \qquad \ell\geq (b_\xi-a_\xi)+2,\qquad \ell\geq q_w^N
\]
for every place $w$ of $F$ above $\Sigma^+$. In particular, we have
\begin{enumerate}[label=(\alph*)]
  \item $\ell$ is unramified in $F$;

  \item $\Pi_w$ is unramified for every place $w$ of $F$ above $\ell$;

  \item $\bar\rho_{\Pi,\lambda}\res_{\Gal(\ol{F}/F(\zeta_\ell))}$ is absolutely irreducible, which implies that $\bar{r}_{\Pi,\lambda}(\Gal(\ol{F}/F^+(\zeta_\ell)))$ is adequate by Remark \ref{re:adequate};

  \item Proposition \ref{pr:deformation_fl} holds for the local deformation problem $\sD^\FL$ of $\bar{r}_{\Pi,\lambda,v}$ for every $v\in\Sigma^+_\ell$;

  \item Proposition \ref{pr:deformation_min} holds for $\bar{r}_{\Pi,\lambda,v}$ for every $v\in\Sigma^+$.
\end{enumerate}

For a collection $\sD_{\Sigma^+}=\{\sD_v\res v\in\Sigma^+\}$ in which $\sD_v$ is an irreducible component of $\Spf\sfR^\loc_{\bar{r}_{\Pi,\lambda,v}}$ for $v\in\Sigma^+$, we define a global deformation problem (Definition \ref{de:global_deformation_problem})
\[
\sS(\sD_{\Sigma^+})\coloneqq(\bar{r}_{\Pi,\lambda},\eta^N\epsilon_\ell^{1-N},\Sigma^+\cup\Sigma^+_\ell,
\{\sD_v\}_{v\in\Sigma^+\cup\Sigma^+_\ell})
\]
where for $v\in\Sigma^+$, $\sD_v$ is the prescribed irreducible component (which is a local deformation problem by Proposition \ref{pr:deformation_min}(2)) in $\sD_{\Sigma^+}$; and for $v\in\Sigma^+_\ell$, $\sD_v$ is the local deformation problem $\sD^\FL$ of $\bar{r}_{\Pi,\lambda,v}$ from Definition \ref{de:deformation_fl}. Now by (a--e), and the same proof of \cite{Tho12}*{Theorem~10.1} (which assumes that $\Sigma^+\cup\Sigma^+_\ell$ consists only of places split in $F$), we know that the global universal deformation ring $\sfR^\univ_{\sS(\sD_{\Sigma^+})}$ is a finite $\sO$-module. Moreover, we have $\mu\equiv N\modulo 2$. By (d,e) and the same proof of \cite{Gee11}*{Lemma~5.1.3} (which assumes that $\Sigma^+\cup\Sigma^+_\ell$ consists only of places split in $F$), we know that the Krull dimension of $\sfR^\univ_{\sS(\sD_{\Sigma^+})}$ is at least one. Thus, $\sfR^\univ_{\sS(\sD_{\Sigma^+})}[1/\ell]$ is nonzero. Fix an isomorphism $\iota_\ell\colon\dC\xrightarrow{\sim}\ol\dQ_\ell$. By choosing a $\ol\dQ_\ell$-point of $\Spec\sfR^\univ_{\sS(\sD_{\Sigma^+})}[1/\ell]$, we obtain an RACSDC representation $\Pi(\sD_{\Sigma^+})$ of $\GL_N(\dA_F)$ satisfying
\begin{itemize}
  \item $\Pi(\sD_{\Sigma^+})$ is unramified outside $\Sigma^+$;

  \item for every place $w$ of $F$ above $\Sigma^+$, there is an open compact subgroup $U_w$ of $\GL_N(F_w)$ depending only on $\Pi_w$, such that $\Pi(\sD_{\Sigma^+})_w^{U_w}\neq\{0\}$;

  \item the archimedean weights of $\Pi(\sD_{\Sigma^+})$ are contained in $[a_\xi,b_\xi-N+1]$;

  \item $\rho_{\Pi(\sD_{\Sigma^+}),\iota_\ell}$ and $\rho_{\Pi,\lambda}\otimes_{E_\lambda}\ol\dQ_\ell$ are residually isomorphic.
\end{itemize}
In fact, the second property is a consequence of Proposition \ref{pr:galois}(2a), Corollary \ref{co:bernstein} (which is applicable since $\ell\geq q_w^N$), and the fact that irreducible admissible representations lying on a given Bernstein component have a common level. Note that there are only finitely many RACSDC representations of $\GL_N(\dA_F)$ up to isomorphism, satisfying the first three properties. By the strong multiplicity one property for $\GL_N$ \cite{PS79}, we know that for $\ell$ large enough, $\Pi$ is the only RACSDC representation of $\GL_N(\dA_F)$ up to isomorphism, satisfying all the four properties.

Now we claim that for two different collections $\sD_{\Sigma^+}$ and $\sD'_{\Sigma^+}$, the RACSDC representations $\Pi(\sD_{\Sigma^+})$ and $\Pi(\sD'_{\Sigma^+})$ are not isomorphic. Assuming this claim, then for $\ell$ large enough, we have only one collection, which is $\{\sD_v^\mnm\res v\in\Sigma^+\}$, that is, Definition \ref{de:rigid}(1) is satisfied. The theorem is proved.

For the claim itself, we take a place $v\in\Sigma^+$. Then the local components of $\Pi(\sD_{\Sigma^+})$ above $v$ give rise to a continuous homomorphism $r\colon\Gamma_{F^+_v}\to\sG_N(\ol\dQ_\ell)$, which corresponds to a $\ol\dQ_\ell$-point $x_r$ in $\Spec\sfR^\loc_{\bar{r}_{\Pi,\lambda,v}}[1/\ell]$. Now the dimension of the tangent space of $\Spec\sfR^\loc_{\bar{r}_{\Pi,\lambda,v}}[1/\ell]$ at $x_r$ is equal to
\begin{align*}
&\quad N^2+\dim_{\ol\dQ_\ell}\rH^1(F^+_v,\ad r)-\dim_{\ol\dQ_\ell}\rH^0(F^+_v,\ad r)=N^2+\dim_{\ol\dQ_\ell}\rH^2(F^+_v,\ad r) \\
&=N^2+\dim_{\ol\dQ_\ell}\rH^0(F^+_v,(\ad r)(1))
\leq N^2+\dim_{\ol\dQ_\ell}\rH^0(F_w,(\ad r^\natural)(1)),
\end{align*}
where $w$ is the place of $F$ induced by the embedding $F\hookrightarrow\ol{F}^+_v$. However, since $\Pi(\sD_{\Sigma^+})_w$ is generic, we have $\dim_{\ol\dQ_\ell}\rH^0(F_w,(\ad r^\natural)(1))=0$ by \cite{BLGGT}*{Lemma~1.3.2(1)}. Thus, by Proposition \ref{pr:deformation_min}(1), $\Spec\sfR^\loc_{\bar{r}_{\Pi,\lambda,v}}[1/\ell]$ is smooth at $x_r$, which implies that $x_r$ can not lie on two irreducible components. The claim then follows.
\end{proof}

\end{document}